\documentclass[reqno]{amsart}
\pdfoutput=1
\usepackage{fullpage}
\usepackage{scalerel}
\usepackage{float}
\usepackage{empheq}
\usepackage{wrapfig}
\usepackage{graphicx}
\usepackage{mathtools}
\mathtoolsset{showonlyrefs}
\usepackage{hyperref}
\usepackage{amssymb}
\DeclareMathOperator{\Gr}{Gr}
\DeclareMathOperator{\GL}{GL}
\DeclareMathOperator{\symb}{symb}
\DeclareMathOperator{\SO}{SO}
\DeclareMathOperator*{\Span}{span}
\DeclareMathOperator{\grad}{grad}

\numberwithin{equation}{section}
\usepackage{microtype}
\usepackage[bibstyle=custom,citestyle=alphabetic,block=space]{biblatex}
\bibliography{cd.bib}
\usepackage{verbatim}
\usepackage{breqn}
\usepackage{tikz}
\newtheorem{specialtheorem}{Theorem}[section]
\newtheorem{specialcorollary}[specialtheorem]{Corollary}
\newtheorem{theorem}[equation]{Theorem}
\newtheorem{lemma}[equation]{Lemma}
\newtheorem{proposition}[equation]{Proposition}
\newtheorem{corollary}[equation]{Corollary}
\theoremstyle{definition}
\newtheorem{definition}[equation]{Definition}
\newtheorem{remark}[equation]{Remark}
\newcommand{\bb}[1]{\mathbb{#1}}
\newcommand{\restr}[2]{\left. \kern-\nulldelimiterspace #1 \right|_{#2}}
\DeclareMathOperator{\Hol}{Hol}

\DeclareMathOperator{\Ad}{Ad}
\DeclareMathOperator{\Isom}{Isom}

\DeclareMathOperator{\proj}{proj}

\DeclareMathOperator{\id}{id}
\begin{document}

\author{Salman Siddiqi}
\title{Decay of correlations for certain isometric extensions of Anosov flows}
\begin{abstract}
	We establish exponential decay of correlations of all orders for locally $G$-accessible isometric extensions of transitive Anosov flows, under the assumption that the strong stable and strong unstable foliations of the base Anosov flow are $C^1$. This is accomplished by translating accessibility properties of the extension into local non-integrability estimates measured by Dolgopyat's infinitesimal transitivity group, from which we obtain contraction properties for a class of `twisted' symbolic transfer operators.
\end{abstract}

\maketitle

\tableofcontents

\section*{Introduction}

One of the strongest characteristics of chaotic behaviour in dynamical systems is the exponential decay of correlations, or exponential mixing; besides being of intrinsic interest, this is typically accompanied by other strong statistical properties for regular observables. Naturally, there has been substantial interest in understanding and characterizing the dynamical systems with this property.

Hyperbolicity and the joint structure of the strong stable and strong unstable foliations are among the most well-understood mechanisms driving chaotic behaviour in both discrete-time and continuous-time dynamical systems. For continuous-time systems in particular, the extent to which these foliations fail to be integrable is known to be especially important.

Unfortunately, even among Anosov flows, there is no complete characterization of those which are exponentially mixing. There have been many significant advances, however, with perhaps the most notable due to Dolgopyat, who showed in \cite{dolgopyatmixing} that the uniform local non-integrability of the strong stable and unstable foliations leads to exponential mixing for smooth Anosov flows, for a large class of equilibrium measures. We will not attempt to give an account of the considerable progress that has been achieved since then, but we refer the reader to \cite[\S 1]{warbutterley} for an excellent narrative.

Our attention will be restricted to the class of Anosov flows studied in \cite{dolgopyatmixing}, which are known to be exponentially mixing. We will consider compact isometric extensions of these flows, and give criteria for these extensions to be exponentially mixing. We prove the following:

\renewcommand{\thespecialtheorem}{\Alph{specialtheorem}}
\begin{specialtheorem}
	Let $M, N$ and $F$ be closed Riemannian manifolds, where $\pi \colon M \to N$ is a fiber bundle with fibers isometric to $F$. Suppose that $g_t \colon N \to N$ is a smooth, transitive Anosov flow preserving an equilibrium measure $\nu$ with a H\"older potential $\varsigma \colon N \to \bb R$. Moreover, suppose that the strong stable and strong unstable foliations are $C^1$, and that $\nu$ has unstable conditional measures $\nu^u$ that are diametrically regular.

	Let $G$ be a closed, connected, normal subgroup of the isometry group $\Isom(F)$ that acts transitively on $F$, and equip $F$ with the pushforward $\omega$ of the normalized Haar measure on $G$. Let $f_t \colon M \to M$ be a $G$-extension of $g_t$. If $f_t$ is locally $G$-accessible, then it enjoys exponential decay of correlations of all orders for the (locally defined) product measure $\nu \times \omega$.
	\label{maintheorem}
\end{specialtheorem}

This is analogous to a result of Dolgopyat, who showed in \cite{dolgopyatextension} that accessible compact group extensions of discrete-time expanding dynamical systems are exponentially mixing, and whose techniques we make heavy use of in our proof. While we will not make use of this in the proof, it is also worth remarking that there is a considerable general theory of partially hyperbolic systems. We note in particular the celebrated result of Burns and Wilkinson, who showed in \cite{burnswilkinson} that essentially accessible center-bunched volume-preserving systems are mixing of all orders.

Our approach broadly follows the strategy Winter used in \cite{winter} to establish exponential mixing for frame flows over convex cocompact hyperbolic manifolds. We begin by constructing a symbolic model for the extension flow, establishing uniform local non-integrability estimates for this symbolic model and using arguments employed by Dolgopyat in both \cite{dolgopyatmixing} and \cite{dolgopyatextension} to obtain bounds for the spectrum of certain `twisted' transfer operators. Our main technical result is the following bound:

\begin{specialtheorem}
	Fix notation as in Theorem \ref{maintheorem}, and suppose that the potential $\varsigma$ is $C^1$. Then there are constants $C > 0$ and $r < 1$ so that we have
	\[
		\|\mathcal L^n_{z,\rho} \varphi \|_{L^2(\nu^u)} \leq C \|\varphi\|_{C^1} r^n
	\]
	for all $\varphi \in C^1(U, V^\rho)$, all nontrivial irreducible representations $\rho$ of $G$, and any $z \in \bb C$ with $|\Re(z) - P(\varsigma)| < 1$.
\end{specialtheorem}

The principal novelty here, and our main point of divergence with \cite{winter}, is that we use the local $G$-accessibility (see Definition \ref{accessible}) of the extension flow $f_t$ to obtain the necessary local non-integrability estimates. As in \cite{winter}, we measure the local non-integrability of the extension using Dolgopyat's infinitesimal transitivity group.

While we rely heavily on the techniques used in \cite{dolgopyatextension}, translating these into our setting presents several difficulties. Most notably, the non-integrability of the strong stable and unstable foliations of the base flow $g_t$ and the nontriviality of the fiber bundle $\pi \colon M \to N$ require some additional care to properly deal with, though we are ultimately able to adapt many of the same arguments.

Finally, while we require a large number of hypotheses in the statement of Theorem \ref{maintheorem}, there are several cases in which it can be rewritten more succinctly. The most natural examples in the class of systems under study here are the full frame flows on closed manifolds of quarter-pinched negative curvature, where we can give criteria for exponential mixing. The doubling property in this case is a consequence of, for instance, \cite[Proposition 3.12]{pps}.

\begin{specialcorollary}
	Let $N$ be the unit tangent bundle of an $n$-manifold of quarter-pinched negative curvature equipped with an equilibrium measure $\nu$ for a H\"older potential, and let $M$ be the oriented full frame bundle over $N$ equipped with the (locally defined) natural extension $\nu \times \omega$ by the Haar measure on $\SO(n-1)$. If the frame flow $f_t$ is locally accessible, then it enjoys exponential decay of correlations of all orders.
\end{specialcorollary}

\addtocontents{toc}{\protect\setcounter{tocdepth}{1}}
\subsection*{Acknowledgements} I am indebted to my advisor, Ralf Spatzier, for his tireless and persistent support, and to Amie Wilkinson for several brief conversations that proved to be exceedingly prescient. I have certainly also benefited from countless other discussions with more people than I can name, but I would like to thank Mitul Islam in particular for many productive exchanges.

\addtocontents{toc}{\protect\setcounter{tocdepth}{2}}
\section{Preliminaries}

We fix some notation that we will use throughout this paper: $N$ will be a closed Riemannian manifold equipped with a probability measure $\nu$ and $g_t \colon N \to N$ a $C^\infty$ Anosov flow preserving $\nu$. Let $M$ be a compact Riemannian fiber bundle $\pi \colon M \to N$ whose fibers $\pi^{-1}(x)$ are each isometric to a fixed compact, Riemannian manifold $F$, and let $f_t \colon M \to M$ be an extension of $g_t$ satisfying $\pi \circ f_t = g_t \circ \pi$. We equip $M$ with the product measure $\mu$ of $\nu$ and a probability measure on $F$.

The motivating example for everything that follows is when $N$ is the unit tangent bundle for a closed $n$-manifold of quarter-pinched negative curvature, $M$ is the oriented orthonormal frame bundle, $g_t$ is the geodesic flow and $f_t$ is the frame flow. The natural choice for $\nu$ in this case is the Liouville measure on $N$, and $\mu$ is then locally a product of the Liouville measure and the normalized Haar measure on $\SO(n-1,\bb R)$.

\subsection{Dynamical preliminaries}

Our goal is to show that, with appropriate hypotheses, $f_t$ enjoys \textit{exponential decay of correlations} or, equivalently, is \textit{exponentially mixing}.

\begin{definition}
	\label{expmixingdefn}
  A flow $f_t$ is said to be \textit{exponentially mixing of order $k$ for $C^\alpha$ functions} if there are constants $C > 0$ and $r < 1$ so that
  \[
	  \left| \int_M  \varphi_0 \cdot   (\varphi_1 \circ f_{t_1}) \cdot \ldots \cdot (\varphi_k \circ f_{t_k})  \, d\nu - \left(\int_M \varphi_i \, d\nu\right) \cdot \ldots \cdot \left(\int_M \varphi_k \, d\nu \right) \right| < Cr^{\smash{\raisebox{-3pt}{$\substack{\scaleto{\min}{4pt} \\ \raisebox{2pt}{$\scaleto{i \neq j}{3pt}$}}$}} |t_i - t_j|}\cdot \|\varphi_0\|_{C^\alpha} \cdot \ldots \cdot \|\varphi_k\|_{C^\alpha}
  \]
  for all $\varphi_i \in C^\alpha(M, \bb C)$. Here, $\| \cdot \|_{C^\alpha}$ denotes the usual $\alpha$-H\"older norm
  \[
    \|\varphi\|_{C^\alpha} \coloneqq \sup_x {\lvert\varphi(x)\rvert} + \sup_{x \neq y} {\frac{|\varphi(x) - \varphi(y)|}{|x - y|^{\mathrlap{\alpha}}}}
  \] 
  on the space $C^\alpha(M, \bb C)$ of $\alpha$-H\"older complex-valued functions.
\end{definition}

Actually, we will prove that there are constants $C > 0$ and $r < $ so that
  \begin{equation}
	  \label{expmixingzero}
	  \left| \int_M \varphi_0 \cdot \left( \varphi_1 \circ f_{t_1} \right) \cdot \ldots \cdot \left( \varphi_k \circ f_{t_k} \right) \, d\nu \right| < C r^{\max t_i} \cdot \|\varphi_0\|_{C^\alpha} \cdot \ldots \cdot \|\varphi_k\|_{C^\alpha}
  \end{equation}
  for all $\varphi_i \in C^\alpha(M, \bb C)$ with $\int_M \varphi_i \, d\nu = 0$. It is an elementary exercise to show that this is equivalent to Definition \ref{expmixingdefn}.

  While having some degree of regularity is critical, exponential mixing for H\"older functions and exponential mixing for more regular functions are equivalent in our case by a standard approximation argument. We provide a brief outline of this argument in the simplest case, to justify our later attention to $C^1$ (rather than H\"older) functions.

\begin{lemma}
  Suppose $f_t$ is exponentially mixing of order $k$ for $C^1$ functions. Then, $f_t$ is exponentially mixing of order $k$ for $C^\alpha$ functions, for any $\alpha > 0$.
  \label{expmixingregularity}
\end{lemma}
\begin{proof}
	We perform the argument in the case $k = 1$. The general case can be obtained by repeating this inductively.

	By \cite[Lemma 2.4]{gorodnikspatzier}, we can find smooth approximations $\varphi_\epsilon$ to $\varphi$ with $\int_M \varphi_\epsilon \, d\mu = 0$ satisfying
  \[
    \|\varphi_\epsilon - \varphi\|_{C^0} \leq \epsilon^\alpha \|\varphi\|_{C^\alpha} \text{ and } \|\varphi_\epsilon\|_{C^1} \leq  \epsilon^{-\dim(M) - 1} \|\varphi\|_{C^0}\\
  \]
  for all $\epsilon > 0$. Since $f_t$ is exponentially mixing (at rate, say, $r^t$) for $C^1$ functions by hypothesis, we can write
   \begin{spreadlines}{0.5em}
  \begin{align*}
    \left| \int_M \varphi \cdot (\psi \circ f_t) \, d\mu \right| &\leq \left| \int_M \varphi \cdot (\psi \circ f_t) \, d\mu - \int_M \varphi \cdot (\psi \circ f_t) \, d\mu \right| - \left| \int_M \varphi_\epsilon \cdot (\psi \circ f_t) \, d\mu \right|\\
    &\leq \left| \int_M (\varphi - \varphi_\epsilon) \cdot (\psi \circ f_t) \, d\mu \right| + Cr^t \|\varphi_\epsilon\|_{C^1} \|\psi\|_{C^1} \\
    &\leq \|\varphi - \varphi_\epsilon\|_{C^0} \|\psi\|_{C^0} + Cr^t \epsilon^{-(\dim(M) + 1)}\|\varphi\|_{C^0} \|\psi\|_{C^1} \\
    &\leq \epsilon^\alpha \|\varphi\|_{C^\alpha} \|\psi\|_{C^0} + C r^t\epsilon^{-(\dim(M)+1)}\|\varphi\|_{C^\alpha} \|\psi\|_{C^1} \\
    &= \left( \epsilon^\alpha + C r^t\epsilon^{-(\dim(M)+1)} \right) \|\varphi\|_{C^\alpha} \|\psi\|_{C^1}
  \end{align*}
  \end{spreadlines}
  for any fixed $t > 0$. Note that the last line does not involve $\varphi_\epsilon$, so we can simply set $\epsilon = r^{kt}$ with $k = \frac 1 {2(\dim(M) + 1)}$. This leaves us with the inequality
  \begin{align*}
    \left| \int_M \varphi \cdot (\psi \circ f_t) \, d\mu \right| &\leq \left(r^{\alpha(\dim(M) + 1)^{-1} t} + Cr^t  r^{-0.5t}\right) \|\varphi\|_{C^\alpha} \|\psi\|_{C^1}\\
    &\leq (1 + C)\left(r^{\min(0.5, \alpha (\dim(M) + 1)^{-1})}\right)^t \|\varphi\|_{C^\alpha} \|\psi\|_{C^1}
  \end{align*}
  and, noting that the base of the exponential term is at most $1$, we see that $f_t$ is exponentially mixing for any $\varphi \in C^\alpha(M)$ and $\psi \in C^1(M)$. Now, set $D \coloneqq C + 1$ and $s \coloneqq r^{\min(0.5, \alpha(\dim(M) + 1)^{-1})}$ and consider $\varphi, \xi \in C^\alpha(M)$. Once again, we can find a smooth approximation $\xi_\epsilon$ to $\xi$ with $\int_M \xi \, d\mu = 0$ satisfying
  \[
    \|\xi_\epsilon - \xi\|_{C^0} \leq \epsilon^\alpha \|\xi\|_{C^\alpha} \text{ and } \|\xi_\epsilon\|_{C^1} \leq  \epsilon^{-\dim(M) - 1} \|\xi\|_{C^0}\\
  \] for all $\epsilon > 0$. By what we have just shown, we can now write
  \begin{spreadlines}{0.5em}
  \begin{align*}
    \left| \int_M \varphi \cdot (\xi \circ f_t) \, d\mu \right| &\leq \left| \int_M \varphi \cdot (\xi \circ f_t) \, d\mu - \int_M \varphi \cdot (\xi_\epsilon \circ f_t) \, d\mu \right| + \left| \int_M \varphi \cdot (\xi_\epsilon \circ f_t) \, d\mu \right|\\
    &\leq \left| \int_M \varphi \cdot ( (\xi - \xi_\epsilon) \circ f_t) \, d\mu \right| + Ds^t \|\varphi\|_{C^\alpha} \|\xi_\epsilon\|_{C^1}\\
    &\leq \|\xi - \xi_\epsilon\|_{C^0} \|\varphi\|_{C^0} + Ds^t \epsilon^{-(\dim(M)+1)} \|\varphi\|_{C^\alpha} \|\xi\|_{C^0}\\
    &\leq \epsilon^\alpha\|\xi\|_{C^\alpha} \|\varphi\|_{C^\alpha} + Ds^t \epsilon^{-(\dim(M)+1)} \|\varphi\|_{C^\alpha} \|\xi\|_{C^\alpha}\\
    &\leq \left(\epsilon^\alpha + Ds^t \epsilon^{-(\dim(M)+1)}\right) \|\varphi\|_{C^\alpha} \|\xi\|_{C^\alpha}
  \end{align*}
  \end{spreadlines}
  for any fixed $t > 0$. Once again, we can set $\epsilon = s^{kt}$ with $k = \frac 1 {2(\dim(M) + 1)}$, leaving us with the bound
  \begin{align*}
    \left| \int_M \varphi \cdot (\xi \circ f_t) \, d\mu \right| &\leq \left(s^{\alpha(\dim(M) + 1)^{-1} t} + Ds^t  s^{-0.5t}\right) \|\varphi\|_{C^\alpha} \|\xi\|_{C^\alpha}\\
    &\leq (1 + D)\left(s^{\min(0.5, \alpha (\dim(M) + 1)^{-1})}\right)^t \|\varphi\|_{C^\alpha} \|\xi\|_{C^\alpha}
  \end{align*}
  which is independent of $\epsilon$. Once again, we note that the base of the exponential term is at most $1$ and does not depend on $\varphi$ or $\xi$, completing our proof.
\end{proof}

Of course, whether a system is exponentially mixing depends on the measure under consideration. We will be interested in equilibrium measures for H\"older potentials.

\begin{definition}
  For a continuous function $\varsigma \colon N \to \bb R$, we call a measure $\nu$ \textit{an equilibrium state for $g_t$ with potential $\varsigma$} if $\nu$ maximizes
  \[
  \int_N \varsigma \, d\nu + h_\nu(g_1)
  \]
  among all $g_t$-invariant probability measures on $N$. To emphasize the potential, we will write $\nu_\varsigma$ for the equilibrium state corresponding to $\varsigma$ when it exists and is unique.
\end{definition}

Of particular importance to us is the fact that equilibrium states admit a local product structure with respect to the strong stable and unstable foliations, and that they are invariant under the appropriate transfer operators; we will expand on both of these properties in due course.

When $g_t$ is an Anosov flow on a compact manifold, it is a classical result of Bowen and Ruelle \cite{bowenruelle} that equilibrium states for H\"older potentials exist and are unique. Of course, the measure of maximal entropy is always an equilibrium state for the trivial potential $\varsigma = 0$. In the case of the geodesic flow in negative curvature, the Liouville measure is the equilibrium state for the geometric potential
\[
\varsigma (x) = -\left.\frac d {dt}\right|_{t=0}  \left(\log \left\| \left.dg_t\right|_{W^u(x)}\right\|\right)
\]
on the unit tangent bundle.

We are interested in extensions of Anosov flows that act fiberwise by isometries.

\begin{definition}
	We call a smooth flow $f_t \colon (M,\mu) \to (M,\mu)$ on a closed Riemannian manifold $M$ a \textit{$G$-extension} of $g_t \colon (N,\nu) \to (N,\nu)$ if $\pi \colon M \to N$ is a smooth fiber bundle where
\begin{itemize}
	\item the fibers $\pi^{-1}(x)$, with the induced metric, are all isometric to a closed Riemannian manifold $F$,
	\item $G$ is a closed, connected normal subgroup of the isometry group $\Isom(F)$,
\item $G$ acts transitively on $F$ and has no proper transitive normal subgroups,
\item $\pi \circ f_t = g_t \circ \pi$,
\item there is an atlas of trivializations of $\pi \colon M \to N$ for which all transition functions lie in $G$,
\item with respect to these trivializations, the isometries of $F$ induced by the flow $f_t$ all lie in $G$,
\item $f_t$ preserves a measure $\mu$ satisfying $\pi_*(\mu) = \nu$, and
\item the fiberwise disintegration of $\mu$ along the fibers of $\pi \colon M \to N$ is the pushforward of the normalized Haar measure on $G$ to each fiber.
\end{itemize}
\end{definition}

The primary driver of exponential mixing in our case will be a stronger variant of local accessibility.

\begin{definition}
	Let $f_t \colon M \to M$ be a $G$-extension of $g_t \colon N \to N$. We call $f_t$ \textit{locally $G$-accessible} if, for every $\epsilon > 0$, any trivialization $\phi \colon \pi^{-1}(B_\epsilon(x)) \to B_\epsilon(x) \times F$ defined near $x \in N$ and any isometry $h \in G$, there is a sequence of points $x_0, \ldots, x_k \in N$ for which
	\begin{itemize}
		\item $x_0 = x_k = x$,
		\item $x_0, \ldots, x_k \in B_\epsilon(x)$,
		\item we either have $x_{i+1} \in W^{su}_{g_t}(x_i)$ or $x_{i+1} \in W^{ss}_{g_t}(x_i)$ for each $i$, and
		\item we have $h = h_k^0 \circ \ldots \circ h_0^1$, where $h_i^{i+1} \colon F \to F$ is given (via $\phi$) by the isometry $\pi^{-1}(x_i) \to \pi^{-1}(x_{i+1})$ induced by leaves of the strong stable or strong unstable foliation of $f_t$.
	\end{itemize}
	\label{accessible}
\end{definition}

\subsection{Symbolic dynamics}
In this section, we will build a discrete, symbolic model for $f_t$ -- we follow \cite{winter}, and accomplish this by artificially extending a standard Markov partition for the base flow. The results of \cite{bowensymbolic, ratnersymbolic} on the existence of Markov partitions for hyperbolic dynamical systems are classical and well-understood; as such, we will recall some of the important points but refrain from delving into the details.

\begin{theorem}[Bowen, Ratner]
	\label{bowenratner}
If $g_t$ is Anosov, then $g_t$ has a Markov partition of size $\epsilon$ for any sufficiently small $\epsilon > 0$.
\end{theorem}

Specifically, Bowen and Ratner construct a Markov partition by taking local strong stable and unstable segments and forming a \textit{Markov rectangle}.

\begin{definition}
	Given $x \in N$ and $\epsilon > 0$, consider the local strong and weak stable and unstable manifolds of size $\epsilon$ through $x$ given by\\
	\begin{minipage}[]{0.49\textwidth}
		\begin{align*}
			W^{ss}_\epsilon(x) &\coloneqq \left( W^{ss}(x) \cap B_\epsilon(x) \right)^\circ\\
			W^{ws}_\epsilon(x) &\coloneqq \left( W^{ws}(x) \cap B_\epsilon(x) \right)^\circ
		\end{align*}
	\end{minipage}\hfill
	\begin{minipage}[]{0.49\textwidth}
		\begin{align*}
			W^{su}_\epsilon(x) &\coloneqq \left( W^{su}(x) \cap B_\epsilon(x) \right)^\circ\\
			W^{wu}_\epsilon(x) &\coloneqq \left( W^{wu}(x) \cap B_\epsilon(x) \right)^\circ
		\end{align*}
	\end{minipage}\vspace{6pt}\\
	where in each case we have taken the connected component through $x$.
	For $u \in W^{su}_\epsilon(x)$ and $s \in W^{ss}_\epsilon(x)$, we define the \textit{bracket of $u$ and $s$} to be the point of intersection
	\[
		[u,s] = W^{ss}_\epsilon(u) \cap W^{wu}_\epsilon(s)
	\]
	which is necessarily unique when $\epsilon > 0$ is sufficiently small. We define the \textit{Markov rectangle} $[W^{su}_\epsilon(x), W^{ss}_\epsilon(x)]$ to be the set
  	\[
		[W^{su}_{\epsilon}(x), W^{ss}_{\epsilon}(x)] \coloneqq \{ [u,s] \mid u \in W^{su}_{\epsilon}(x) \text{ and } s \in W^{ss}_{\epsilon}(x) \}
 	 \]
	 assuming $\epsilon > 0$ is sufficiently small.
\end{definition}

For now, fix $\epsilon > 0$ chosen to be small enough that a Markov partition exists; we will likely need to adjust our choice of $\epsilon$ as we proceed. We let $\mathcal R = \left\{ R_1, \ldots, R_k \right\}$ be a Markov partition of size $\epsilon$ for $g_t$, where each rectangle $R_i \coloneqq [U_i,S_i]$ is generated by local strong stable and unstable manifolds $S_i \coloneqq W^{ss}_\epsilon(z_i)$ and $U_i \coloneqq W^{su}_\epsilon(z_i)$ through points $z_i \in N$. Set $R \coloneqq \bigcup R_i$, $U \coloneqq \bigcup U_i$ and $S \coloneqq \bigcup S_i$, and let $\mathcal P \colon R^* \to R^*$ be the Poincar\'e return map defined on the appropriate full-measure residual subset $R^* \subset R$ with return time $\tau \colon R^* \to \bb R$.

\begin{remark}
	Since we assumed that the strong stable and unstable foliations for $g_t$ were $C^1$, each $R_i$ is naturally an open $C^1$ submanifold of $N$. It is important to note that both the Poincar\'e return map $\mathcal P$ and the return time map $\tau$ defined above are the restrictions of locally $C^1$ functions on $R$, but neither $\mathcal P$ nor $\tau$ is even continuous on $R$. Indeed, the points of discontinuity for $\mathcal P$, $\tau$ and their iterates are exactly what is removed in passing from $R$ to $R^*$.  \end{remark}

A more extensive discussion of this can be found in \cite[p.~380-382]{chernov}. We will write $(\Sigma, \mathcal P)$ for the Markov shift corresponding to the partition $\mathcal R$.

\begin{definition}
	The \textit{suspension} of $(\Sigma, \mathcal P)$ with roof function $\tau$ is the flow $g'_t \colon \Sigma \times \bb R/\sim \to \Sigma \times \bb R/\sim$ defined by $g'_t(x, s) = (x, s + t)$, where we have declared $(x, \tau(x)) \sim (\mathcal P(x), 0)$.
\end{definition}

\noindent Theorem \ref{bowenratner} says exactly that the natural inclusion of $R$ into $N$ induces a H\"older-continuous semi-conjugacy between $g_t$ and $g'_t$ -- this is precisely the result we wish to extend to $f_t$. 

Fix a finite cover $\left\{ V_i \right\}$ that trivializes the fiber bundle $\pi \colon M \to N$, with isometries $\phi_i \colon \pi^{-1}(V_i) \to V_i \times F$. At this point, we may need to revisit our choice of $\epsilon$: let us assume that we have taken $\epsilon$ to be smaller than the Lebesgue number of the cover $\left\{ V_i \right\}$. By definition, this means that each $R_j$ lies entirely in some $V_{k(j)}$, and hence $\phi_{k(j)}$ induces an isometry $\pi^{-1}(R_j) \to R_j \times F$. We can put these together to form an isometry $\phi \colon \pi^{-1}(R) \to R \times F$.

It is worth noting that we have considerable freedom in our choice of isometries $\phi_i$, and this is something we will want to exploit to simplify our arguments in \S\ref{snli}. Since the center-stable foliation of $f_t$ is $C^1$, we can modify a given $\phi_i$ so that it is constant along the (local) leaves of this foliation. Specifically, we will assume that the projection of each $\phi_i \colon \pi^{-1}(V_i) \to V_i \times F$ onto $F$ is constant on each connected component $\left(W^{su}_{f_t}(x) \cap V_i\right)^\circ$ for each $x \in V_i$..

We will realize $f_t$ as a suspension flow on $\Sigma \times F$ with roof function $\tau$. To do this, we will need information on the fiberwise action of $g_t$.

\begin{definition}
	For $x \in R_{j_1}$ with $\mathcal P(x) \in R_{j_2}$, we define the \textit{temporal holonomy $\Hol(x)$ at $x$} to be the isometry between $\phi_{k(j_1)}(\pi^{-1}(x))$ and $\phi_{k(j_2)}(\pi^{-1}(\mathcal P(x)))$ induced by $f_t$. This defines a function $\Hol \colon R^* \to G$.\\
	\indent	We will often write $\Hol^{(n)}(x)$ for $\Hol(\mathcal P^{n-1}(x)) \circ \ldots \circ \Hol(x)$. Note that function composition is the multiplication operation in $G$.
	\label{temporalholonomy1}
\end{definition}

\noindent Given the temporal holonomy function $\Hol \colon R^* \to G$, we can of course recover $f_t$ as a suspension flow on $(\Sigma \times F, \mathcal P \times \Hol)$ with roof function $\tau$. However, we will push this a step further.

We can define a projection along the leaves of the strong stable foliation $\proj_S \colon R \to U$ by setting $\proj_S([u,s]) = u$. This allows us to construct a uniformly expanding model $U$ for $g_t$, where the Poincar\'e return map descends to a map $\sigma \colon U^* \to U^*$ given by $\sigma \coloneqq \proj_S \circ \mathcal P$.

\begin{remark}
	By construction, the return time map $\tau$ and the temporal holonomy function $\Hol$ are both constant along the leaves of the strong stable foliation of $g_t$, and hence descend to functions $\tau \colon U \to \bb R$ and $\Hol \colon U \to G$.
\end{remark}

Set $U_{\tau} \coloneqq U^* \times \bb R / \sim$ and $U_{\tau, \Hol} \coloneqq U^* \times \bb R \times F / \sim$, where we declare $(u, \tau(u)) \sim (\mathcal P(u), 0)$ and $(u, \tau(u), k) \sim (\mathcal P(u), 0, (\Hol(u))(k))$ respectively. We write $f'_t$ for the suspension flow $f'_t(u, s, k) = (u, s+t, k)$ on $U_{\tau,\Hol}$. It is not too difficult to show that the exponential mixing of $f_t$ is equivalent to exponential mixing for $f'_t$, though we must first fix a measure on $U_{\tau, \Hol}$ to make sense of this.

\begin{remark}
	Since $\nu$ is an equilibrium state for a H\"older potential $\phi$, there are measures $\nu_i^s$ and $\nu_i^u$ on each of the strong stable and unstable segments $S_i$ and $U_i$ so that $\nu$ is absolutely continuous with respect to the product $\nu_i^u \times \nu_i^s \times dt$. Moreover, these measures can be chosen so that the Radon-Nikodym derivative of $\nu$ with respect to the product $\nu_i^u \times \nu_i^s \times dt$ is uniformly bounded above by a constant $K > 1$ and below by $K^{-1} < 1$. We will write $\nu^u$ and $\nu^s$ for the corresponding measures on $U$ and $S$ respectively, and suppose that they have been normalized so that $\nu_U(U) = \nu_S(S) = 1$. 
\end{remark}

\begin{remark}
	Up to replacing $\phi$ with a cohomologous function on $(\Sigma, \mathcal P)$, we can assume that $\phi$ is the extension of a H\"older potential $\phi^U$ on $U$. As a consequence, we may as well assume that $\nu^u$ is in fact an equilibrium state for a H\"older potential $\phi^U$ on $(U, \sigma)$.
	\label{cohomologous}
\end{remark}

These are classical results in thermodynamic formalism -- we refer the reader to \cite{leplaideur} and \cite[p.~87-91]{margulis} respectively for more details. We will in addition require the conditional measure $\nu^u$ to have a doubling property later on, in order to control the spectrum of our transfer operators using Dolgopyat's methods.

\begin{definition}
	We say that a measure $\nu^u$ has the \textit{doubling} or \textit{Federer} property, or is \textit{diametrically regular}, if for any $k > 1$ there is a uniform constant $C > 0$ so that
	\[
		\nu^u(B_{kr}(x)) < C \nu^u(B_r(x))
	\]
	for all $x \in U$ and $r > 0$.
\end{definition}

With the following lemma, we are reduced to establishing exponential mixing for the suspension flow $f'_t$ on the expanding model $U_{\tau,\Hol}$.

\begin{lemma}
	If $f'_t$ is exponentially mixing of order $k$ for functions in $C^1(U_{\tau, \Hol}, \bb C)$, then $f_t$ is exponentially mixing of order $k$ for functions in $C^1(M, \bb C)$.
\end{lemma}
\begin{proof}
	Once again, we will perform the argument in the case $k = 1$. The general case can be obtained by repeating this inductively.
	Given $\varphi, \psi \in C^1(M, \bb C)$ with $\int_M \varphi \, d\mu = \int_M \psi \, d\mu = 0$ and a fixed $k \in F$, we consider corresponding functions $\varphi_t, \psi_t \in C^1(U_{\tau,\Hol},\bb C)$ given by
	\begin{align*}
		\varphi_t (u, h, r) &\coloneqq \int_{S} \varphi\left( f_{t+r}\left( \phi^{-1}\left( [u,s], h(k) \right) \right) \right) \, d\nu^s\\
		\psi_t (u, h, r) &\coloneqq \int_{S} \psi\left( f_{t+r}\left( \phi^{-1}\left( [u,s], h(k) \right) \right) \right) \, d\nu^s
	\end{align*}
	for each $t$.
	Let $C_0q^t$ be the rate of contraction of $S$ under $g_t$. Since $\varphi$ is $C^1$, 
	\[
		\left| \varphi\left(f_{t+r}(\phi^{-1}([u,s],h(k))\right) - \varphi\left(f_{t+r}(\phi^{-1}([u,s_0],h(k))\right) \right| < C_0 q^t\|\varphi\|_{C^1}
	\]
	for any $s_0 \in S$. Of course, this yields
	\[
		\left| \int_S \varphi\left(f_{t+r}(\phi^{-1}([u,s],h(k))\right) \, d\nu^s - \int_S \varphi\left(f_{t+r}(\phi^{-1}([u,s_0],h(k))\right) \, d\nu^s \right| < C_0 q^t\|\varphi\|_{C^1} \nu^s(S)
	\]
	after simply integrating both sides with respect to $s$. This can be rewritten as
	\[
		\left| \int_S \varphi\left(f_{t+r}(\phi^{-1}([u,s],h(k))\right) \, d\nu^s - \varphi\left(f_{t+r}(\phi^{-1}([u,s_0],h(k))\right) \cdot \nu^s(S) \right| < C_0 q^t\|\varphi\|_{C^1} \nu^s(S)
	\]
	since $s_0$ is fixed. We then see quickly that the difference in the integrals
	\begin{equation}	
	\label{symbolicmixingeqn}
		\int_{U_{\tau, \Hol}} \left(\int_{S} \varphi\left( f_{t+r}\left( \phi^{-1}\left( [u,s], h(k) \right) \right) \right) \, d\nu^s(s) \right) \left(\int_{S} \psi\left( f_{r}\left( \phi^{-1}\left( [u,s'], h(k) \right) \right) \right) \, d\nu^s(s')\right) \, d\omega \, dr \, d\nu^u
	\end{equation}
	and
	\begin{equation}
		\label{symbolicmodeleqn}
		\nu^s(S)\int_{U_{\tau, \Hol}} \int_{S} \varphi\left( f_{t+r}\left( \phi^{-1}\left( [u,s_0],h(k) \right) \right) \right)\cdot \psi\left( f_{r}\left( \phi^{-1}\left( [u,s'], h(k) \right) \right) \right) \, d\nu^s(s')\, d\omega \, dr \, d\nu^u
	\end{equation}
	is at most $C_1q^t \|\varphi\|_{C^1} \|\psi\|_{C^0}$. But now, the same argument shows that
	\[
		\varphi\left(f_{t+r}\left(\phi^{-1}\left( [u,s_0],h(k) \right)\right)\right) \cdot \psi\left( f_r\left( \phi^{-1}\left( [u,s'],h(k) \right) \right) \right)
	\]
	and
	\[
		\varphi\left(f_{t+r}\left(\phi^{-1}\left( [u,s'],h(k) \right)\right)\right) \cdot \psi\left( f_r\left( \phi^{-1}\left( [u,s'],h(k) \right) \right) \right)
	\]
	must be within $C_2q^t \|\varphi\|_{C^1} \|\psi\|_{C^0}$ of each other. Hence,
	\begin{equation}
		\nu^s(S)\int_{U_{\tau, \Hol}} \int_{S} \varphi\left( f_{t+r}\left( \phi^{-1}\left( [u,s'],h(k) \right) \right) \right)\cdot \psi\left( f_{r}\left( \phi^{-1}\left( [u,s'], h(k) \right) \right) \right) \, d\nu^s(s')\, d\omega \, dr \, d\nu^u
	\end{equation}
	is within $C_3 q^t \|\varphi\|_{C^1} \|\psi\|_{C^0}$ of \eqref{symbolicmodeleqn}.
	From the local product structure of $\nu$, we see that this is within a constant multiplicative factor of $K$ of the integral
	\begin{equation}
		\label{actualmixingeqn}
		\int_M \varphi(f_t(x)) \cdot \psi(x) \, d\mu
	\end{equation}
	for any $\varphi$ and $\psi$.
	To conclude, we simply observe that if $\int_M \varphi \, d\mu = \int_M \psi \, d\mu = 0$, then we must have 
	\[
		\int_{U_{\tau, \Hol}} \varphi_t(u,h,r) \, d\omega \, dr \, d\nu^u = \int_{U_{\tau, \Hol}} \psi_t (u,h,r) \, d\omega \, dr \, d\nu^u = 0
	\]
	for any $t \in \bb R$. Moreover, the regularity of $\varphi_t$ and $\psi_t$ is determined by the regularity of the bracket operation $[\cdot,\cdot]$ -- since we assumed that the foliations were $C^1$, we see that $\varphi_t$ and $\psi_t$ must also be $C^1$.
	Hence, if $f'_t$ is exponentially mixing for functions in $C^1(U_{\tau,\Hol},\bb C)$, \eqref{symbolicmixingeqn} must decay exponentially in $t$, from which we conclude that \eqref{actualmixingeqn} must also decay exponentially.
\end{proof}

\subsection{Representation theory}

In this section, we recall some classical results from the representation theory and harmonic analysis of compact Lie groups; our primary references are \cite{fourierseries} and \cite{applebaum}.

Following \cite[\S 3.6]{winter}, our strategy is to decompose functions on $U_{\tau, \Hol}$ into components corresponding to irreducible representations of $G$. A function $\varphi \in C^1(U_{\tau, \Hol}, \bb C)$ can be viewed as a $C^1$ function $\tilde \varphi \colon U_{\tau} \to L^2(G)$ by setting $\tilde \varphi(u, r) \coloneqq \varphi(u, \cdot, r)$. 

Since $G$ is a compact, connected Lie group, we can decompose $\tilde \varphi(u, r) \in L^2(G)$ into isotypic components corresponding to irreducible representations of $G$ -- this is, of course, the classical Peter-Weyl theorem.

\begin{theorem}[Peter-Weyl]
	If $G$ is a compact, connected Lie group, then there is a decomposition
	\[
	L^2(G) = \bigoplus_\rho (V^\rho)^{\oplus \dim \rho}
	\]
where the sum is taken over pairwise non-isomorphic irreducible representations of $G$, and the $(V^\rho){\oplus \dim \rho}$ associated to non-isomorphic irreducible representations are pairwise orthogonal with respect to the standard inner product on $L^2(G)$.
\end{theorem}

We fix such a decomposition and write $\tilde \varphi(u,r) = \sum_\rho \tilde \varphi^\rho(u, r)$ for the decomposition of $\tilde \varphi(u,r)$ obtained by projecting onto each $(V^\rho){\oplus \dim \rho}$. Abusing notation, we will use $\varphi$ to refer interchangeably to a function $U_\tau \to L^2(G)$ or to the function $U_{\tau, \Hol} \to \bb C$ -- there should be little ambiguity in either case.

For our later analysis, it will be helpful to consider the derived representation $d\rho$ of the Lie algebra $\mathfrak g$ of $G$ acting on $L^2(G)$, induced by the representation $\rho$ of $G$ on $L^2(G)$ -- see \cite[\S 2.5.1]{applebaum} for details. We will always assume that we have a fixed $\Ad$-invariant norm $\| \cdot \|_{\mathfrak g}$ on $\mathfrak g$.

\begin{definition}
	Given an irreducible representation $\rho \colon G \to \GL(V^\rho)$ of $G$ (where we view $V^\rho \subset L^2(G)$), we define the norm $\|\rho\|$ of $\rho$ to be the supremum
	\[
		\|\rho\| \coloneqq \sup_{\|X\|_{\mathfrak g} = 1} \|d\rho(X)\|_{L^2(G)}
	\]
	where $\|d\rho(X)\|_{L^2(G)}$ is the operator norm of $d\rho(X)$ viewed as an automorphism of $L^2(G)$.
	\label{derivativenorm}
\end{definition}

It is a classical fact that $\|\rho\|$ is finite, and can be bounded in terms of the highest weight associated to $\rho$.

\begin{proposition}
	Let $\rho$ be a nontrivial irreducible representation, and let $\lambda$ be its highest weight. There are uniform constants $C > 0$ and $m > 0$ so that
	\[
		\|\rho\| \leq C \lambda^m
	\]
	for all $X \in \mathfrak g$ with $\|X\|_{\mathfrak g} = 1$.
	\label{normexists}
\end{proposition}
\begin{proof}
	See \cite[Theorem 3.4.1]{applebaum}; note that the Hilbert-Schmidt norm is an upper bound for the operator norm.
\end{proof}

We will also require some particular results on the growth rate of $\|\rho\|$.

\begin{proposition}
	There is a constant $N > 0$ so that $\sum_\rho \|\rho\|^{-n}$ converges for any $n \geq N$.
	\label{seriesconverges}
\end{proposition}
\begin{proof}
	This is a combination of \cite[Lemma 1.3]{fourierseries} and Proposition \ref{normexists}.
\end{proof}

\noindent And more generally, we can obtain decay estimates for the Fourier coefficients $\varphi^\rho$ associated to irreducible representations $\rho$.

\begin{theorem}
	There are constants $C, N > 0$ so that
	\[
		\|\rho\|^n \|\varphi^\rho\|_{L^2(G)} \leq {C \|\varphi\|_{C^n}}
	\]
	for all irreducible $\rho$, any $n > N$ and all $\varphi \in C^n(U_\tau, L^2(G))$.
	\label{reptheorylemma}
\end{theorem}
\begin{proof}
	This is contained in the proof of \cite[Theorem 1]{fourierseries}.
\end{proof}

\section{Twisted transfer operators}
\label{transferoperators}
In this section, we will define Dolgopyat's `twisted' transfer operators, and show how the spectral bounds we intend to obtain for these operators lead to correlation decay estimates for the expanding suspension semi-flow $f'_t$ constructed in the previous section. 

Recall that we can view a smooth function $\psi \in C^1(U_{\tau, \Hol}, \bb C)$ as a function $\psi \in C^1(U_\tau, L^2(G))$. We can integrate out the time variable to get
\[
	\tilde\psi(u) \coloneqq \int_0^{\tau(u)} \psi(u,r) \, dr
\]
in $C^1(U, L^2(G))$; this is the space on which we would like to define our operators. The advantages of working with smooth (as opposed to H\"older) functions will become clear in \S \ref{transopsec}, but we will need to assume that $\varsigma$ is $C^1$ for most of our arguments. We will explain how to modify our proof to deal with the general case where $\varsigma$ is H\"older in Corollary \ref{holderpotentialapprox}, using a standard approximation argument.

Let $\rho$ be an irreducible representation of $G$ acting on an isotypic component $V^\rho \subset L^2(G)$, and fix $z \in \bb C$. We define the {transfer operator} $\tilde {\mathcal L}_{z,\rho} \colon C^1(U,V^\rho) \to C^1(U,V^\rho)$ by
	\[
		\left( \tilde {\mathcal L}_{z,\rho} \varphi \right)(u) \coloneqq \sum_{\sigma(u') = u} e^{\varsigma_z(u')} \left( \rho(\Hol(u')) \cdot \varphi(u') \right)
	\]
	where $\varsigma_z$ is the potential on the one-sided Markov model $(\Sigma^+, \sigma)$ obtained as the restriction of the potential
	\[
		\int_0^{\tau(u)} (\varsigma \circ p)(u,s,t) \, dt - z \cdot \tau(u,s)
	\]
	defined on $(\Sigma, \mathcal P)$. Note that $\varsigma_z$ is well-defined in light of Remark \ref{cohomologous}, where we assumed that both $\alpha$ and $\tau$ are constant in $s$. It is also worth remarking that both $\tau$ and $p$ are $C^1$, since we assumed that the strong stable and unstable foliations of $g_t$ were $C^1$. As a consequence, $\varsigma_z$ and hence $\tilde {\mathcal L}_{z,\rho}\varphi$ are both $C^1$, since we have restricted ourselves to the case where $\varsigma$ is smooth. 

Let us recall some classical results of thermodynamic formalism; for a slightly more detailed treatment, we refer the reader to \cite[p.~87-91]{margulis}. Let $P(\varsigma)$ be the topological pressure of $\varsigma$ for the map $g_1$. By the Ruelle-Perron-Frobenius theorem, the operator $\tilde {\mathcal L}_{P(\varsigma),0}$ associated to the trivial representation $\rho = 0$ has a unique positive eigenvector $\varphi_\varsigma \in C^1(U, \bb R)$ with eigenvalue $e^{P(\varsigma)}$. 

Recall that $\nu^u$ is an equilibrium measure for the restriction of the potential $\varsigma$ to $(\Sigma^+, \sigma)$, which by construction has the same topological pressure $P(\varsigma)$. By the Lanford-Ruelle variational principle, this means that
\[
	e^{P(\varsigma)} \int_U \varphi \, d\nu^u = \int_U \tilde {\mathcal L}_{P(\varsigma),0}\varphi \, d\nu^u
\]
for all $\varphi \in C^1(U, V^\rho)$. It will be convenient to renormalize $\tilde {\mathcal L}_{P(\varsigma),0}\varphi$ so that it preserves the measure $\nu^u$: let $\mathcal L_{z,\rho}$ be the operator defined by
\[
	\left(\mathcal L_{z,\rho} \varphi\right)(u) \coloneqq \varphi_\varsigma(u)\frac{\left( \tilde {\mathcal L}_{z,\rho} (\varphi \cdot \varphi_\varsigma^{-1}) \right)(u)}{e^{P(\varsigma)}}
\]
for all $\varphi \in C^1(U, V^\rho)$.
\begin{remark}
	With the renormalizations above, we have
	\[
		\int_U \varphi \, d\nu^u = \int_U \mathcal L_{P(\varsigma),0}\varphi \, d\nu^u
	\]
	for all $\varphi \in C^1(U,\bb R)$.
	\label{rpflr}
\end{remark}

We can alternatively write
\[
	\left(\mathcal L_{z,\rho}\varphi\right)(u) = \sum_{\sigma(u')=u} e^{\alpha_z(u')} (\rho(\Hol(u')) \cdot \varphi(u'))
\]
where we set
\[
	\alpha_z(u) \coloneqq \int_0^{\tau(u)} (\varsigma \circ p) (u,s,t) \, dt - z \cdot \tau(u,s) - \log (\varphi_\varsigma(u)) + \log (\varphi_\varsigma(\sigma(u))) - \log P(\varsigma)
\]
for all $u \in U$ -- note that the positivity of $\varphi_\varsigma$ is required to ensure that $\log (\varphi_\varsigma(u))$ is well-defined. It will be helpful to have a similar formulation for the iterates
\[
	\left(\mathcal L_{z,\rho}^n\varphi\right)(u) = \sum_{\sigma^n(u')=u} e^{\alpha^{(n)}_z(u')} (\rho(\Hol^{(n)}(u')) \cdot \varphi(u'))
\]
with
\[
	\alpha^{(n)}_z(u) \coloneqq \int_0^{\tau^{(n)}(u)} (\varsigma \circ p) (u,s,t) \, dt - z \cdot \tau^{(n)}(u,s) - \log (\varphi_\varsigma^{(n)}(u)) + \log (\varphi_\varsigma^{(n)}(\sigma(u))) - n\log P(\varsigma)
\]
for all $u \in U$. Here, we write
\[
	\varphi_\varsigma^{(n)}(u) \coloneqq \prod_{i=0}^{n-1} \varphi_\varsigma\left( \varsigma^i(u) \right)
\]
for the $n^{\rm th}$ ergodic product along $\sigma$.

Our overarching goal is to establish spectral bounds for these operators; in \S\ref{transopsec}, we will ultimately prove:

\begin{theorem}
	\label{spectralgap}
	There are constants $C > 0$ and $r < 1$ so that
	\[
	\| \mathcal L^n_{z,\rho} \varphi \|_{L^2(\nu^u)} \leq C \|\varphi\|_{C^1} r^n
	\]
	for all $\varphi \in C^1(U, V^\rho)$, all nontrivial irreducible representations $\rho$ of $G$, and any $z \in \bb C$ with $|\Re(z) - P(\varsigma)| < 1$.
\end{theorem}

\noindent The remainder of this section will be devoted to showing how we obtain Theorem \ref{maintheorem} from Theorem \ref{spectralgap}. 

	Given $\varphi_0, \ldots, \varphi_k \in C^1(U_{\tau, \Hol}, \bb C)$ with $\int_{U_{\tau, \Hol}} \varphi_i \, d\nu^u \, d\omega \, dr = 0$, let 
	\[
		\beta_k(t_1, \ldots, t_k) \coloneqq \int_U \int_0^{\tau(u)} \int_G \varphi_0(u, h, r) \left( \prod_{i=1}^k \varphi_i(u, h, r + t_i) \right) \, d\omega \, dr \, d\nu^u
	\]
	be the $k^{\rm th}$ correlation function $\beta_k \colon \bb R^k_+ \to \bb R$. In order to show that $\beta_k$ decays exponentially in $t_1, \ldots, t_k$, we will show that the integral defining its Laplace transform $\hat \beta(\xi_1, \ldots, \xi_k)$ converges absolutely for some fixed values of $\xi_1, \ldots, \xi_k$. The following lemma expresses the Laplace transform in terms to the transfer operators we have just defined, the proof of which consists almost entirely of elementary integral manipulations.

We will be somewhat cavalier in interchanging sums and integrals, though this is eventually justified as the final expression we obtain is absolutely convergent.

\begin{lemma}
	Given $\varphi_0, \ldots, \varphi_k \in C^1(U_{\tau, \Hol}, \bb C)$ as above with $\int_{U_{\tau, \Hol}} \varphi_i \, d\nu^u \, d\omega \, dr = 0$, we can bound the Laplace transform of the $k^{th}$-order correlation by
	\[
		\left|\hat\beta_k(\xi_1, \ldots, \xi_k)\right| \; \leq \sum_\rho \sum_{n_1, \ldots, n_k=1}^\infty \int_U \left(\mathcal L^{\max n_j}_{\rho, P(\varsigma)+\xi}\left|\hat\varphi^\rho_{0,-\xi}\right|\right)(u) \left( \prod_{i=1}^k \|\hat\varphi_{i,\xi_i}\|_{C^0} \right) \, d\nu^u
	\]
	for $\Re(\xi_1), \ldots, \Re(\xi_k) < 0$. Here, we use $\hat\varphi_{i, \xi_i}$ to denote the function
	\[
		\hat\varphi_{i,\xi_i}(u, h) \coloneqq \int_0^{\tau(u)} \varphi_i(u, h, t_i) e^{-\xi_it_i} \, dt_i
	\]
	and $\hat \beta_k$ to denote the Laplace transform
	\[
		\hat \beta_k (\xi_1, \ldots, \xi_k) \coloneqq \int_0^\infty \cdots \int_0^\infty \beta_k(t_1, \ldots, t_k) e^{-(\xi_1 t_1 + \ldots + \xi_kt_k)} \, dt_1 \, \ldots \, dt_k
	\]
	for $\xi_i \in \bb R$.
	\label{laplacetransform}
\end{lemma}
\begin{proof}
	By definition, the Laplace transform $\hat\beta_k(\xi_1, \ldots, \xi_k)$ of $\beta_k(t_1, \ldots, t_k)$ is given by
	\begin{equation}
		\int_0^\infty \cdots \int_0^\infty \int_U \int_{\max(0,\tau(u)-t_i)}^{\tau(u)}\int_G \varphi_0(u,h,r) \left( \prod_{i=1}^k \varphi_i(u,h,r+t_i) e^{-\xi_i t_i} \right) \, d\omega \, dr \, d\nu^u \, dt_1 \, \cdots \, dt_k
		\label{laplacetransformeq}
	\end{equation}
	for $\xi_1, \ldots, \xi_k \in \bb C$. Since the systems of inequalities
	\begin{center}
	\begin{minipage}[]{0.3\linewidth}
		\begin{empheq}[left=\empheqlbrace, right=\empheqrbrace]{align*}
			0 <\; &r < \tau(u)\\
			\tau(u) - t_i < \;&r\\
			0 < \;&t_i
		\end{empheq}
	\end{minipage} and 
	\begin{minipage}[]{0.3\linewidth}
		\begin{empheq}[left=\empheqlbrace, right=\empheqrbrace]{align*}
			0 < \; &r < \tau(u)\\
			\tau(u) < \; &r + t_i\\
			0 < \; &t_i
		\end{empheq}
	\end{minipage}
\end{center}
are obviously equivalent, we can rewrite \eqref{laplacetransformeq} as
	\[
		\int_U \int_G \int_0^{\tau(u)} \int_{\tau(u)}^\infty \cdots \int_{\tau(u)}^\infty \varphi_0(u, h, r) \left( \prod_{i=1}^k \varphi_i(u, h, t_i)e^{-\xi_i(t_i-r)} \right) \, dt_1 \, \cdots \, dt_k \, dr \, d\omega \, d\nu^u
	\]
	by reparametrizing the domain of integration.
	Let us focus on the $k$ innermost integrals for now, where we can break up 
	\[
		\int_{\tau(u)}^\infty \cdots \int_{\tau(u)}^\infty \varphi_0(u, h, r) \left( \prod_{i=1}^k \varphi_i(u, h, t_i)e^{-\xi_i(t_i-r)} \right) \, dt_1 \, \cdots \, dt_k 
	\]
	into a sum of integrals
	\[
		\sum_{n_1, \ldots, n_k = 1}^\infty \int_{\tau^{(n_1)}(u)}^{\tau^{(n_1 + 1)}(u)} \cdots \int_{\tau^{(n_k)}(u)}^{\tau^{(n_k+1)}(u)} \varphi_0(u, h, r) \left( \prod_{i=1}^k \varphi_i(u, h, t_i)e^{-\xi_i(t_i-r)} \right) \, dt_1 \, \cdots \, dt_k 
	\]
	over intervals of the form $[\tau^{(n)}(u), \tau^{(n+1)}(u)]$. Of course, we can rewrite this as 
	\[
		\sum_{n_1, \ldots, n_k = 1}^\infty \int_{0}^{\tau(\sigma^{n_1}(u))} \cdots \int_{0}^{\tau(\sigma^{n_k}(u))} \varphi_0(u, h, r) \left( \prod_{i=1}^k \varphi_i(u, h, t_i + \tau^{(n_i)}(u))e^{-\xi_i(t_i-r + \tau^{(n_i)}(u))} \right) \, dt_1 \, \cdots \, dt_k 
	\]
	by changing variables, replacing $t_i$ with $t_i + \tau^{(n_i)}(u)$.
	Now, note that we can rewrite the integrand
	\[
		\varphi_0(u, h, r) \left( \prod_{i=1}^k \varphi_i\left(u, h, t_i + \tau^{(n_i)}(u)\right)  e^{-\xi_i(t_i - r + \tau^{(n_i)}(u))} \right)
	\]
	as
	\begin{equation}
		\varphi_0(u, h, r) \left( \prod_{i=1}^k\varphi_i\left( \sigma^{n_i}(u), \left( \Hol^{(n_i)}(u) \right)^{-1} \circ h, t_i \right) e^{-\xi_i(t_i - r + \tau^{(n_i)}(u))} \right)
		\label{integrand}
	\end{equation}
	using the identifications we made in constructing $U_{\tau, \Hol}$.
	If we integrate \eqref{integrand} with respect to $t_1$ through $t_k$, it becomes
	\begin{equation}
		\varphi_0(u,h,r) e^{r\xi} \left( \prod_{i=1}^k \hat\varphi_{i,\xi_i}\left( \sigma^{n_i}(u), \left( \Hol^{(n_i)}(u) \right)^{-1} \circ h \right) e^{-\xi_i \tau^{(n_i)}(u)} \right)
		\label{intintegrand}
	\end{equation}
	where we set $\xi \coloneqq \displaystyle \smash{\sum_{i=1}^k} \xi_i$. Similarly, integrating \eqref{intintegrand} with respect to $r$ yields
	\[
		\hat\varphi_{0,-\xi}(u,h) \left( \prod_{i=1}^k \hat\varphi_{i,\xi_i}\left( \sigma^{n_i}(u), \left( \Hol^{(n_i)}(u) \right)^{-1} \circ h \right) e^{-\xi_i \tau^{(n_i)}(u)} \right)
	\]
	and so \eqref{laplacetransformeq} becomes
	\begin{equation}
		\sum_{n_1, \ldots, n_k=1}^\infty \int_U \int_G \hat\varphi_{0,-\xi}(u,h) \left( \prod_{i=1}^k \hat\varphi_{i,\xi_i}\left( \sigma^{n_i}(u), \left( \Hol^{(n_i)}(u) \right)^{-1} \circ h \right) e^{-\xi_i \tau^{(n_i)}(u)} \right) \, d\omega \, d\nu^u
		\label{intintegrand1}
	\end{equation}
	after interchanging the order of integration and summation. 

	Since $\omega$ is bi-invariant, we can replace $h$ with $\left( \Hol^{(\max n_j)}(u) \right) \circ h$. Of course, we have the identity $\left( \Hol^{(n_i)}(u) \right)^{-1} \circ \Hol^{(\max n_j)}(u) = \Hol^{(n_i - (\max n_j))}(\sigma^{n_i}(u))$, and \eqref{intintegrand1} becomes
	\[
		\sum_{n_1, \ldots, n_k=1}^\infty \int_U \int_G \hat\varphi_{0,-\xi}\left(u,\Hol^{(\max n_j)}(u) \circ h\right) \left( \prod_{i=1}^k \hat\varphi_{i,\xi_i}\left( \sigma^{n_i}(u), \Hol^{(n_i - (\max n_j))}(u) \circ h \right) e^{-\xi_i \tau^{(n_i)}(u)} \right) \, d\omega \, d\nu^u
	\]
	with this change of variables. Once again, this becomes
	\[
		\sum_{n_1, \ldots, n_k = 1}^\infty \int_G \int_U \hat\varphi_{0,-\xi}\left(u,\Hol^{(\max n_j)}(u) \circ h\right) \left( \prod_{i=1}^k \hat\varphi_{i,\xi_i}\left( \sigma^{n_i}(u), \Hol^{(n_i - (\max n_j))}(u) \circ h \right) e^{-\xi_i \tau^{(n_i)}(u)} \right) \, d\nu^u \, d\omega
	\]
	by simply reversing the order of integration. Let us focus on the innermost integral 
	\begin{equation}
		\int_U \hat\varphi_{0,-\xi}\left(u,\Hol^{(\max n_j)}(u) \circ h\right) \left( \prod_{i=1}^k \hat\varphi_{i,\xi_i}\left( \sigma^{n_i}(u), \Hol^{(n_i - (\max n_j))}(u) \circ h \right) e^{-\xi_i \tau^{(n_i)}(u)} \right)  \, d\nu^u
		\label{integrand2}
	\end{equation}
	for the time being. Applying $\mathcal L_{P(\varsigma),0}$ to the integrand in \eqref{integrand2} a total of $\max n_j$ times yields 
	\[
		\sum_{{\sigma^{\max n_j}(u')=u}} \hspace{-17pt} \hat\varphi_{0,-\xi}\left(u',\Hol^{(\max n_j)}(u') \circ h\right) \left( \prod_{i=1}^k \hat\varphi_{i,\xi_i}\left( \sigma^{n_i}(u'), \Hol^{(n_i - (\max n_j))}(u') \circ h \right) e^{-\xi_i \tau^{(n_i)}(u')} \right) e^{\alpha_{P(\varsigma)}^{(\max n_j)}(u')}
	\]
	for any given values of $n_1, \ldots, n_k \in \bb N$. Now, observe that $e^{-\Re(\xi_i) \tau^{(n_i)}(u)}$ is at most $e^{- \Re(\xi_i) \tau^{(\max n_j)}(u)}$ so long as each $\Re(\xi_i)$ is negative. Hence, we can bound the magnitude of the integrand in \eqref{integrand2} above by
	\[
		\sum_{\sigma^{\max n_j}(u')=u} \hspace{-10pt} \left|\hat\varphi_{0,-\xi}\left(u',\Hol^{(\max n_j)}(u') \circ h\right)\right|  \left( \prod_{i=1}^k \left\|\hat\varphi_{i,\xi_i}\right\|_{C^0}   \right) e^{-\Re(\xi) \tau^{(\max n_j)}(u') +\alpha_{P(\varsigma)}^{(\max n_j)}(u')}
	\]
	using the triangle inequality. Rearranging this expression, we see that the magnitude of the integrand in \eqref{integrand2} is bounded above by
	\begin{equation}
		\left( \prod_{i=1}^k \left\|\hat\varphi_{i,\xi_i}\right\|_{C^0}   \right) \left(\sum_{\sigma^{\max n_j}(u')=u} \left|\hat\varphi_{0,-\xi}\left(u',\Hol^{(\max n_j)}(u') \circ h\right)\right|e^{-\Re(\xi) \tau^{(\max n_j)}(u') +\alpha_{P(\varsigma)}^{(\max n_j)}(u')} \right)
	\label{integrand3}
\end{equation}
which should be reminiscent of the expression defining the transfer operator. To make this concrete, recall that we have an $L^2(G)$-invariant decomposition of the function $\hat \varphi_{0,-\xi} \in C^1(U,L^2(G))$ in terms of its isotypic components
	\[
		\hat \varphi_{0,-\xi}\left( u', \Hol^{(\max n_j)}(u') \circ h \right) = \sum_{\rho} \hat \varphi_{0,-\xi}^\rho\left( u', \Hol^{(\max n_j)}(u') \circ h \right)
	\]
	where the sum is taken over irreducible representations $\rho$ of $G$ -- including the trivial representation. Once again, by the triangle inequality, we can bound \eqref{integrand3} above by
	\[
	\left( \prod_{i=1}^k \left\|\hat\varphi_{i,\xi_i}\right\|_{C^0}   \right) \left(\sum_\rho \sum_{\sigma^{\max n_j}(u')=u} \left|\hat\varphi^\rho_{0,-\xi}\left(u',\Hol^{(\max n_j)}(u') \circ h\right)\right|e^{-\xi \tau^{(\max n_j)}(u') +\alpha_{P(\varsigma)}^{(\max n_j)}(u')} \right)
	\]
	where we quickly recognize the transfer operator $\mathcal L_{\rho,P(\varsigma)+\Re(\xi)}^{\max n_j}$ applied to $\left|\hat \varphi^\rho_{0,-\xi} \right|$, noting of course that
	\[
		\left| \hat \varphi_{0,-\xi}^\rho \left( u', \Hol^{(\max n_j)}(u') \circ h \right) \right| = \left| \rho\left( \Hol^{(\max n_j)}(u') \right) \cdot \hat \varphi_{0,-\xi}^\rho \left( u', h \right)\right|
	\]
	for each $\rho$. Putting this back together, we see that \eqref{laplacetransformeq} is bounded above by
	\[
		\sum_\rho \sum_{n_1, \ldots, n_k=1}^\infty \int_U \int_G \left(\mathcal L^{\max n_j}_{\rho, P(\varsigma) + \Re(\xi)}\left|\hat\varphi^\rho_{0,-\xi}\right|\right)(u, h) \left( \prod_{i=1}^k \left\|\hat\varphi_{i,\xi_i}\right\|_{C^0} \right) \, d\omega \, d\nu^u
	\]
	in magnitude, as desired.
\end{proof}

We now simply use the bounds in Theorem \ref{spectralgap} to conclude that the Laplace transform $\hat \beta$ in Lemma \ref{laplacetransform} converges. 

\begin{theorem}
	With conditions as above, there are uniform constants $C > 0$ and $r < 1$ so that
	\[
		\left| \int_{U_{\tau, \Hol}} \varphi_0(u,h,r) \left( \prod_{i=1}^k \varphi_i(u,h,r+t_i) \right) \, d\nu^u \, d\omega \, dr \right| \leq C r^{\max t_j} \left(\|\varphi_0\|_{C^1} \cdot \ldots \cdot \|\varphi_k\|_{C^1}\right)
	\]
	for all $\varphi_0, \ldots, \varphi_k \in C^1(U_{\tau, \Hol}, \bb C)$ with $\int_{U_{\tau,\Hol}} \varphi_i \, d\nu^u \, d\omega \, dr = 0$.
	\label{spectralgapmixing}
\end{theorem}
\begin{proof}
	We assume that Theorem \ref{spectralgap} holds, and so we have
	\begin{equation}
		\|\mathcal L^n_{z,\rho} \varphi \|_{L^2(\nu^u)} \leq C \|\varphi\|_{C^1} r^n
		\label{spectralinequality}
	\end{equation}
	for all non-trivial irreducible $\rho$, all $\varphi \in C^1(U, V^\rho)$ and for each $z \in \bb C$ with $|\Re(z) - P(\varsigma)| < 1$. Up to choosing larger values of $C$ and $r$, we can also assume that the same inequality holds when $\rho$ is trivial -- this is precisely the main result in \cite{dolgopyatmixing}.
	We will show that the expression bounding $\hat\beta(\xi_1, \ldots, \xi_k)$ in Lemma \ref{laplacetransform} converges whenever the real parts of $\xi_i$ simultaneously lie in the interval $-\frac 1 k < \Re(\xi_i) < 0$. The decay desired will follow immediately from applying the inverse Laplace transform to the specific bounds we obtain.

	Fix $\xi_1, \ldots, \xi_k$ with $-\frac 1 k < \Re(\xi_i) < 0$. As before, we consider the function
	\[
		\hat \varphi_{i,\xi_i}(u,h) = \int_0^{\tau(u)} \varphi_i(u,h,t) e^{-\xi_it_i} \, dt_i
	\]
	and decompose $\hat \varphi_{0,-\xi}$ into its isotypic components
	\[
		\hat \varphi_{0,-\xi} = \sum_\rho \hat \varphi_{0,-\xi}^\rho
	\]
	where $\xi = \xi_1 + \ldots + \xi_k$ as before, noting that the decomposition of $L^2(G)$ into irreducible subspaces commutes with the Laplace transform. By \eqref{spectralinequality}, whenever $- \frac 1 k < \Re(\xi_i) < 0$, we have
	\begin{align*}
		\int_U \int_G \left( \mathcal L^{\max n_j}_{\rho,P(\varsigma) + \Re(\xi)} \left| \hat \varphi^\rho_{0,-\xi} \right| \right)(u,h) \, d\omega \, d\nu^u &\leq \left\| \left(\mathcal L^{\max n_j}_{\rho, P(\varsigma) + \Re(\xi)} \left|\hat \varphi^\rho_{0,-\xi} \right|\right)\right\|_{L^2(\nu^u)}\\
		&\leq C \|\hat \varphi_{0,-\xi}^\rho\|_{C^1} r^{\max n_j}
	\end{align*}
	for each $\rho$. Combining this with Lemma \ref{laplacetransform}, we are reduced to ensuring that
	\begin{equation}
		\sum_\rho \sum_{n_1, \ldots, n_k = 1}^\infty C \|\hat \varphi^\rho_{0,-\xi} \|_{C^1} r^{\max n_j} \left( \prod_{i=1}^k \| \hat \varphi_{i,\xi_i} \|_{C^0} \right) < k! \sum_\rho \sum_{n=1}^\infty C \|\hat \varphi^\rho_{0,-\xi} \|_{C^1} r^n \left( \prod_{i=1}^k \|\hat \varphi_{i,\xi_i} \|_{C^0} \right)
		\label{upperbound}
	\end{equation}
	converges. While this seems promising, note that we can only bound
	\begin{equation}
		\| \hat \varphi_{0,-\xi}^\rho \|_{C^1} \leq D \|\rho\| \left( \sup_{u \in U} \|\hat \varphi_{0,-\xi}^\rho (u, \cdot)\|_{L^2(G)} \right)
	\end{equation}
	for some constant $D > 0$. To salvage this, we will assume for the moment that we have chosen $\varphi_0 \in C^3(U_{\tau,\Hol},\bb R)$ so that we can invoke Theorem \ref{reptheorylemma} to bound
	\begin{equation}
		\|\rho\|\|\hat \varphi^\rho_{0,-\xi}(u,\cdot) \|_{L^2(G)} \leq \frac {D' \| \hat \varphi_{0,-\xi} (u,\cdot) \|_{C^{N+1}}}{\|\rho\|^N}
		\label{fouriereq}
	\end{equation}
	pointwise with a fixed constant $D' > 0$, where $N > 0$ is the constant guaranteed by Proposition \ref{seriesconverges}. With this, it is clear that the expression on the right side of \eqref{upperbound} converges absolutely, which says immediately that $\beta_k$ must decay exponentially fast in $t_1, \ldots, t_k$; this is almost what we wanted to show, but we need an explicit bound on $\hat \beta_k$ in order to obtain uniform estimates with the desired constants.

	Integrating the defining expression for $\hat \varphi_{i,\xi_i} (u,h)$ by parts, we have
	\[
		\hat \varphi_{i, \xi_i}(u,h) = \left. \frac{\varphi_i(u,h,t_i) e^{-\xi_it_i}}{-\xi_i} \right|_0^{\tau(u)}  + \frac 1 {\xi_i} \int_0^{\tau(u)} \left(\frac{\partial}{\partial t_i} \varphi_i (u,h,t_i)\right) e^{-\xi_i t_i} \, dt_i
	\]
	for each $u$ and $h$. Assuming that we choose $\varphi_i \in C^2(U_{\tau,\Hol},\bb R)$, we can crudely bound
	\[
		\| \hat \varphi_{i, \xi_i} \|_{C^1} \leq \frac{D'' \|\varphi_i\|_{C^2}}{1 + |\Im(\xi_i)|}
	\]
	for some constant $D''$ that depends only on $\|\tau\|_{C^1}$ -- note that it is essential here that $\Re(\xi_i)$ is confined to a bounded interval. We can similarly bound
	\[
	\left\| \hat \varphi_{0, -\xi} \right\|_{C^{N+1}} \leq \frac{D'' \|\varphi_0\|_{C^{N+2}}}{1 + |\Im(\xi)|}
	\]
	by choosing a larger value for $D''$ if necessary. Putting this all together, we have
	\[
		|\hat \beta_k(\xi_1, \ldots, \xi_k)| \leq \frac {C' \|\varphi_0\|_{C^{{N+2}}}}{1 + |\Im(\xi)|} \left( \prod_{i=1}^k \frac{\|\varphi_i\|_{\mathrlap{C^2}}}{1 + |\Im(\xi_i)|} \right)
	\]
	for all $\xi_i \in \bb C$ with $-\frac 1 k < \Re(\xi_i) < 0$. Now, we simply take the inverse Laplace transform in the variables $\xi_1$ through $\xi_k$ in succession to get
	\[
		|\beta_k(t_1, \ldots, t_k)| \leq \frac {e^{\Re(\xi_1)t_1+ \ldots + \Re(\xi_k) t_k}}{(2\pi)^k} \int_{-\infty}^\infty \cdots \int_{-\infty}^\infty \frac {\|\varphi_0\|_{C^{N+2}}} {1 + |s_1 + \ldots + s_k|} \left( \prod_{i=1}^k \frac {\|\varphi_i\|_{C^2}} {1+|s_i|} \right)\, ds_1 \, \ldots \, ds_k
	\]
	where $s_i$ is to be interpreted as the imaginary part of $\xi_i$. To complete our decay estimate, we simply need to show that the integral
	\begin{equation}
		\int_{-\infty}^\infty \cdots \int_{-\infty}^\infty \frac 1 {1 + |s_1 + \ldots + s_k|} \left( \prod_{i=1}^k \frac 1 {1 + |s_i|} \right) \, ds_1 \, \ldots \, ds_k
		\label{elementaryintegral}
	\end{equation}
	converges. Unfortunately, this is somewhat involved, so we postpone our remarks on the proof for the moment. Once the convergence of this integral has been established, we will have shown that
	\[
		|\beta_k(t_1, \ldots, t_k)| < Cr^{\max t_j} \|\varphi_0\|_{C^{N+2}} \left( \prod_{i=1}^k \|\varphi_i\|_{C^2}\right)
	\]
	for all $\varphi_0 \in C^{N+2}(U_{\tau,\Hol}, \bb R)$ and $\varphi_i \in C^2(U_{\tau,\Hol},\bb R)$. An identical argument to the one given in Lemma \ref{expmixingregularity} extends this to $C^1$ functions, using \cite[Lemma 2.4]{gorodnikspatzier} once again.
\end{proof}

We now indicate how to establish the convergence of the integral encountered in the preceding proof; we have included this for the sake of completeness, but the details are not relevant to the rest of our argument and can be safely skipped.

\begin{lemma}
	The integral
	\[
		\int_{-\infty}^\infty \cdots \int_{-\infty}^\infty \frac 1 {1+|s_1 + \ldots + s_k|} \left( \prod_{i=1}^k \frac 1 {1+|s_i|} \right) \, ds_1 \, \ldots \, ds_k
	\]
	converges.
\end{lemma}
\begin{proof}
	We will in fact show that the function defined by the integral
	\begin{equation}
		f(x) \coloneqq \int_{-\infty}^\infty \frac 1 {(1+|x+y|)^{0.5-\epsilon}} \left( \frac 1 {1 + |y|} \right) \, dy
		\label{inductionintegral}
	\end{equation}
	decays at a rate of
	\[
		|f(x)| \leq \frac C {(1 + |x|)^{0.5-2\epsilon}}
	\]
	for all sufficiently small $\epsilon > 0$; the statement of the lemma follows immediately by direct successive integration. We will work in the case when $x > 0$, and split the domain of integration in \eqref{inductionintegral} into regions where $x+y, y < 0$, where $x + y > 0$ but $y < 0$ and where $x + y, y > 0$. In the first case, where $x + y$ and $y$ are both negative, we evaluate
	\[
		\int_{-\infty}^{-x} \frac 1 {(1 - (x+y))^{0.5-\epsilon}} \left( \frac 1 {1 - y} \right) \, dy 
	\]
	for a fixed $x > 0$. One can verify that the antiderivative of this expression is given by
	\[
		\frac {\left( \frac {1-(x+y)}{1-y} \right)^{0.5-\epsilon} {}_2F_1(0.5-\epsilon,0.5-\epsilon;1.5-\epsilon;\frac{x}{1-y})}{(0.5-\epsilon) (1-(x+y))^{0.5-\epsilon}}
	\]
	where ${}_2F_1(\cdot,\cdot;\cdot;\cdot)$ is (the principal branch of the analytic continuation of) the Gaussian hypergeometric function. Hence, the integral evaluates to
	\[
		\int_{-\infty}^{-x} \frac 1 {(1 - (x+y))^{0.5-\epsilon}} \left( \frac 1 {1 - y} \right) \, dy = \frac {{}_2F_1(0.5-\epsilon,0.5-\epsilon;1.5-\epsilon;\frac{x}{1+x})}{(0.5-\epsilon)(1+x)^{0.5-\epsilon}}
	\]
	which can be bounded above by $\frac C {(1+x)^{0.5-\epsilon}}$ for an appropriate choice of constant $C > 0$, since by \cite[15.1.1]{abramowitzstegun} the defining series for ${}_2F_1(a,b;c;z)$ converges on the unit disk in the complex plane when $c - (a+b) > 0$.
	Similarly, on the region where $x+y$ is positive and $y$ is negative, we evaluate
	\begin{align*}
		\int_{-x}^0 \frac 1 {(1 + (x+y))^{0.5-\epsilon}} \left( \frac 1 {1 - y} \right) \, dy &= \left.\left( \frac{(1+(x+y))^{0.5+\epsilon} {}_2F_1(0.5+\epsilon,1;1.5+\epsilon;\frac{1+(x+y)}{2+x})}{(0.5+\epsilon)(2+x)} \right)\right|_{y=-x}^{y=0}\\
				&= \frac{(1+x)^{0.5+\epsilon} {}_2F_1(0.5+\epsilon,1;1.5+\epsilon;\frac{1+x}{2+x})}{(0.5+\epsilon)(2+x)} - \frac{{}_2F_1 (0.5+\epsilon,1;1.5+\epsilon;\frac 1 {2+x})}{(0.5+\epsilon)(2+x)}
	\end{align*}
	which can once again be bounded above by $\frac C {(1+x)^{0.5-\epsilon}}$ for the same reasons, though with a possibly larger choice of $C > 0$. Finally, when both $x+y$ and $y$ are positive, we evaluate
	\begin{align*}
		\int_0^\infty \frac 1 {(1+(x+y))^{0.5-\epsilon}} \left( \frac 1 {1+y} \right) \, dy &= \left. \left( \frac{\left( \frac {1+(x+y)}{1+y} \right)^{0.5-\epsilon} {}_2F_1(0.5-\epsilon,0.5-\epsilon;1.5-\epsilon;\frac{-x}{1+y})}{(0.5-\epsilon)(1+(x+y))^{0.5-\epsilon}} \right)\right|_{y=0}^{y=\infty}\\
			&= -\frac{{}_2F_1(0.5-\epsilon,0.5-\epsilon;1.5-\epsilon;-x)}{0.5-\epsilon}
	\end{align*}
	and we simply need to understand the asymptotics of ${}_2F_1(0.5-\epsilon,0.5-\epsilon;1.5-\epsilon;-x)$ as $x \to \infty$. By \cite[15.3.1]{abramowitzstegun}, we have an integral representation given by
\begin{equation}
	{}_2F_1 (0.5 - \epsilon, 0.5-\epsilon; 1.5-\epsilon;-x) = \frac{\Gamma(1.5-\epsilon)}{\Gamma(0.5-\epsilon)\Gamma(1)} \int_0^1 \frac 1 {t^{0.5+\epsilon}(1+tx)^{0.5-\epsilon}} \, dt
		\label{lasthypergeo}
	\end{equation}
	for $x > 1$. Setting $u = tx$, we can rewrite
	\begin{align*}
		\int_0^1 \frac 1 {t^{0.5+\epsilon}(1+tx)^{0.5-\epsilon}} \, dt &= \frac 1 {x^{0.5-\epsilon}} \int_0^x \frac 1 {u^{0.5+\epsilon}(1+u)^{0.5-\epsilon}} \, du\\
		&= \frac 1 {x^{0.5-\epsilon}} \left( \int_0^1 \frac 1 {u^{0.5+\epsilon}(1+u)^{0.5-\epsilon}} \, du + \int_1^x \frac 1 {u^{0.5+\epsilon}(1+u)^{0.5-\epsilon}} \, du \right)
	\end{align*}
	at which point we note that the expression in parentheses can be bounded above by $C (1 + \log x)$. Hence, \eqref{lasthypergeo} can be bounded above by $\frac C {(1+x)^{0.5 - 2\epsilon}}$ with a possibly larger choice of $C$, as desired. The proof in the case when $x < 0$ is identical.
\end{proof}

\section{Uniform local non-integrability from local \texorpdfstring{$G$}{G}-accessibility}

In this section, we will use the local accessibility of $f_t$ to establish the uniform local non-integrability estimates necessary to prove Theorem \ref{spectralgap}, drawing on techniques introduced by Dolgopyat in \cite{dolgopyatextension} for group extensions of expanding maps. These arguments require some additional care to adapt to our setting, with the principal difficulties stemming from the nontriviality of the fiber bundle $\pi \colon M \to N$ and the non-integrability of the strong stable and unstable foliations of $g_t$.

We want to translate the local accessibility of $f_t$ into an infinitesimal statement on the Markov model we constructed in \S 2; we will accomplish this in two main steps. The first step is to define a subalgebra of the Lie algebra $\mathfrak g$ of $G$ that measures the `non-integrability' of the fiber bundle over the weak stable and strong unstable foliations; this will be accomplished before making any reference to our symbolic model. The second step is to translate this into the symbolic model.

For most of what follows, we will need to be careful to specify which chart $V_*$ of the trivialization we are working with at any given point. This is a necessary complication to many of our arguments, since many of the objects we are working with are highly sensitive to the choice of trivialization. Fortunately, however, this will also afford us the flexibility later on to work with trivializations that are specially adapted to our needs.

To start, we want to measure and relate three different holonomies associated to $f_t$: namely, the holonomies induced by the leaves of the strong stable foliation, the leaves of the strong unstable foliation and the flow.

\begin{definition}
	Fix $x,y \in N$ with $y \in W^{su}_{g_t}(x)$, along with trivializations $\phi_{x}, \phi_{y}$ of $\pi \colon M \to N$ at $x$ and $y$ corresponding to subsets $V_x, V_y \subset N$ respectively. We define the \textit{unstable holonomy} \[\Theta_{V_x,V_y}^+(x,y)\colon F \to F\] between $x$ and $y$ to be the isometry induced by the map $\pi^{-1}(x) \to \pi^{-1}(y)$ that takes $a \in \pi^{-1}(x)$ to the (necessarily unique) point $b \in \pi^{-1}(y) \cap W^{su}_{f_t}(a)$. The identifications of $\pi^{-1}(x)$ and $\pi^{-1}(y)$ with $F$ are obtained via the trivializations $\phi_x, \phi_y$.
	The \textit{stable holonomy} $\Theta^-_{V_x,V_y}(x,y)$ is defined analogously for $y \in W^{ss}_{g_t}(x)$.
	\label{unstableholonomy2}
\end{definition}

\begin{definition}
	Fix $x,y \in N$ with $g_t(x) = y$, along with trivializations $\phi_x,\phi_y$ of $\pi \colon M \to N$ at $x$ and $y$ corresponding to subsets $V_x,V_y \subset N$ respectively. We define the \textit{temporal holonomy} \[\Hol_{\phi_x}^{\phi_y}(x,y) \colon F \to F\] between $x$ and $y$ to be the isometry induced by the map $\pi^{-1}(x) \to \pi^{-1}(y)$ that takes $a \in \pi^{-1}(x)$ to $f_t(a) \in \pi^{-1}(y)$. The identifications of $\pi^{-1}(x)$ and $\pi^{-1}(y)$ with $F$ are obtained via the trivializations $\phi_x, \phi_y$.
	\label{holonomy}
\end{definition}

By the end of this section, we will only need to work with a fixed, finite collection of trivializations that cover $N$. At this stage, however, the flexibility in these definitions will be crucial. Our first observation is that the unstable holonomy can be expressed in terms of the temporal holonomies induced by the flow; this is made precise in the following proposition, whose proof is largely summarized in Figure \ref{fig:unstableholonomy}.

 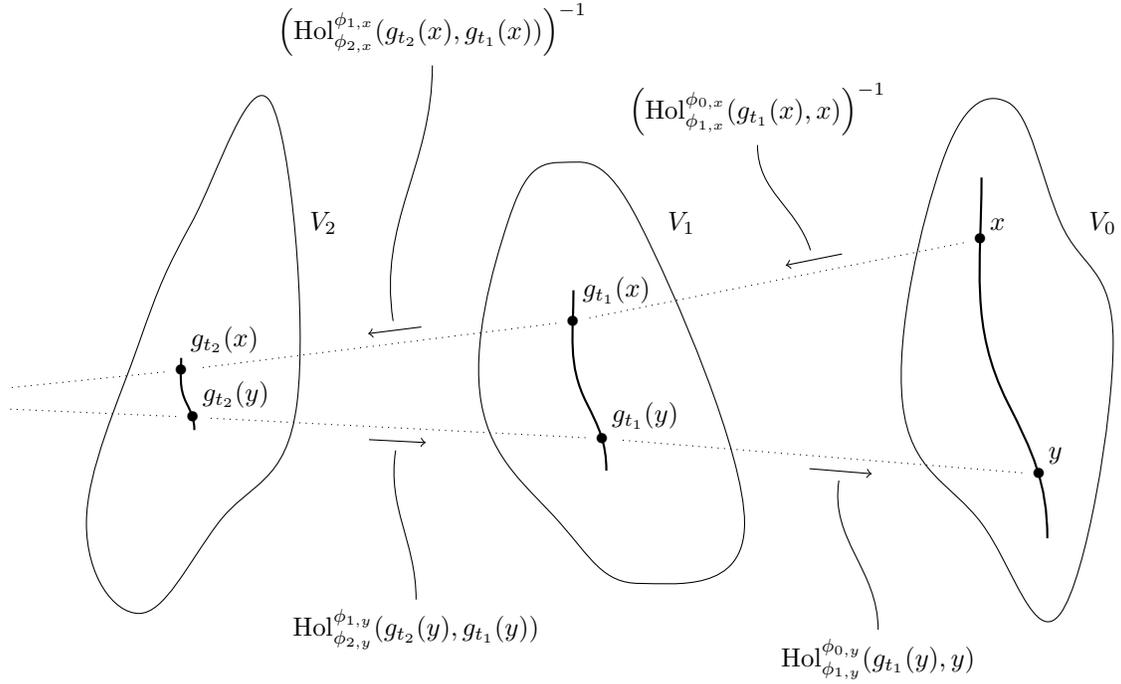
\begin{figure}[tb]
	 \centering
	 \begin{tikzpicture}[xscale=0.7,yscale=0.8]
		 \draw plot[smooth cycle, tension=0.7] coordinates {(0,0) (0.75,4) (2,5) (3,3) (4,1) (3,-3.5) (1.5,-2)};
		 \draw[thick] (1.5,3.75) to[out=-90,in=120] node[pos=0.25] (A) {$\bullet$} node[pos=0.25,above right] {$x$} (2,0)  to[out=-60,in=90] node[pos=0.55] (D) {$\bullet$} node[pos=0.55,above right] {$y$} (2.75,-2.25);
		 \draw[thick, xshift=-7cm,scale=0.5] (1.5,3.75) to[out=-90,in=120] node[pos=0.25] (B) {$\bullet$} node[pos=0.25,above right,yshift=3pt] {$g_{t_1}(x)$} (2,0)  to[out=-60,in=90] node[pos=0.55] (E) {$\bullet$} node[pos=0.55,above right] {$g_{t_1}(y)$} (2.75,-2.25);
		 \draw[thick, xshift=-14cm,scale=0.2] (1.5,3.75) to[out=-90,in=120] node[pos=0.25] (C) {$\bullet$} node[pos=0.25,above right,yshift=3pt] {$g_{t_2}(x)$} (2,0)  to[out=-60,in=90] node[pos=0.55] (F) {$\bullet$} node[pos=0.55,above right] {$g_{t_2}(y)$} (2.75,-2.25);
		 \draw[dotted] (A) -- node[yshift=5pt, pos=0.3] (H1) {} node[yshift=5pt, pos=0.5] (H2){} (B) -- node[yshift=5pt, pos=0.35] (I1) {} node[yshift=5pt, pos=0.55] (I2) {} (C);
		 \draw[->] (H1) -- node[pos=0.5] (H12) {} (H2);
		 \draw[->] (I1) -- node[pos=0.5] (I12) {} (I2);
		 \draw[dotted] (D) -- node[yshift=-5pt, pos=0.35] (J2) {} node[yshift=-5pt, pos=0.55] (J1) {} (E) -- node[yshift=-5pt, pos=0.4] (K2) {} node[yshift=-5pt, pos=0.6] (K1) {} (F);
		 \draw[->] (J1) -- node[pos=0.5] (J12) {} (J2);
		 \draw[->] (K1) -- node[pos=0.5] (K12) {} (K2);
	 \draw[dotted] (C) -- (-17,0.25);
	 \draw[dotted] (F) -- (-17,-0.1);
	 \node[xshift=-0.75cm,yshift=2cm] (H11) at (H12) {$\left(\Hol_{\phi_{1,x}}^{\phi_{0,x}}(g_{t_1}(x),x) \right)^{-1}$};
	 \draw[-] (H11) to[out=-90,in=110] (H12);
	 \node[overlay,xshift=0.5cm,yshift=4cm] (I11) at (I12) {$\left(\Hol_{\phi_{2,x}}^{\phi_{1,x}}(g_{t_2}(x),g_{t_1}(x)) \right)^{-1}$};
	 \draw[-] (I11) to[out=-90,in=100] (I12);
	 \node[overlay,xshift=0.25cm,yshift=-2.5cm] (K11) at (K12) {$\Hol_{\phi_{2,y}}^{\phi_{1,y}}(g_{t_2}(y), g_{t_1}(y))$};
	 \draw[-] (K11) to[out=90,in=-100] (K12);
 \node[xshift=0.5cm,yshift=-2.5cm] (J11) at (J12) {$\Hol_{\phi_{1,y}}^{\phi_{0,y}}(g_{t_1}(y), y)$};
	 \draw[-] (J11) to[out=90,in=-100] (J12);
	 \draw[xshift=-8cm] plot[smooth cycle, tension=0.7] coordinates {(0,0) (0.5,3) (1.5,4) (3,3) (5,-2) (3,-3) (1.5,-2)};
	 \draw[xshift=-11.5cm,xscale=-1] plot[smooth cycle, tension=0.7] coordinates {(0,0) (0.5,5) (2,3) (3,1) (4,-2) (3,-3.5) (1.5,-2)};
	 \node at (3.8,3) {$V_{0}$};
	 \node at (-4.2,3) {$V_{1}$};
	 \node at (-11,3) {$V_{2}$};
	 \end{tikzpicture}
	 \caption{Measuring the unstable holonomy between $x$ and $y$ along a sequence of times $0 > t_1 > t_2 > \ldots$ with respect to trivializations $(\phi_{n,x})$ and $(\phi_{n,y})$ defined over charts $V_{n}$, illustrated in the case when the trivializations for $x$ and $y$ coincide. As the unstable leaf through $x$ and $y$ contracts under $g_{t_n}$, the remaining contribution to the unstable holonomy decreases.}
	 \label{fig:unstableholonomy}
 \end{figure}

\begin{proposition}
	Fix $x, y \in N$ with $y \in W^{su}_{g_t}(x)$, along with trivializations $\phi_{0,x}$ and $\phi_{0,y}$ defined at $x$ and $y$ respectively. Let $T = (t_n)$ be a monotonic sequence of times with $t_n = 0$ and $t_n \to -\infty$, and let $I_x = \left(\phi_{n,x}\right)$ and $I_y = \left(\phi_{n,y}\right)$ be sequences of trivializations for which $\phi_{k,x} = \phi_{k,y}$ for all $k \geq N$. Then we can write
	\begin{equation}
		\Theta^+_{\phi_{0,x},\phi_{0,y}}(x,y) = \lim_{n \to \infty} \Hol^{(n)}_{I_y,T}(y) \left( \Hol^{(n)}_{I_x,T}(x) \right)^{-1}
		\label{holonomyeqinf}
	\end{equation}
	where
	\[
		\Hol^{(n)}_{I_*,T}(*) \coloneqq \Hol_{\phi_{n,*}}^{\phi_{n-1,*}}(g_{t_n}(*),g_{t_{n-1}}(*)) \circ \ldots \circ \Hol_{\phi_{1,*}}^{\phi_{0,*}}( g_{t_1}(*),*)
	\]
	is the $n$-step temporal holonomy measured with respect to the trivializations $I_\cdot$ at times given by $T$.
	\label{holonomyprop}
\end{proposition}
\begin{proof}
	The convergence of the limit is simply a consequence of the fact that $d(g_{t_n}(x),g_{t_n}(y)) \to 0$ as $n \to \infty$. More precisely, we can rewrite
	\[
		\Hol^{(n+1)}_{I_y,T}(y) \left( \Hol_{I_x,T}^{(n+1)}(x) \right)^{-1}
	\]
	as
	\[
		\Hol_{\phi_{n+1,y}}^{\phi_{n,y}}(g_{t_{n+1}}(y),g_{t_n}(y)) \circ \left(\Hol^{(n)}_{I_y,T}(y) \left( \Hol_{I_x,T}^{(n)}(x) \right)^{-1}\right) \circ \left(\Hol_{\phi_{n+1,x}}^{\phi_{n,x}}(g_{t_{n+1}}(x),g_{t_n}(x)) \right)^{-1}
	\]
	and since $f_t$ is $C^1$, we see that $\Hol_{\phi_{n+1,x}}^{\phi_{n,x}}(g_{t_{n+1}}(x),g_{t_n}(x))$ must also be locally $C^1$ in $g_{t_{n+1}}(x)$. Since $d_N(g_{t_{n+1}}(y),g_{t_{n+1}}(x))$ decay exponentially fast as $n \to \infty$ and the trivializations $I_x$ and $I_y$ eventually agree, we see that  
	\[
		d_G \left(  \Hol_{\phi_{n+1,y}}^{\phi_{n,y}}(g_{t_{n+1}}(y),g_{t_n}(y)) , \Hol_{\phi_{n+1,x}}^{\phi_{n,x}}(g_{t_{n+1}}(x),g_{t_n}(x)) \right)
	\]
	must also decay exponentially fast. In particular, for any $h \in G$, this means that
	\[
		d_G \left( \Hol_{\phi_{n+1,y}}^{\phi_{n,y}}(g_{t_{n+1}}(y),g_{t_n}(y)) \circ h \circ \left(\Hol_{\phi_{n+1,x}}^{\phi_{n,x}}(g_{t_{n+1}}(x),g_{t_n}(x))\right)^{-1},h  \right)
	\]
	decays exponentially fast and hence
	\[
		d_G\left( \Hol^{(n+1)}_{I_y, T}(y) \left( \Hol^{(n+1)}_{I_x, T}(x) \right)^{-1}, \Hol^{(n)}_{I_y, T}(y) \left( \Hol^{(n)}_{I_x, T}(x) \right)^{-1}  \right)
	\]
	decays exponentially fast as $n \to \infty$. Since $G$ is complete, the limit must exist.
	A similar argument shows that this limit is in fact equal to $\Theta^+_{\phi_{0,x},\phi_{0,y}}(x, y)$: since the unstable foliation of $f_t$ is invariant under the flow, we can rewrite
	\[
		\Theta^+_{\phi_{0,x},\phi_{0,y}}(x, y)
	\]
	as
	\begin{equation}
		\Hol^{(n)}_{I_y,T}(y) \circ \Theta^+_{\phi_{n,x},\phi_{n,y}}(g_{t_n}(x), g_{t_n}(y)) \circ \left( \Hol^{(n)}_{I_x, T}(x) \right)^{-1}
		\label{holonomyeq}
	\end{equation}
	for any $n > 0$ and any sequences $I_x, I_y$ and $T$ as above. As $t_n \to -\infty$, $d_N(g_{t_n}(x), g_{t_n}(y))$ decreases exponentially fast, and so $\Theta^+_{\phi_{n,x},\phi_{n,y}}(g_{t_n}(x), g_{t_n}(y))$ converges to the identity in $G$. Of course, this means that, as $n \to \infty$, \eqref{holonomyeq} converges to the limit in \eqref{holonomyeqinf}. Since the expression in \eqref{holonomyeq} is constant at $\Theta^+_{\phi_{0,x},\phi_{0,y}}(x,y)$, this proves the proposition.
\end{proof}

We need to understand the infinitesimal behaviour of the stable and unstable foliations - rather than working with the unstable holonomy as defined, we will instead consider its derivative along a leaf of the unstable foliation.

\begin{proposition}
	The unstable holonomy $\Theta^+_{\phi_1, \phi_2}(x,y)$ is simultaneously $C^1$ in $x$ and $y$, as $x$ and $y$ vary in a fixed leaf of the strong unstable foliation of $g_t$, and within charts associated to fixed $C^1$ trivializations $\phi_1$ and $\phi_2$.
	\label{unstholdif}
\end{proposition}
\begin{proof}
	This follows immediately from the fact that the leaves of the strong unstable foliation of $f_t$ are $C^1$.
\end{proof}

\begin{definition}
	Fix $x \in N$, a trivialization $\phi$ defined near $x$ and a vector $w \in T^1_xW^{su}_{g_t}(x)$. We define the \textit{infinitesimal holonomy} at $x$ in the direction of $w$ to be the element
	\[
		X_w^\phi(x) \coloneqq \left(\restr{\frac d {du}}{u=x}\left( \Theta^+_{V_x,V_x}(x, u) \right)\right)(w)
	\]
	of the Lie algebra $\mathfrak g$ of $G$.
	Let $\epsilon > 0$ be small enough that $\phi$ is defined over $B_\epsilon(x)$. The \textit{$\epsilon$-infinitesimal transitivity group} at $x$ is defined to be the linear span
	\[
		\mathfrak h_\epsilon^\phi(x) \coloneqq \Span_{y,w} \left(X_{w'}(y) - X_w(x)\right)
	\]
	taken over all $y \in W^{ss}_\epsilon(x)$ and $w \in T^1_xW^{su}_{g_t}(x)$. Here, $w'$ denotes the pushforward of $w$ to $T^1_yW^{su}_{g_t}(y)$ along the leaves of the center stable foliation of $g_t$.
	\label{inftransgrp}
\end{definition}

We will soon verify that $\mathfrak h^\phi(x)$ is largely independent of the choice of trivialization $\phi$, but it is worth making a few comments first.

\begin{remark}
	Under our hypotheses, the foliation $W^{ws}_{g_t}$ is $C^1$, and so the holonomy it induces between the leaves of the foliaton $W^{su}_{g_t}$ is also $C^1$. This is necessary for the pushforward of $w \in T^1_yW^{su}_{g_t}(y)$ in Definition \ref{inftransgrp} to make sense.
\end{remark}

\begin{remark}
	It is necessary to consider the \textit{relative} infinitesimal holonomy, as we did in Definition \ref{inftransgrp}. As we will see in the course of proving Proposition \ref{conjinvtchoice}, the vector $X_w^\phi(x)$ in Definition \ref{inftransgrp} is extremely sensitive to the choice of trivialization $\phi$. For instance, it is certainly possible for $X_w^\phi(x)$ to be $0$ for all $w \in T^1_xW^{su}_{g_t}(x)$ if the trivialization $\phi$ is built to be constant along the leaves of the strong unstable foliation, and the existence of such trivializations will be extremely helpful in the course of proving Theorem \ref{bptheorem}.
	\label{sensitivity}
\end{remark}

\begin{remark}
	The vectors $X_w^\phi(x)$ and $X_{w'}^\phi(y)$ vary continuously in $x$ and $w$, by Proposition \ref{unstholdif}. However, since $\mathfrak h_\epsilon(x)$ is defined as the linear span of continuously varying vectors, it is only lower semi-continuous. In particular, there can be singular sets where the dimension of $\mathfrak h_\epsilon(x)$ jumps down.
\end{remark}

It turns out that $\mathfrak h^\phi_\epsilon(x)$ is not particularly sensitive to $\epsilon$, though we will not prove this directly. We will show instead that, if $f_t$ is locally $G$-accessible, then $\mathfrak h_\epsilon(x)$ is generically equal to $\mathfrak g$. For most of what follows, we will treat $\epsilon > 0$ as a fixed constant with no particular restrictions. Our first important calculation is that the conjugacy class of $\mathfrak h_\epsilon^\phi(x)$ does not depend on the trivialization $\phi$, if the trivializations are chosen appropriately.

\begin{proposition}
	Fix $\epsilon > 0$, $x \in N$ and trivializations $\phi_{i} \colon \pi^{-1}(V_{i}) \to V_{i} \times F$ for $i = 1,2$. If $B_\epsilon(x) \subset V_{i}$, then
	\[
		\mathfrak h_\epsilon^{\phi_2}(x) = \Ad_{\id_{\phi_1}^{\phi_2}(x)} \left(\mathfrak h_\epsilon^{\phi_1}(x)\right)
	\]
	so long as $\phi_1$ and $\phi_2$ have constant projection to $F$ along each leaf of the strong stable foliation of $f_t$ and each flowline of $f_t$. 
	
	Here, $\id_{\phi_1}^{\phi_2}(x) \colon F \to F$ is used to denote the isometry induced by the identity map $\pi^{-1}(x) \to \pi^{-1}(x)$ with the domain and target identified with $F$ via $\phi_1$ and $\phi_2$ respectively.
	\label{conjinvtchoice}
\end{proposition}
\begin{proof}
	We can relate the unstable holonomies between $x$ and $u \in W^{su}_{g_t}(x)$ with respect to $\phi_1$ and $\phi_2$ by
	\begin{equation}
		\Theta^+_{\phi_2,\phi_2} \left( x,u \right) = \id_{\phi_1}^{\phi_2}(u) \circ \Theta^+_{\phi_1,\phi_1}\left( x,u \right) \circ \left( \id_{\phi_1}^{\phi_2}(x) \right)^{-1}
		\label{unstholrel}
	\end{equation}
	by definition. We now simply take the derivative of each side of \eqref{unstholrel} with respect to $u$ at $u = x$; in the notation of Definition \ref{inftransgrp}, this becomes
	\begin{equation}
		\label{firstinft}
		X_w^{\phi_2}(x) = \Ad_{\id_{\phi_1}^{\phi_2}(x)} \left( X_{w\phantom{'}}^{\phi_1}(x) \right) + \left((dR)_{\left(\id_{\phi_1}^{\phi_2}(x)\right)^{-1}} \circ d\left( \id_{\phi_1}^{\phi_2} \right)_x \right)(w)
	\end{equation}
for any $w \in T^1_xW^{su}_{g_t}(x)$, where $dR$ denotes the derivative of right multiplication in $G$. Given any $y \in W^{ss}_{\epsilon}(x)$ and $w'$ corresponding to $w$ as in Definition \ref{inftransgrp}, exactly the same calculation yields
\begin{equation}
		X_{w'}^{\phi_2}(y) = \Ad_{\id_{\phi_1}^{\phi_2}(y)} \left( X_{w'}^{\phi_1}(y) \right) + \left((dR)_{\left(\id_{\phi_1}^{\phi_2}(y)\right)^{-1}} \circ d\left( \id_{\phi_1}^{\phi_2} \right)_y \right)({w'})
		\label{secondinft}
	\end{equation}
	assuming, of course, that $y$ is sufficiently close to $x$ that we are able to use the same trivializations $\phi_1, \phi_2$.
	Since both trivializations are constant along the strong stable foliation of $f_t$ and we chose $y \in W^{ss}_{\epsilon}(x)$, we clearly have $\id_{\phi_1}^{\phi_2}(x) = \id_{\phi_1}^{\phi_2}(y)$ and hence
	\[
		(dR)_{\left(\id_{\phi_1}^{\phi_2}(x)\right)^{-1}} = (dR)_{\left(\id_{\phi_1}^{\phi_2}(y)\right)^{-1}}  
	\]
	as functions $T^1G \to T^1G$. Moreover, since the trivializations are also constant along the flowlines of $f_t$, we see that $\id_{\phi_1}^{\phi_2}$ must be constant along the leaves of the center stable foliation of $g_t$. Hence, we must have
	\[
		\left(d\left( \id_{\phi_1}^{\phi_2} \right)_x\right)(w) = \left( d\left( \id_{\phi_1}^{\phi_2} \right)_y \right)(w')
	\]
	for all $w \in T^1_xW^{su}_{g_t}(x)$. Subtracting \eqref{firstinft} from \eqref{secondinft} and using the fact that $\id_{\phi_1}^{\phi_2}(x) = \id_{\phi_1}^{\phi_2}(y)$, we then get
	\[
		X_{w'}^{\phi_2}(y) - X_{w}^{\phi_2}(x) = \Ad_{\id_{\phi_1}^{\phi_2}(y)} \left( X_{w'}^{\phi_1}(y) \right) - \Ad_{\id_{\phi_1}^{\phi_2}(x)} \left( X_{w \phantom{'}}^{\phi_1}(x) \right) 
	\]
	as desired.
\end{proof}

There is an analogous relation between the $\epsilon$-infinitesimal transitivity groups at points along a flowline of $g_t$, though the expansion of the unstable leaves prevent us from obtain an equality in this case.

\begin{proposition}
	Fix $\epsilon > 0$, $x \in N$ and $t > 0$. Let $\phi_x$ and $\phi_{g_t(x)}$ be trivializations near $x$ and $g_t(x)$ for which $B_\epsilon(x) \subset V_x$ and $B_\epsilon(g_t(x)) \subset V_{g_t(x)}$, and write $h$ for the temporal holonomy
	\[
		h(x) \coloneqq \Hol_{\phi_x}^{\phi_{g_t(x)}}( x, g_t(x) )
	\]
	measured with respect to $\phi_x$ and $\phi_{g_t}(x)$. We then have
	\[
		\Ad_{h(x)} \left( \mathfrak h_\epsilon^{\phi_x}(x) \right) \subset \mathfrak h_\epsilon^{\phi_{g_t(x)}}(g_t(x))
	\] 
	so long as $\phi_x$ and $\phi_{g_t(x)}$ have constant projection to $F$ along each leaf of the strong stable foliation of $f_t$ and each flowline of $f_t$.
	\label{conjinvtflow}
\end{proposition}
\begin{proof}
	By Proposition \ref{holonomyprop}, we have
	\[
		\Theta^+_{\phi_2,\phi_2}(g_t(x), g_t(u)) = h(u) \circ \Theta^+_{\phi_1, \phi_1}(x,u) \circ \left( h(x) \right)^{-1}
	\]
	so long as $u$ is sufficiently close to $x$. Noting the resemblance to \eqref{unstholrel}, simply repeating our calculations in Proposition \ref{conjinvtchoice} yields
	\[
		X_{w'}^{\phi_{g_t(x)}}(g_t(y)) - X_{w}^{\phi_{g_t(x)}}(g_t(x)) = \Ad_{h(x)} \left( X_{w'}^{\phi_x}(y) \right) - \Ad_{h(x)} \left( X_{w\phantom{'}}^{\phi_x}(x) \right)
	\]
	for all $y \in W^{ss}_{\epsilon}(x)$ and all $w \in T^1_xW^{su}_{g_t}(x)$. This completes the proof; note that we do not obtain equality this time since the strong stable leaves for $g_t$ contract, and there will be $y' \in W^{ss}_{\epsilon}(g_t(x))$ that are not of the form $g_t(y)$ for $y \in W^{ss}_\epsilon(x)$.
\end{proof}

Note that Proposition \ref{conjinvtflow} only yields an inclusion of the $\epsilon$-infinitesimal transitivty groups, and only in forward time. Our goal is to show that $\mathfrak h_\epsilon^\phi(x)$ is exactly $\mathfrak g$ at every $x \in N$; unfortunately, the proof of Proposition \ref{conjinvtflow} suggests that even the dimension of $\mathfrak h_\epsilon(x)$ may fail to be constant in general. Fortunately, given the topological transitivity of $g_t$, what we have proven so far is enough to show that the dimension is constant on a large set.

In light of Proposition \ref{conjinvtchoice}, we can be somewhat cavalier in specifying the trivialization $\phi$ used in defining $\mathfrak h_\epsilon^\phi(x)$, if we are solely concerned with the dimension and restrict our attention to trivializations that satisfy the hypotheses of the proposition. We will henceforth always assume that every trivialization we work with has constant projection to $F$ along the strong stable leaves and flowlines of $f_t$.

\begin{corollary}
	Fix a collection of trivializations $\phi_1, \ldots, \phi_k$ defined over a cover $V_1, \ldots, V_k$ of $N$, and let $\epsilon > 0$ be the Lebesgue number of this cover. Then $\dim \mathfrak h_\epsilon^*(\cdot)$ attains its maximum value on an open, dense subset of full measure.
	\label{maxdim}
\end{corollary}
\begin{proof}
	Let $x \in N$ be a point at which $\mathfrak h_\epsilon^*(x)$ has maximal dimension. Since $\mathfrak h_\epsilon^*(\cdot)$ is lower semi-continuous, it has maximal dimension on an open neighborhood $W$ containing $x$. By Proposition \ref{conjinvtflow}, $\mathfrak h_\epsilon^*(\cdot)$ therefore has maximal dimension on an open set containing the forward orbit of $g_t$. This is a dense set if $g_t$ is topologically transitive.

	Since $g_t$ is ergodic and $\dim \mathfrak h_\epsilon^*(\cdot)$ is measurable, it must be constant almost everywhere. The measure $\nu$ is an equilibrium measure with a H\"older potential and therefore has full support; hence, the open and dense set on which $\dim \mathfrak h_\epsilon^*(\cdot)$ has maximal dimension must also have full measure.
\end{proof}

Our next objective is to relate the $\epsilon$-infinitesimal transitivity groups between points along a leaf of the strong unstable foliation of $g_t$. If we indeed had equality in Proposition \ref{conjinvtflow}, this would be a relatively straightforward application of Proposition \ref{holonomyprop}. The lack of equality makes such an approach impossible, but we can still argue as in Corollary \ref{maxdim}.

\begin{lemma}
	Fix $\epsilon > 0$, $x \in N$ and $y \in W^{su}_{g_t}(x)$, along with trivializations $\phi_x$ and $\phi_y$ for which we have $B_{2\epsilon}(x) \subset V_x$ and $B_\epsilon(y) \subset V_y$. If $x$ is backwards-recurrent under $g_t$ and $\dim \mathfrak h_\epsilon^*(x)$ is maximal, then
	\[
		\mathfrak h_\epsilon^{\phi_y}(y) = \Ad_{\Theta^+_{\phi_x,\phi_y}(x,y)}\left( \mathfrak h_\epsilon^{\phi_x}(x) \right)
	\]
	and, in particular, $\dim \mathfrak h^*_\epsilon(\cdot)$ is constant on $W^{su}_{g_t}(x)$.
	\label{unsttrans}
\end{lemma}
\begin{proof}
	Since $\dim \mathfrak h^*_\epsilon(\cdot)$ is lower semi-continuous, there is an open set $W \subset B_\epsilon(x)$ on which it is maximal. Because $x$ is backwards recurrent, we can find a monotonic sequence of times $T = \left\{ t_n \right\}$ with $t_n \to -\infty$ for which $g_{t_n}(x) \in W$ for each $n > 0$. Moreover, we can suppose that $t_1$ is large enough that we also have $g_{t_n}(y) \in W$ for each $n > 0$.
	Now, write
	\[
		h^{(n)}(x) \coloneqq \Hol_{\phi_x}^{\phi_x}(g_{t_n}(x),g_{t_{n-1}}(x)) \circ \ldots \circ \Hol_{\phi_x}^{\phi_x}\left( g_{t_1}(x),x \right)
	\]
	and
	\[
		h^{(n)}(y) \coloneqq \Hol_{\phi_x}^{\phi_x}(g_{t_n}(y),g_{t_{n-1}}(y)) \circ \ldots \circ \Hol_{\phi_x}^{\phi_y}\left( g_{t_1}(y),y \right)
	\]
	for the $n$-step holonomies at $g_{t_n}(x)$ and $g_{t_n}(y)$. Since we have $g_{t_n}(x), g_{t_n}(y) \in W$ by construction, we have
	\[
		\Ad_{h^{(n)}(x)}\left( \mathfrak h_\epsilon^{\phi_x}\left( g_{t_n}(x) \right) \right) = \mathfrak h_\epsilon^{\phi_x}(x)
	\]
	and
	\[
		\Ad_{h^{(n)}(y)}\left( \mathfrak h_\epsilon^{\phi_x}\left( g_{t_n}(y) \right) \right) = \mathfrak h_\epsilon^{\phi_y}(y)
	\]
	by Proposition \ref{conjinvtflow} -- note that we have implicitly used the fact that $B_{2\epsilon}(x) \subset V_x$ in writing $\mathfrak h_\epsilon^{\phi_x}\left( g_{t_n}(x) \right)$ and $\mathfrak h_\epsilon^{\phi_x}\left( g_{t_n}(y) \right)$, where $V_x$ is the chart over which $\phi_x$ is defined. By rearranging these equations, we then have
	\[
		d_{\Gr(\dim \mathfrak h^*_\epsilon(x),\mathfrak g)}\left( \mathfrak h_\epsilon^{\phi_y}(y), \Ad_{h^{(n)}(y) \left( h^{(n)}(x) \right)^{-1}}\left( \mathfrak h_\epsilon^{\phi_x}(x) \right)  \right) = d_{\mathfrak \Gr(\dim \mathfrak h^*_\epsilon(x),\mathfrak g)}\left( \mathfrak h_\epsilon^{\phi_x}\left( g_{t_n}(y) \right), \mathfrak h_\epsilon^{\phi_x}\left( g_{t_n}(x) \right) \right)
	\]
	where distances are measured in the standard metric on the Grassmannian of $\left(\dim \mathfrak h^*_\epsilon(x)\right)$-dimensional subspaces of $\mathfrak g$. Since $\mathfrak h_\epsilon^{\phi_x}(\cdot)$ is lower semi-continuous and has maximal dimension on $W$, it must be continuous on $W$. Up to passage to the interior of a compact subset of $W$, we can assume that $\mathfrak h_\epsilon(\cdot)$ is uniformly continuous on $W$. Since $d_N(g_{t_n}(y),g_{t_n}(x)) \to 0$ as $n \to \infty$, we must then have
	\[
		d_{\mathfrak \Gr(\dim \mathfrak h^*_\epsilon(x),\mathfrak g)}\left( \mathfrak h_\epsilon^{\phi_x}\left( g_{t_n}(y) \right), \mathfrak h_\epsilon^{\phi_x}\left( g_{t_n}(x) \right) \right)
	\]
	as $n \to \infty$. This yields
	\[
		\mathfrak h_\epsilon^{\phi_y}(y) = \lim_{n \to \infty} \Ad_{h^{(n)}(y) \left( h^{(n)}(x) \right)^{-1}}\left( \mathfrak h_\epsilon^{\phi_x}(x)  \right)
	\]
	at which point we simply observe that $\Ad_g(\cdot)$ is continuous in $g$ and that $h^{(n)}(y) \left(h^{(n)}(x)\right)^{-1}$ converges to $\Theta^+_{\phi_x,\phi_y}(x,y)$ by Proposition \ref{holonomyprop}.
\end{proof}

In addition to the preceding lemma, we will require its analogue for the stable holonomies. The proof is identical, and we will not repeat it.

\begin{lemma}
	Fix $\epsilon > 0$, $x \in N$ and $y \in W^{ss}_{g_t}(x)$, along with trivializations $\phi_x$ and $\phi_y$ for which we have $B_{2\epsilon}(x) \subset V_x$ and $B_\epsilon(y) \subset V_y$. If $x$ is forwards-recurrent under $g_t$ and $\dim \mathfrak h^*_\epsilon(x)$ is maximal, then
	\[
		\mathfrak h_\epsilon^{\phi_y}(y) = \Ad_{\Theta^-_{\phi_x,\phi_y}(x,y)}\left( \mathfrak h_\epsilon^{\phi_x}(x) \right)
	\]
	and, in particular, $\dim \mathfrak h^*_\epsilon(\cdot)$ is constant on $W^{ss}_{g_t}(x)$.
	\label{stabtrans}
\end{lemma}

Now that we have Lemma \ref{unsttrans} and Lemma \ref{stabtrans} to connect the $\epsilon$-infinitesimal transitivity groups to the unstable and stable holonomies respectively, we can achieve the first major goal of this section: translating the local accessibility of $f_t$ into a statement about $\mathfrak h_\epsilon^*(\cdot)$. To begin, we will show that if $f_t$ is $G$-accessible, then $\mathfrak h_\epsilon^\phi(x)$ must be $\Ad_G$-invariant for any bi-recurrent $x \in N$.

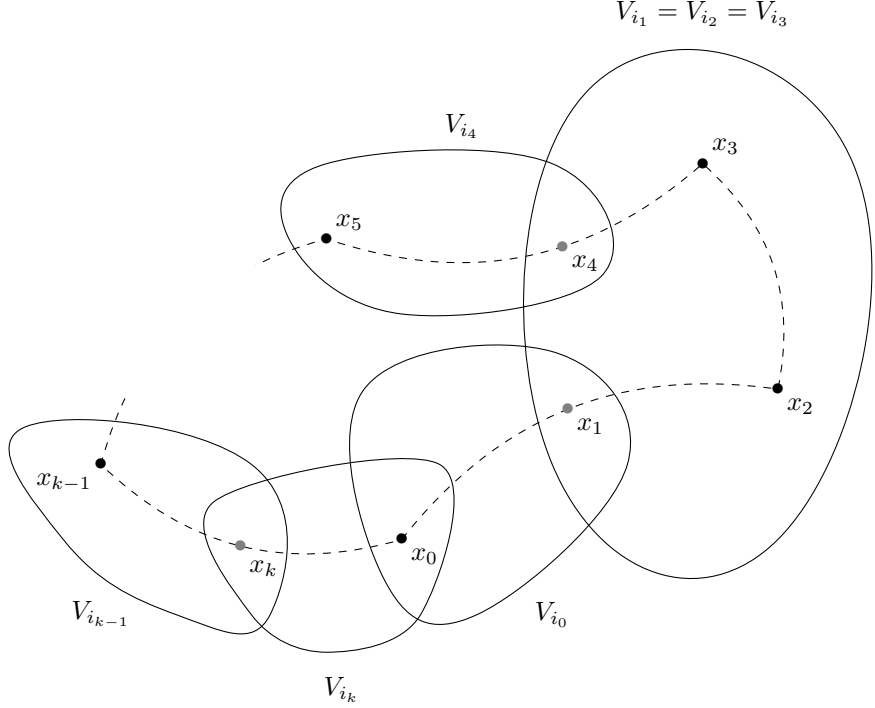
\begin{figure}[tb]
	\centering
	\begin{tikzpicture}
		\node at (0,0) {$\bullet$};
		\draw[dashed] (0,0) to[bend left] node[pos=0.5] (X1) {\textcolor{gray}{$\bullet$}} (5,2) node {$\bullet$} to[bend right] (4,5) node {$\bullet$} to[bend left] node[pos=0.4] (X4) {\textcolor{gray}{$\bullet$}} (-1,4) node (A) {$\bullet$} to[bend right] (-4,1) node (B) {$\bullet$} to[bend right] node[pos=0.5] (Xk) {\textcolor{gray}{$\bullet$}} (0,0);
		\draw plot[smooth cycle, tension=0.7] coordinates {(-0.25,-0.5) (-0.5,2) (2,2.5) (3,1) (1,-1)};
		\draw plot[smooth cycle, tension=0.9] coordinates {(2,1) (2.5,6) (6,5) (5,0)};
		\draw[opacity=0] (A) to[bend right] node[pos=0.5] (C) {} (B);
		\node[overlay, circle, fill=white, inner sep=25pt] at (C) {};
		\draw plot[smooth cycle, tension=0.7] coordinates {(-1.5,4) (-1,5) (2,5) (2.65,3.5) (0,3)};
		\draw plot[smooth cycle, tension=0.8] coordinates {(-5,1.5) (-2,1) (-1.75,-1) (-3,-1) (-4.5,0)};
		\draw plot[smooth cycle, tension=0.6] coordinates {(-2.5,0.5) (0.5,1) (0.25,-1) (-1,-1.5) (-2,-1)};
		\node[below right] at (X4) {$x_4$};
		\node[below right] at (0,0) {$x_0$};
		\node[below right] at (X1) {$x_1$};
		\node[below right] at (5,2) {$x_2$};
		\node[above right] at (4,5) {$x_3$};
		\node[above right] at (-1,4) {$x_5$};
		\node[below left] at (-4,1) {$x_{k-1}$};
		\node[below right, yshift=-2pt] at (Xk) {$x_k$};
		\node[overlay] at (2,-1) {$V_{i_0}$};
		\node[overlay] at (4,7) {$V_{i_1} = V_{i_2} = V_{i_3}$};
		\node[overlay] at (0.8,5.5) {$V_{i_4}$};
		\node[overlay] at (-4,-1) {$V_{i_{k-1}}$};
		\node at (-0.8,-2) {$V_{i_k}$};
	\end{tikzpicture}
	\caption{A refined stable-unstable sequence $x_0, x_1, \ldots, x_k, x_{k+1} = x_0$. Any stable-unstable sequence can be refined so that for each $0 \leq n \leq k$, there is a trivialization $\phi_{x_n}$ over a chart $V_{x_n}$ containing both $x_n$ and $x_{n+1}$. This refinement has the same total holonomy.}
	\label{fig:bpholonomy}
\end{figure}

\begin{theorem}
	Fix $\epsilon > 0$, $x \in N$ and a trivialization $\phi_x$ for which $B_\epsilon(x) \subset V_x$. Suppose that $2\epsilon$ is smaller than the Lebesgue number of a finite cover $\left\{ V_i \right\}$ of $N$ corresponding to trivializations $\left\{ \phi_i \right\}$. If $x$ is bi-recurrent under $g_t$, $\dim \mathfrak h_\epsilon^{\phi_x}(x)$ is maximal and $f_t$ is $G$-accessible, then $\mathfrak h_\epsilon^{\phi_x}(x)$ is $\Ad_G$-invariant.
	\label{bptheorem}
\end{theorem}
\begin{proof}
	Suppose that $x$ is bi-reccurent. Fix an isometry $g \in G$ and consider a stable-unstable sequence $x_0, x_1, \ldots, x_k, x_{k+1} = x_0$ in $N$ with $x_0 = x$ and whose total holonomy is $g$, where $x_{i+1}$ is either on the strong stable or strong unstable leaf through $x_i$ for $g_t$. For the total holonomy to be $g$, we want
	\[
		\id_{\phi_{x_k}}^{\phi_{x_0}}(x_0) \circ \Theta^\pm_{\phi_{x_k},\phi_{x_k}}\left( x_k,x_0 \right) \circ \ldots \id_{\phi_{x_1}}^{\phi_{x_2}}(x_2) \circ \Theta^\pm_{\phi_{x_1},\phi_{x_1}} \left( x_1,x_2 \right) \circ \id_{\phi_{x_0}}^{\phi_{x_1}}(x_1) \circ \Theta^\pm_{\phi_{x_0},\phi_{x_0}}\left( x_0,x_1 \right) = g
	\]
	where we can freely assume that each consecutive pair $x_i, x_{i+1}$ has a common trivialization $\phi_{x_i}$ for which $B_\epsilon(x_i),B_\epsilon(x_{i+1}) \subset V_{x_i}$ -- this is true up to refining the sequence. Suppose, moreover, that we have chosen $\phi_{x_0} = \phi_x$.

	We would like to now invoke Proposition \ref{conjinvtchoice}, Lemma \ref{unsttrans} and Lemma \ref{stabtrans} to show that $\mathfrak h^{\phi_x}_\epsilon(x)$ is $\Ad_g$-invariant for the $g$ corresponding to the total holonomy along this sequence. The $x_i$ we have chosen, however, may fail to be forwards- or backwards-recurrent as necessary. However, note that $\id_{\phi_{x_i}}^{\phi_{x_{i+1}}}(\cdot)$, $\Theta^+_{\phi_{x_i},\phi_{x_i}}(\cdot,\cdot)$ and $\Theta^-_{\phi_{x_i},\phi_{x_i}}(\cdot,\cdot)$ are all locally continuous in all of their arguments. Since bi-recurrent points are dense in $N$, given any $\delta > 0$, we can find a sequence of bi-recurrent points $x'_0, x'_1, \ldots, x'_k, x'_{k+1} = x'_0$ near $x_0, x_1, \ldots, x_k, x_{k+1} = x_0$ with $x'_0 = x$ and
	\[
		\id_{\phi_{x_k}}^{\phi_{x_0}}(x'_0) \circ \Theta^\pm_{\phi_{x_k},\phi_{x_k}}\left( x'_k,x'_0 \right) \circ \ldots \circ \id_{\phi_{x_1}}^{\phi_{x_2}}(x'_2) \circ \Theta^\pm_{\phi_{x_1},\phi_{x_1}} \left( x'_1,x'_2 \right) \circ \id_{\phi_{x_0}}^{\phi_{x_1}}(x'_1) \circ \Theta^\pm_{\phi_{x_0},\phi_{x_0}}\left( x'_0,x'_1 \right) = g'
	\]
	so that we have $d_G(g',g) < \delta$. Since each $x'_i$ is bi-recurrent, successive applications of Proposition \ref{conjinvtchoice}, Lemma \ref{unsttrans} and Lemma \ref{stabtrans} show that we have
	\[
		\mathfrak h_\epsilon^{\phi_x}(x) = \Ad_{g'}\left( \mathfrak h_\epsilon^{\phi_x}(x) \right)
	\]
	for $g'$ arbitrarily close to $G$. Since $\Ad_{g'}$ is continuous in $g'$, we then obtain
	\[
		\mathfrak h_\epsilon^{\phi_x}(x) = \Ad_{g}\left( \mathfrak h_\epsilon^{\phi_x}(x) \right)
	\]
	as desired.
\end{proof}

It is worth remarking that, under our standing assumption that trivializations must have constant projection to $F$ along strong stable leaves of $f_t$, the stable holonomies $\Theta^-_{\phi_{x_i},\phi_{x_i}}(x_i,x_{i+1})$ that appeared in the preceding proof must have all be trivial.

We want to show that $\mathfrak h_\epsilon^*(\cdot)$ is typically equal to the full Lie algebra $\mathfrak g$, though it seems unlikely that this should be true if we merely assume that $f_t$ is \textit{globally} $G$-accessible. With Theorem \ref{bptheorem}, however, we can show that having $\mathfrak h_\epsilon^\phi(x) = \mathfrak g$ at any point is (typically) equivalent to $f_t$ being \textit{locally} $G$-accessible at that point.

\begin{theorem}
	Fix $\epsilon > 0$, $x \in N$ and a trivialization $\phi_x$ defined over $V_x \subset N$ for which we have $B_{3\epsilon}(x) \subset V_x$. Moreover, suppose that $\mathfrak h_\epsilon^{\phi_x}(\cdot)$ is continuous and has maximal dimension on $B_{3\epsilon}(x)$, and that the forwards orbit of $x$ under $g_t$ is dense in $B_{2\epsilon}(x)$. If $f_t$ is locally $G$-accessible at $x$, then $\mathfrak h_\epsilon^{\phi_x}(x) = \mathfrak g$.
	\label{inftransfull}
\end{theorem}
\begin{proof}
	Without loss of generality, suppose that we have chosen $\phi_x$ so that $\Theta^+_{\phi_x,\phi_x}(x,u)$ is trivial for all $u \in \left(W^{su}_{g_t}(x) \cap B_{2\epsilon}(x) \right)^\circ$, and so that $\Theta^-_{\phi_x,\phi_x}(s_1,s_2)$ is trivial whenever $s_1$ and $s_2$ lie on the same (local) leaf of the strong stable foliation of $g_t$ in $B_{2\epsilon}(x)$.

	We will write $\mathfrak h \coloneqq \mathfrak h_\epsilon^{\phi_x}(x)$ and suppose for the sake of contradiction that $\mathfrak h \subsetneq \mathfrak g$ is a proper subalgebra. By Proposition \ref{unsttrans} and our choice of $\phi_x$, we must have $\mathfrak h_\epsilon^{\phi_x}(u) = \mathfrak h$ for all $u \in B_\epsilon(x)$ lying on the local leaf through $x$ of the strong unstable foliation of $g_t$. Hence, for any $u_1 \in B_\epsilon(x)$, each vector $X_{w'}^{\phi_x}(u_1)$ used in the definition of $\mathfrak h_\epsilon^{\phi_x}(u)$ must lie in the Lie algebra $\mathfrak h$. Integrating this, we see that the unstable holonomies $\Theta^+_{\phi_x,\phi_x}(u_1,u_2)$ are constrained to $\exp(\mathfrak h)$ for all $u_1,u_2 \in B_\epsilon(x)$ that lie on the same local leaf of the strong unstable foliation of $g_t$.

	Now, consider $H \coloneqq \exp(\mathfrak h)$ and consider an element $g \in G \setminus H$ that lies in the complement. By our construction of $\phi_x$, all unstable holonomies are constrained to $H$ and all stable holonomies are trivial -- hence, no local sequence of stable and unstable holonomies along a sequence of points $x, x_1, x_2, \ldots x_k, x$ lying in $B_\epsilon(x)$ can result in a total holonomy of $g$. Moreover, since $\mathfrak h$ is an ideal by Theorem \ref{bptheorem}, $H$ is a normal subgroup of $G$; hence, we cannot obtain a total holonomy of $g$ for \textit{any} choice of trivialization. Since $f_t$ is locally $G$-accessible at $x$, this is a contradiction.
\end{proof}

Theorem \ref{inftransfull} gives us the infinitesimal analogue of local accessibility that we sought, and all that we need to do now is verify that this translates properly into the symbolic model we are working with. In principle, the difficulty is that unstable leaves for the discrete dynamical system $(R, \mathcal P)$ are typically not unstable leaves for $g_t$ -- fortunately, our choice of trivializations will circumvent almost all of these problems.

Recall that $(R, \mathcal P)$ is a Markov partition associated to $g_t$, which descends to a $C^1$ expanding model $(U, \sigma)$ by projecting along leaves of the strong stable foliation of $g_t$. We want to define unstable holonomies entirely within the symbolic model; a natural candidate for a definition comes from Proposition \ref{holonomyprop}. For everything that follows, we will assume that each $R_i \subset N$ has been assigned a fixed trivialization $\phi_i$ defined on a neighborhood $B_\epsilon(R_i)$, with the property that $\phi_i$ has constant projection to $F$ along each leaf of the strong stable foliation of $f_t$ and each flowline of $f_t$. By choosing and fixing trivializations at each point in $R$, we will no longer need to specify which trivialization we are using, at least when dealing with the symbolic model.

\begin{definition}
	Fix $x \in R$ and $y \in W^{su}_{\mathcal P}(x)$. The \textit{symbolic unstable holonomy} $\Theta^+_{\symb}(x,y)$ is defined to be the limit
	\[
		\Theta^+_{\symb}(x,y) \coloneqq \lim_{n \to \infty} \Hol^{(n)}\left(\mathcal P^{-n}(y)\right) \left( \Hol^{(n)}\left( \mathcal P^{-n}(x) \right) \right)^{-1}
	\]
	where we write
	\[
		\Hol^{(n)}(u) \coloneqq \Hol\left( \mathcal P^{n-1}(u) \right) \circ \ldots \circ \Hol\left( u \right)
	\]
	for the $n^{\rm th}$ holonomy under the Poincar\'e return map.
	\label{symbholdef}
\end{definition}

It is straightforward to verify that this limit exists. Moreover, our choice of trivializations ensures that the symbolic unstable holonomies agree with the appropriate (non-symbolic) unstable holonomies.

\begin{proposition}
	Fix $x \in R$ and $y \in W^{su}_{\mathcal P}(x)$. There is a (possibly negative) $t$ so that $g_t(y) \in W^{su}_{g_t}(x)$ for which we have
	\[
		\Theta^+_{\symb}(x,y) = \Theta^+_{\phi_x,\phi_y}(x,g_t(y))
	\]
	assuming $x$ and $y$ are sufficiently close, where $\phi_x,\phi_y$ are trivializations corresponding to the respective parts of the Markov partition.
	\label{samehols}
\end{proposition}
\begin{proof}
	The existence of such a $t$ satisfying $|t| < \tau(y), \tau\left( \mathcal P^{-1}(y) \right)$ follows immediately from the construction of the Markov partition $(R, \mathcal P)$. We then clearly have
	\[
		\Theta^+_{\symb}(x,y) = \Hol_{\phi_y}^{\phi_y}(g_t(y),y) \circ \Theta^+_{\phi_x,\phi_y}(x, g_t(y))
	\]
	by Proposition \ref{holonomyprop} and Definition \ref{symbholdef}. By our choice of trivialization and the fact that $t$ does not exceed the return time of $y$, $\Hol_{\phi_y}^{\phi_y}(g_t(y),y)$ is the identity in $G$, as desired.
\end{proof}

And now, we can analogously define the \textit{symbolic} infinitesimal transitivity group:

\begin{definition}
	Fix $x \in U_i$ and a vector $w \in T^1_xU_i$. We define the \textit{symbolic infinitesimal holonomy} at $x$ in the direction of $w$ to be the element
	\[
		X_w^{\symb}(x) \coloneqq \left( \restr{\frac d {du}}{u = x} \left( \Theta^+_{\symb}\left( x,u \right) \right) \right)(w)
	\]
	of the Lie algebra $\mathfrak g$ of $G$. The \textit{symbolic infinitesimal transitivity group} at $x$ is defined to be th linear span
	\[
		\mathfrak h_{\symb}(x) \coloneqq \Span_{s,s',w} \left( X_{w'}([x,s']) - X_w([x,s] \right)
	\]
	where we take the span over $s,s' \in S_i$, $w \in T_x^1U_i$ and let $w'$ be the projection of $w$ to $[x,s']$ via center-stable leaves followed by the flow.
	\label{symbtransgrp}
\end{definition}

With very little work, we can now prove

\begin{theorem}
	Fix a bi-recurrent $x \in U_i$ at which $\dim \mathfrak h_\epsilon(x)$ is maximal, and suppose that $f_t$ is $G$-accessible. Then $\mathfrak h_\symb(x) = \mathfrak g$.
\end{theorem}
\begin{proof}
	By Proposition \ref{samehols}, the unstable holonomies used in Definition \ref{inftransgrp} and Definition \ref{symbtransgrp} are the same. Hence, the regular and symbolic transitivity groups agree, and so by Theorem \ref{bptheorem}, we have $\mathfrak h_\symb(x) = \mathfrak g$.
\end{proof}

This is almost the result we want, but we need to use some Lie theory to extract the explicit estimates that we will use in the next section. We will want to phrase this in terms of $\sigma$, which means minor notational changes in the preceding theorems. Recall that $\sigma \colon U \to U$ is not actually invertible; to make sense of the inverse, we must choose branches of $\sigma^{-n}$ locally.

\begin{definition}
	A \textit{consistent past} for $u \in U_i$ is a sequence of maps $\{v^{(n)} \colon U_i \to U_{j(n)} \mid n \geq 0\}$ where $v^{(0)} = \restr{\id}{U_i}$ and $\sigma \circ v^{(n)} = v^{(n-1)}$.
\end{definition}

\begin{remark}
	A consistent past for $u \in U_i$ corresponds exactly to a choice of stable element $s \in S_i$ -- we can recover the maps $\left\{ v^{(n)} \right\}$ by projecting the Poincar\'e return map $\mathcal P^{(-n)}$ along leaves of the strong stable foliation.
\end{remark}

Finally, we can establish the main estimate.

\begin{theorem}
	\label{nli}
	Let $\rho$ be an isotypic component of the representation of $G$ on $L^2(G)$. There is an open subset $U_{lni} \subset U$ and constants $\epsilon, \delta, n_0 > 0$ so that, for any $x \in U_{lni}$ and any $\varphi \in C^1(U, V^\rho)$, there are consistent pasts $v_1 = \left\{v_1^{(n)}\right\}$ and $v_2 = \left\{v_2^{(n)}\right\}$ defined near $x$ and a $C^1$ vector field $w \colon U_{lni} \to T^1U_{lni}$ satisfying
	\[
		\left\| \left(d\rho\left( X_{w(u),v_1}^\symb(u) - X_{w(u),v_2}^\symb(u) \right)\right)(\varphi(u)) \right\|_{L^2(G)} \geq \epsilon \|\rho\|
	\]
	for all $u \in B_\delta(x)$.
\end{theorem}
\begin{proof}
	By definition, we have
	\[
		X_{w,v_i}^\symb(x) = \restr{\frac d {du}}{u=x} \left( \lim_{n \to \infty}  \Hol^{(n)}\left( v_i^{(n)}(u) \right) \left( \Hol^{(n)}(v_i^{(n)}(x) \right) \right)
	\]
	for all $u \in B_\delta(x)$. The derivatives of the terms in the sequence converge exponentially fast in $n$, so we can interchange the limit and the derivative to get
	\[
		X_{w,v_i}^\symb(x) = \lim_{n\to \infty} \left( \restr{\frac d {du}}{u=x} \left( \Hol^{(n)}(v_i^{(n)}(y)) \left( \Hol^{(n)}\left( v_i^{(n)}(x) \right) \right)^{-1} \right)  \right)
	\]
	for any consistent past $v_i$. Since the $X_{w(u),v_i}^\symb - X_{w(u),v_j}^\symb$ taken over pasts $v_i, v_j$ and vectors $w$ form a basis of $\mathfrak g$, there is a finite $n_0$ so that the approximations
	\[
		X_{w,v_i^{(n_0)}}^{\symb}(x) \coloneqq \restr{\frac d {du}}{u=x} \left( \Hol^{(n_0)}(v_i^{(n_0)}(y)) \left( \Hol^{(n_0)}\left( v_i^{(n_0)}(x) \right) \right)^{-1} \right)
	\]
	can also be used to form a basis $X_{w(u),v_i^{(n_0)}}^{\symb}(x) - X_{w(u),v_j^{(n_0)}}^{\symb}(x)$, taken again over pasts $v_i, v_j$ and vectors $w$. Hence, we have a Casimir
	\[
		\Omega = \sum \left( g_{ij} \right)^{-1} \left( X_{w,v_i^{(n_0)}}^{\symb}(x) - X_{w,v_j^{(n_0)}}^{\symb}(x) \right) \left( X_{w,v_i^{(n_0)}}^{\symb}(x) - X_{w,v_j^{(n_0)}}^{\symb}(x) \right)
	\]
	for some finite collection of pasts $v_i, v_j$ and a given vector $w$. This acts on $V^\rho$ by scalar multiplication by $\|\rho\|^2$, and so we get
	\[
		d\rho\left(\sum (g_{ij})^{-1} \left( X_{w,v_i^{(n_0)}}^{\symb}(x) - X_{w,v_j^{(n_0)}}^{\symb}(x) \right)^2\right)\left( \varphi(x) \right) = \|\rho\|^2 \varphi(x)
	\]
	for this collection of pasts, and the same vector $w$. Now, there is a uniform constant $\epsilon > 0$ -- depending only on the $g_{ij}$ and $\mathfrak g$ -- so that
	\[
		\left\| d\rho\left( X_{w,v_i^{(n_0)}}^{\symb}(x) - X_{w,v_j^{(n_0)}}^{\symb}(x) \right) \left( \varphi(x) \right) \right\|_{L^2(G)} \geq \epsilon \|\rho\| \varphi(x)
	\]
	for some choice of $v_i, v_j$ and $w$. Since $X_{w,v}^\symb(u)$ varies continuously in $w$ and $u$, there is a neighborhood $B_\delta(x)$ of $x$ in $U$ and a $C^1$ vector field $w \colon B_\delta(x) \to T^1B_\delta(x)$ for which
	\[
		\left\| d\rho\left( X_{w(u),v_i^{(n_0)}}^{\symb}(u) - X_{w(u),v_j^{(n_0)}}^{\symb}(u) \right) \left( \varphi(u) \right) \right\|_{L^2(G)} \geq \frac \epsilon 2 \|\rho\| \varphi(u)
	\]
	holds for all $u \in B_\delta(x)$, as desired.
\end{proof}

\label{snli}

\section{\texorpdfstring{Spectral bounds for $\mathcal L^n_{z,\rho}$}{Spectral bounds for the transfer operators}}
\label{transopsec}

In this section, we establish bounds for the twisted transfer operators $\mathcal L^n_{z,\rho}$ acting on $C^1(U, V^\rho)$ with respect to the $L^2(\nu^u)$-norm. The key ingredients are the local non-integrability estimate in Theorem \ref{nli}, the diametric regularity of the measure $\nu^u$ and the $C^1$ regularity of $\alpha_z$ from the definition of the transfer operators.

The strategy adopted in this section is by now classical, dating back to Dolgopyat in \cite{dolgopyatmixing} and \cite{dolgopyatextension}; an account of this almost completely adapted to our setting was also given by Winter in \cite{winter}. The reader already familiar with these arguments should find no surprises in this section, but we feel it necessary to include them given their delicate nature and the minor differences in our contexts.

We begin by recalling the definition of the twisted transfer operator
\[
	\left(\mathcal L^n_{z,\rho}\varphi\right)(u) \coloneqq \sum_{\sigma^n(u')=u} e^{\alpha_z^{(n)}(u')} \rho\left( \Hol^{(n)}(u') \right) \varphi(u')
\]
associated to an irreducible representation $\rho$ of $G$ on $V^\rho \subset L^2(G)$. In principle, the difficulty in obtaining contraction for $\left\|\mathcal L^n_{z,\rho} \varphi \right\|_{L^2(\nu^u)}$ as $n \to \infty$ lies in the possibility that the rotation introduced by the action of $\rho$ may cause resonances between the vectors $\varphi(u') \in V^\rho$, and this may happen on a set of large measure.

The local non-integrability estimate provided by Theorem \ref{nli}, however, suggests that we should typically be able to find $u_1', u_2'$ with $\sigma^n(u_1') = \sigma^n(u_2') \in U_{lni}$ so that
\[
	\rho \left( \Hol^{(n)}(u_1') \right) \varphi(u_1')\qquad \text{and}\qquad \rho \left(  \Hol^{(n)}(u_2') \right) \varphi(u_2')
\]
are `uniformly' non-parallel. The main argument in the section boils down to verifying that this can be accomplished on an adequately large set, with explicit uniformity estimates.

Throughout this section, we will work with a fixed isotypic component $V^\rho$ of the regular representation of $G$ on $L^2(G)$. However, it is worth noting that most of the intermediate constants will fundamentally depend on $\rho$, and keeping track of these dependencies is essential to obtaining a final bound in Theorem \ref{spectralgapmaintheorem} that is independent of $\rho$.

Though we need to deal with $\mathcal L^n_{z,\rho} \varphi$ for any $\varphi \in C^1(U, V^\rho)$, it will be helpful to work instead with slightly more regular real-valued functions $\Phi \in C^1(U, \bb R^+)$ with bounded logarithmic derivative; in other words,we will require
	\begin{align*}	
		\sup_{w \in T_u^1U} \left| \left( d\Phi \right)_u(w) \right| &< C \Phi(u)
	\end{align*}
	for all $u \in U$. We will use $\mathcal K_C$ to denote the class of such functions
\[
	\mathcal K_C \coloneqq \left\{ \Phi \in C^1(U, \bb R^+) \mid \text{$\log \Phi$ is $C$-Lipschitz} \right\}
\]
for any constant $C > 0$.

Since $g_t$ is Anosov, the expansion rates of $g_t$ on $U$ are bounded away from $1$. Since the return times $\tau \colon R \to \bb R$ are bounded away from $0$, the slowest expansion rates of $d\sigma^n$ on $U$ are therefore also bounded away from $1$. For what follows, let $f \kappa^n$ and $b K^n$ with $1 < \kappa < K$ be the slowest and fastest expansion rate of any unit vector in $T^1U$ under $d\sigma^n$.

We are interested in functions $\Phi \in \mathcal K_C$ with bounded logarithmic derivative because they can be used to control less regular functions $\varphi \in C^1(U, \bb R)$. The following lemma makes this precise.

\begin{lemma}
	Fix $C >  0$, $\varphi \in C^1(U, V^\rho)$ and $\Phi \in \mathcal K_C$. There is a $\delta > 0$ so that, if
	\begin{align*}
		\left\| \varphi(u) \right\|_{L^2(G)} &< \Phi(u)\\
		\sup_{w \in T_u^1U} \left\| \left( d\varphi \right)_u(w) \right\|_{L^2(G)} &< C \Phi(u)
	\end{align*}
	for all $u \in U$, then for any $u_0 \in U$,
	\begin{align*}
		{\left\| \varphi\left(v^{(n)}(u)\right) \right\|_{L^2(G)}} &\leq \frac 3 4 {\Phi\left(v^{(n)}(u)\right)}\\
		\intertext{or}
		{\left\| \varphi\left(v^{(n)}(u)\right) \right\|_{L^2(G)}}&\geq \frac 1 4{\Phi\left(v^{(n)}(u)\right)}
	\end{align*}
	for all $u \in B_\delta(u_0)$, each $n > 0$ and each consistent past $v = \left\{ v^{(n)} \right\}$ defined near $u_0$. Moreover, for any $u_0 \in U$,
	\[
		\Phi\left(v^{(n)}(y)\right) \leq 2 \Phi\left(v^{(n)}(x)\right)
	\]
	for all $x, y \in B_\delta(u_0)$. The choice of constant $\delta > 0$ can be made so that we have $\delta C = A$ for some uniform constant $A$, which does not depend on $\rho$, $\Phi$, $\varphi$, $u_0$, $v$ or $n$.
	\label{alternative}
\end{lemma}
\begin{proof} 
	Since we chose $\Phi \in \mathcal K_C$, $\log \Phi$ is $C$-Lipschitz and we therefore have
	\begin{align*}
		\left| \log \left( \Phi\left(v^{(n)}(y)\right) \right) - \log \left( \Phi\left( v^{(n)}(x) \right) \right) \right| &\leq \frac C {f\kappa^n} d(x,y)\\
		&\leq \frac C f d(x,y)
	\end{align*}
	for any $x,y \in U$. Exponentiating both sides, this means
	\[
		\Phi\left( v^{(n)}(y) \right) \leq e^{ \frac C {f} d(x,y)} \Phi\left( v^{(n)}(x) \right)
	\]
	for any $x,y \in U$. Now, suppose that $d(x,y) \leq 2\delta$ for some $\delta > 0$, and fix a unit speed path $\gamma \colon [0,2\delta] \to U_0$ with $\gamma(0) = y$ and $\gamma(2\delta) = x$. We then have 
	\begin{align*}
		\left| \left\|\varphi\left(v^{(n)}(x)\right) \right|_{L^2(G)} - \left\|\varphi\left( v^{(n)}(y) \right) \right\|_{L^2(G)} \right| &\leq \int_0^{2\delta} \left|\left\langle \grad\left( \varphi \circ v^{(n)} \right)(\gamma(t)), \gamma'(t) \right\rangle_{T^1N} \right| \, dt\\
		&\leq \frac C {f\kappa^n} \int_0^{2\delta} \left( \Phi \circ v^{(n)} \right)(\gamma(t)) \, dt\\
		&\leq \frac C {f \kappa^n} {2\delta} e^{\frac C {f\kappa^n} d(x,y)} \Phi\left( v^{(n)}(y) \right)\\
		&\leq \frac C {f} \delta e^{\frac C {f} 2\delta} \Phi\left( v^{(n)}(y) \right)
	\end{align*}
	by the fundamental theorem of calculus and our bounds on $\varphi$. Fix $\delta$ small enough to ensure
	\[
		\frac C {f} 2\delta e^{\frac C f 2\delta} \leq \frac 1 8
	\]
	and
	\[
		e^{\frac C f 2\delta} \leq 2
	\]
	hold simultaneously -- note that we really only require that $C \delta$ is sufficiently small, and so $\delta$ can be chosen inversely proportional to $C$. We therefore have
	\begin{equation}
		\Phi\left( v^{(n)}(y) \right) \leq 2 \Phi\left( v^{(n)}(x) \right)
		\label{lipest2}
	\end{equation}	
	and
	\begin{equation}
		\left\| \varphi\left( v^{(n)}(x) \right) \right\|_{L^2(G)} \leq \left\| \varphi\left( v^{(n)}(y) \right) \right\|_{L^2(G)} + \frac 1 8 \Phi\left( v^{(n)}(y) \right)
		\label{lipest}
	\end{equation}
	whenever $d(x,y) < 2\delta$. To conclude, suppose that we had
	\[
		{\left\| \varphi\left(v^{(n)}(y)\right) \right\|_{L^2(G)}} \leq \frac 1 4{\Phi\left(v^{(n)}(y)\right)}
	\]
	at some $y \in B_{\delta}(u_0)$, for a given $u_0 \in U$ and $n$. Then by \eqref{lipest} and \eqref{lipest2}, we must have
	\begin{align*}
		\left\| \varphi\left( v^{(n)}(x) \right) \right\|_{L^2(G)} &\leq \frac 1 4 \Phi\left( v^{(n)}(y) \right) + \frac 1 8 \Phi\left( v^{(n)}(y) \right)\\
		&\leq \frac 2 4 \Phi\left( v^{(n)}(x) \right) + \frac 2 8 \Phi\left( v^{(n)}(x) \right)\\
		&\leq \frac 3 4 \Phi\left( v^{(n)}(x) \right)
	\end{align*}
	for any $x \in B_{\delta}(u_0)$, as desired.
\end{proof}

\begin{lemma}
	Fix $C > 0$, $\varphi \in C^1(U, V^\rho)$ and $\Phi \in \mathcal K_C$. Take $\delta > 0$ as in Lemma \ref{alternative} and suppose that we have
	\begin{align*}
		\|\varphi(u)\|_{L^2(G)} &< \Phi(u)\\
		\sup_{w\in T^1_uU} \|\left( d\varphi \right)_u(w)\| &< C \Phi(u)
	\end{align*}
	for all $u \in U$. If $v^{(n)}$ is a past defined on $B_\delta(u_0)$ satisfying
	\[
		\left\| \varphi \left( v^{(n)}(u) \right) \right\|_{L^2(G)} \geq \frac 1 4 \Phi\left( v^{(n)}(u) \right)
	\]
	for all $u \in B_\delta(u_0)$, and some given $u_0 \in U$, then we have
	\[
		\sup_{w \in T^1_uB_\delta(u_0)} \left\| \left( d \left( \frac{\varphi \circ v^{(n)}}{\left\|\varphi \circ v^{(n)} \right\|_{\mathrlap{L^2(G)}}} \hphantom{L^2(G}\right) \right)_u(w) \right\|_{L^2(G)} \leq \frac {8C}{f \kappa^n}
	\]
	for all $u \in B_\delta(u_0)$.
	\label{lipschitzest}
\end{lemma}
\begin{proof}
	Differentiating the fraction, we see that we need to bound 
	\begin{equation}
		\sup_{w \in T^1_uB_\delta(u_0)} \left\| \frac{\left(d\left(\varphi \circ v^{(n)}\right)\right)_u\left(w\right) \left\|\varphi\left( v^{(n)}(u) \right)\right\|_{L^2(G)} -  \varphi\left( v^{(n)}(u) \right) \left(d\left( \left\| \varphi \circ v^{(n)} \right\|_{L^2(G)} \right)\right)_u(w)}{\left\|\varphi\left( v^{(n)}(u) \right)\right\|_{L^2(G)} } \right\|_{\mathrlap{L^2(G)}}
		\label{lip1}
	\end{equation}
	for any $u \in U$.
	Note that we always have
	\[
		\hspace{12pt}\left( d\left( \left\| \varphi \circ v^{(n)} \right\|_{L^2(G)} \right) \right)_u(w) \leq \left\| \left( d\left( \varphi \circ v^{(n)} \right) \right)_u(w) \right\|_{L^2(G)}
	\]
	and so \eqref{lip1} is at most
	\begin{equation}
		\sup_{w \in T^1_uB_\delta(u_0)} \frac{2\left\|\varphi\left( v^{(n)}(u) \right)\right\|_{L^2(G)}\left\|\left(d\left(\varphi \circ v^{(n)}\right)\right)_u\left(w\right)\right\|_{\mathrlap{L^2(G)}}}{\left\|\varphi\left( v^{(n)}(u) \right)\right\|^2_{\mathrlap{L^2(G)}}}
		\label{lip2}
	\end{equation}
	by the triangle inequality. Cancelling terms, we can reduce \eqref{lip2} to
	\begin{equation}
		\sup_{w \in T^1_uB_\delta(u_0)} \frac{2\left\|\left(d\left(\varphi \circ v^{(n)}\right)\right)_u\left(w\right)\right\|_{\mathrlap{L^2(G)}}}{\left\|\varphi\left( v^{(n)}(u) \right)\right\|_{\mathrlap{L^2(G)}}}
		\label{lip3}
	\end{equation}
	which is at most
	\begin{equation}
		\sup_{\substack{w \in T^1_uB_\delta(u_0)\\w' \in T^1_{v^{(n)}(u)}B_\delta(u_0)}} \frac{2\left\|\left(d\left(\varphi\right)\right)_{v^{(n)}(u)}\left(w'\right)\right\|_{{L^2(G)}} \left\|\left( dv^{(n)} \right)_u(w)\right\|_{\mathrlap{T^1N}}}{\left\|\varphi\left( v^{(n)}(u) \right)\right\|_{\mathrlap{L^2(G)}}}
		\label{lip4}
	\end{equation}
	by the chain rule. Since $f \kappa^n$ is the slowest expansion rate of any vector in $T^1U$ under $d\sigma^n$, we can bound
	\begin{equation}
		\sup_{w \in T^1_uU} \left\| \left( dv^{(n)} \right)_u(w) \right\|_{T^1N} \leq \frac 1 { f \kappa^n}
	\end{equation}
	for all $u \in U$. By hypothesis, we also have
	\begin{equation}
		\sup_{w' \in T^1_{v^{(n)}(u)}U} \left\|  \left( d\varphi \right)_{v^{(n)}(u)}(w')\right\|_{L^2(G)} \leq C \Phi\left( v^{(n)}(u) \right)
	\end{equation}
	for all $u \in U$, and so \eqref{lip4} can be bounded above by
	\begin{equation}
		\frac {2 C \Phi\left( v^{(n)}(u) \right)}{f\kappa^n \left\| \varphi\left( v^{(n)}(u) \right)\right\|_{\mathrlap{L^2(G)}}}
		\label{lip5}
	\end{equation}
	which is in turn at most
	\begin{equation}
	\frac {8 C \Phi\left( v^{(n)}(u) \right)}{f\kappa^n \Phi\left( v^{(n)}(u) \right)}
	\end{equation}
	by hypothesis. Cancelling $\Phi$, we obtain the result desired.
\end{proof}

Before we get to the main argument in this section, we need an elementary linear algebra result.

\begin{proposition}
	Let $(V, \cdot)$ be an inner product space, and suppose that we have
	\[
		\| \hat v - \hat w \| \, \geq \epsilon
	\]
	where $\hat v, \hat w$ denote the unit vectors in the directions of $v$ and $w$ respectively. If $\|v\| \, \leq \|w\|$, then
	\[
		\| v + w \| \, \leq \left( 1 - \frac {\epsilon^2} 2 \right)\|v\| + \|w\|
	\]
	with equality exactly when $\| \hat v - \hat w \| \, = \epsilon$.
	\label{linearalgebra}
\end{proposition}
\begin{proof}
	We expand using the polarization identity, obtaining
	\begin{align*}
		\|v + w\| &=| \hat v \cdot \hat w | \|v\| + \|w\| \\
		&=\left( \frac{\|\hat v\|^2 + \|\hat w\|^2 - \|\hat v - \hat w\|^2}{2} \right) \|v\|  + \|w\|\\
	&\leq \left( 1 - \frac {\epsilon^2} 2 \right) \|v\| + \|w\|
	\end{align*}
	as desired.
\end{proof}

We might be tempted to argue that, if $\varphi \in C^1(U, V^\rho)$ is controlled by $\Phi \in \mathcal K_C$ as in Lemma \ref{alternative}, then we can similarly bound $\left\| \left(\mathcal L^n_{z,\rho} \varphi\right)(u) \right\|_{L^2(G)}$ by $\left(\mathcal L^n_{\Re(z),0} \Phi\right)(u)$ pointwise -- however, while this turns out to be true, it is unhelpful since $\mathcal L^n_{P(\varsigma),0} \Phi$ fails to contract as $n \to \infty$. Indeed, since $\Phi$ is strictly positive by definition, $\mathcal L^n_{P(\varsigma),0} \Phi$ will converge to $\int_U \Phi \, d\nu^u > 0$.

The solution is to artificially introduce contraction into the transfer operators, and Theorem \ref{nli} is precisely what ensures that we can do this while maintaining a pointwise bound. In the next lemma, we will show that we can uniformly and explicitly bound $\left\| \left( \mathcal L^n_{z,\rho} \varphi \right)(u) \right\|_{L^2(G)}$ away from $\left( \mathcal L^n_{\Re(z),0} \Phi \right)(u)$ on a measurable portion of any sufficiently small set. 
\begin{lemma}
	\label{betalemma}
	Fix $C > 0$, $\varphi \in C^1(U, V^\rho)$ and $\Phi \in \mathcal K_C$ with
	\begin{align*}
		\left\| \varphi(u) \right\|_{L^2(G)} &< \Phi(u)\\
		\sup_{w \in T_u^1U} \left\| \left( d\varphi \right)_u(w) \right\|_{L^2(G)} &< C \Phi(u)
	\end{align*}
	for all $u \in U$. Let $U_{lni} \subset U$ be the open subset given by Theorem \ref{nli} and let $\delta > 0$ be the constant given by Lemma \ref{alternative}.
	There are constants $n_0 > 0$, $\epsilon > 0$ and $s < 1$ so that, for any $x \in U_{lni}$ with $B_\delta(x) \subset U_{lni}$, we can find
	\begin{itemize}
		\item a point $y \in U_{lni}$ with $B_{s\delta}(y) \subset B_\delta(x)$ and
		\item a pair of pasts $v_1^{(n_0)}, v_2^{(n_0)}$ defined on $B_\delta(x)$
	\end{itemize}
	for which we can bound 
	\[
		\left\|e^{\alpha_z^{(n_0)}\left(v_1^{(n_0)}(u)\right)}\hspace{-1.5pt}\rho\left( \Hol^{(n_0)}\hspace{-2pt}\left( v_1^{(n_0)}(u) \right) \right)\hspace{-1.5pt} \varphi\left( v_1^{(n_0)}(u) \right)\hspace{-1pt} + \hspace{-1pt} e^{\alpha_z^{(n_0)}\left(v_2^{(n_0)}(u)\right)}\hspace{-1.5pt} \rho\left( \Hol^{(n_0)} \hspace{-2pt}\left( v_2^{(n_0)}(u) \right) \right)\hspace{-1.5pt} \varphi\left( v_2^{(n_0)}(u) \right) \right\|_{L^2(G)}
	\]
	above by
	\[
		\left( 1 - \frac {\left( \epsilon \delta (1+|\Im(z)|) \|\rho\| \right)^2}{2048} \right) e^{\alpha_{\Re(z)}^{(n_0)}\left(v_1^{(n_0)}(u)\right)} \Phi\left( v_1^{(n_0)}(u) \right) + e^{\alpha_{\Re(z)}^{(n_0)}\left(v_2^{(n_0)}(u)\right)} \Phi\left( v_2^{(n_0)}(u) \right) 
	\]
	for all $u \in B_{s\delta}(y)$ and all $z \in \bb C$ with $\left|\Re(z) - P(\varsigma) \right| < 1$. The constant $\epsilon$ can be chosen independently of $\rho$ and $C$, while $n_0$ depends only on $\frac C {\|\rho\|}$. The constant $s$ can be chosen uniformly in $C$ and $\rho$. 
\end{lemma}
\begin{proof} 
	We will deal with a fixed $n > 0$ throughout the proof, and specify how large $n$ needs to be as we proceed -- it is important to note that we cannot deal with arbitrarily large $n$ without foregoing the uniformity of the bounds we wish to obtain. We proceed in cases depending on which alternative of Lemma \ref{alternative} holds. For any pasts $v_1^{(n)}$ and $v_2^{(n)}$, if we have
	\begin{equation}
		\left\|\varphi\left(v_{1}^{(n)}(u)\right)\right\|_{L^2(G)} \leq \frac 3 4 \Phi\left( v_{1}^{(n)}(u) \right)
		\label{firstalt}
	\end{equation}
	for all $u \in B_\delta(x)$, then since $\rho$ is unitary we can clearly bound
	\[
		\left\|e^{\alpha_z^{(n)}\left(v_1^{(n)}(u)\right)}\rho\left( \Hol^{(n)}\left( v_1^{(n)}(u) \right) \right) \varphi\left( v_1^{(n)}(u) \right) + e^{\alpha_z^{(n)}\left(v_2^{(n)}(u)\right)} \rho\left( \Hol^{(n)}\left( v_2^{(n)}(u) \right) \right) \varphi\left( v_2^{(n)}(u) \right) \right\|_{\mathrlap{L^2(G)}}
	\]
	above by
	\begin{equation}
		\frac 3 4 \left(e^{\alpha_{\Re(z)}^{(n)}\left(v_1^{(n)}(u)\right)} \Phi\left( v_1^{(n)}(u) \right) + e^{\alpha_{\Re(z)}^{(n)}\left(v_2^{(n)}(u)\right)} \Phi\left( v_2^{(n)}(u) \right) \right)
		\label{alt1}
	\end{equation}
	for all $u \in B_\delta(x)$. In this case, we are done by simply setting $y \coloneqq x$. Similarly, if we had 
	\begin{equation}
		\left\|\varphi\left(v_{2}^{(n)}(u)\right)\right\|_{L^2(G)} \leq \frac 3 4 \Phi\left( v_{2}^{(n)}(u) \right)
	\end{equation}
	for all $u \in B_\delta(x)$, then we are again done by setting $y \coloneqq x$, up to interchanging our choice of $v_1$ and $v_2$. So we may as well assume that the second alternative of Lemma \ref{alternative} holds for both $v_1^{(n)}$ and $v_2^{(n)}$, and that we therefore have
	\begin{equation}
		\left\|\varphi\left(v_{\ell}^{(n)}(u)\right)\right\|_{L^2(G)} \geq \frac 1 4 \Phi\left( v_{\ell}^{(n)}(u) \right)
		\label{secondalt}
	\end{equation}
	for all $u \in B_\delta(x)$. We will temporarily abbreviate
		\begin{align*}
			g_\ell(u) &\coloneqq \Hol^{(n)}\left( v_\ell^{(n)}(u) \right)\\
			\hat \varphi_\ell(u) &\coloneqq \frac{\varphi \left( v_\ell^{(n)}(u) \right)}{\left\| \varphi \left( v_\ell^{(n)}(u) \right)\right\|_{\mathrlap{L^2(G)}}}
			\intertext{and}
			\hat \psi_\ell (u) &\coloneqq e^{\Im(z)\tau^{(n)}\left( v_\ell^{(n)}(u) \right) i }\frac {\varphi\left( v_\ell^{(n)}(u) \right)}{\left\| \varphi \left(  v_\ell^{(n)}(u) \right)\right\|_{\mathrlap{L^2(G)}}}
		\end{align*}
	for the sake of clarity. Note that $\hat\varphi_\ell$ and $\hat \psi_\ell$ are well-defined on $B_\delta(x)$ as an immediate consequence of $\eqref{secondalt}$. Now, by reverse the triangle inequality,
	\begin{equation}
		\left\| \rho\left( g_1(x) \right) \hat \psi_1(x) - \rho\left( g_2(x) \right) \hat \psi_2(x)\right\|_{\mathrlap{L^2(G)}}
		\label{firsteq}
	\end{equation}
	is at least
	\begin{equation}
	\begin{gathered}
		\left\| \rho\left( g_{1}(x) \right) \hat\psi_{1}\left( y \right) - \rho\left( g_{2}(x) \right)\hat\psi_{2}\left( y \right)\right\|_{\mathrlap{L^2(G)}}\\
		- \left\| \rho\left( g_{1}(x) \right) \hat\psi_{1}\left( y \right) - \rho\left( g_{1}(x) \right)\hat\psi_{1}\left( x \right)\right\|_{L^2(G)}	- \left\| \rho\left( g_{2}(x) \right) \hat\psi_{2}\left( x \right) - \rho\left( g_{2}(x) \right)\hat\psi_{2}\left( y \right)\right\|_{L^2(G)}
		\label{nlistep1}
	\end{gathered}
\end{equation}
	for any $y \in B_\delta(x)$. Since group elements act by isometries with respect to the $L^2(G)$ norm, \eqref{nlistep1} is equal to
	\begin{gather*}
		\left\| \rho\left( g_2(y)g_2^{-1}(x)g_{1}(x) \right) \hat\psi_{1}\left( y \right) - \rho\left( g_2(y) \right)\hat\psi_{2}\left( y \right)\right\|_{\mathrlap{L^2(G)}}\\	- \left\| \hat\psi_{1}\left( y \right) - \hat\psi_{1}\left( x \right)\right\|_{L^2(G)}	- \left\|  \hat\psi_{2}\left( x \right) - \hat\psi_{2}\left( y \right)\right\|_{\mathrlap{L^2(G)}}
	\end{gather*}
	which we can bound below by
	\begin{equation}
	\begin{gathered}
		\left\| \rho\left( g_2(y)g_2^{-1}(x)g_{1}(x) \right) \hat\psi_{1}\left( y \right) - \rho\left( g_1(y) \right)\hat\psi_{1}\left( y \right)\right\|_{\mathrlap{L^2(G)}}\\	
		- \left\| \hat\psi_{1}\left( y \right) - \hat\psi_{1}\left( x \right)\right\|_{L^2(G)}	- \left\|  \hat\psi_{2}\left( x \right) - \hat\psi_{2}\left( y \right)\right\|_{\mathrlap{L^2(G)}}\\
			- \left\|\rho\left( g_2(y) \right)\hat\psi_2(y) - \rho\left( g_1(y) \right)\hat\psi_1(y)\right\|_{\mathrlap{L^2(G)}}
	\end{gathered}
	\label{nlistep5}
\end{equation}
using the reverse triangle inequality once again. We can rewrite \eqref{nlistep5} as
	\begin{equation}
	\begin{gathered}
		\left\| \rho\left( g_{1}(x)g_1^{-1}(y)g_1(y) \right) \hat\psi_{1}\left( y \right) - \rho\left( g_2(x)g_2^{-1}(y)g_1(y) \right)\hat\psi_{1}\left( y \right)\right\|_{\mathrlap{L^2(G)}}\\	
		- \left\| \hat\psi_{1}\left( y \right) - \hat\psi_{1}\left( x \right)\right\|_{L^2(G)}	- \left\|  \hat\psi_{2}\left( x \right) - \hat\psi_{2}\left( y \right)\right\|_{\mathrlap{L^2(G)}}\\
			- \left\|\rho\left( g_2(y) \right)\hat\psi_2(y) - \rho\left( g_1(y) \right)\hat\psi_1(y)\right\|_{\mathrlap{L^2(G)}}
			\label{nlistep2}
	\end{gathered}
	\end{equation}
	by replacing $\hat \psi_1(y)$ with $g_1^{-1}(y)g_1(y)\hat \psi_1(y)$ in the first term of the first line, and multiplying both terms on the first line by $g_2(x)g_2^{-1}(y)$.
	Hence, \eqref{firsteq} is bounded below by \eqref{nlistep2}, and we see that
	\begin{equation}
		 \left\| \rho\left( g_1(x) \right) \hat \psi_1(x) - \rho\left( g_2(x) \right) \hat \psi_2(x)\right\|_{L^2(G)} + \left\|\rho\left( g_2(y) \right)\hat\psi_2(y) - \rho\left( g_1(y) \right)\hat\psi_1(y)\right\|_{\mathrlap{L^2(G)}}
		\label{nlistep6}
	\end{equation}
	is at least
	\begin{equation}
		\begin{gathered}
			\left\| \rho\left( g_{1}(x)g_1^{-1}(y)g_1(y) \right) \hat\psi_{1}\left( y \right) - \rho\left( g_2(x)g_2^{-1}(y)g_1(y) \right)\hat\psi_{1}\left( y \right)\right\|_{\mathrlap{L^2(G)}}\\	
		- \left\| \hat\psi_{1}\left( y \right) - \hat\psi_{1}\left( x \right)\right\|_{L^2(G)}	- \left\|  \hat\psi_{2}\left( x \right) - \hat\psi_{2}\left( y \right)\right\|_{\mathrlap{L^2(G)}} \\ 
		\end{gathered}
		\label{nlistep7}
	\end{equation}
	after rearranging.
	By Theorem \ref{nli}, there is an $N > 0$, $\epsilon > 0$, a choice of pasts $v_1^{(n)}$ and $v_2^{(n)}$, and a $y$ with $d(x,y) = \frac \delta 2$ for which we have
	\begin{equation}
		\label{mainest}
		\left\| \rho\left( g_{1}(x)g_1^{-1}(y)g_1(y) \right) \hat\psi_{1}\left( y \right) - \rho\left( g_2(x)g_2^{-1}(y)g_1(y) \right)\hat\psi_{1}\left( y \right)\right\|_{L^2(G)} \geq \epsilon (1+|\Im(z)|) \|\rho\| \frac \delta 2
	\end{equation}
	for all $n \geq N$ -- note that we have applied the theorem to
	\[
		e^{\Im(z) \tau^{(n)}(u) i}\rho\left( \Hol^{(n)}(u) \right)  \frac{\varphi(u)}{\|\varphi(u)\|_{\mathrlap{L^2(G)}}}
	\]
	which is certainly a smooth function in $C^1(U, V^\rho)$. On the other hand, by Lemma \ref{lipschitzest}, $\hat \varphi_\ell(u)$ is at worst $\frac {8C}{f\kappa^n}$-Lipschitz in $u$ on $B_\delta(x)$, and we can estimate
	\begin{align*}
		\left|\left(d \left(e^{\Im(z)\left(\tau^{(n)} \circ v^{(n)}_\ell\right)i}\right)\right)_u(w) \right|_{L^2(G)}&= \left(\left|\Im(z) e^{\Im(z) \left( \tau^{(n)} \circ v^{(n)}_\ell \right)i} \right|\right)\left|\sum_{i=1}^n \left( d\left( \tau \circ v_\ell^i \right) \right)_{v_\ell^{(n)}(u)}(w') \right| \\
		&\leq |\Im(z)|\left(\sum_{i=1}^n \frac{\|\tau\|_{C^1}}{f\kappa^i}\right)\\
		&\leq |\Im(z)|\frac {\|\tau\|_{\mathrlap{C^1}}}{f(\kappa - 1)}
	\end{align*}
for all $u \in U$, $w \in T^1_uB_\delta(x)$ and $w' \in T^1_{v_\ell(u)}B_\delta(x)$. Hence, we can bound
	\[
		\left\| \hat \psi_\ell(x) - \hat \psi_\ell(y) \right\| \leq \left(\frac {8C}{f\kappa^n} + |\Im(z)| \frac{\|\tau\|_{C^1}}{f(\kappa - 1)}\right)d(x,y)
	\]
	since we chose $y \in B_\delta(x)$. Suppose that $n$ and $\|\rho\|$ are large enough so that we have
\begin{equation}
	\left( \frac {8 C}{f \kappa^n} + |\Im(z)| \frac{\|\tau\|_{C^1}}{f(\kappa -1)} \right) < \frac {\epsilon} 8 (1+|\Im(z)|) \|\rho\|
	\label{largen}
\end{equation}
for some constant $K > 0$. Note that we can make this choice of $n$ so that it depends only on $\frac C {\|\rho\|}$ and not $\|\rho\|$ directly; moreover, the requirement that $\|\rho\|$ be sufficiently large can be made absolute, and in particular is independent of $z$. This ensures
	\begin{equation}
		\left\| \hat \psi_\ell(x) - \hat \psi_\ell(y)\right\|_{L^2(G)} \leq \frac \epsilon 8 \left( 1 + |\Im(z)| \right)\| \rho \| \delta
		\label{lipest5}
	\end{equation}
	since $d(x,y) \leq \delta$. Note that our choice of $n$ here depends only on $\frac C {\|\rho\|}$. Plugging \eqref{mainest} and \eqref{lipest5} into \eqref{nlistep7}, we conclude that 
\begin{gather}
	\left\|\rho(g_1(x))\hat \psi_1(x) - \rho\left( g_2(x) \right)\hat \psi_2(x) \right\|_{L^2(G)} + \left\|\rho(g_1(y)) \hat\psi_1(y) - \rho\left( g_2(y) \right)\hat \psi_2(y)\right\|_{L^2(G)} \geq \frac \epsilon 4 \delta (1+|\Im(z)|) \|\rho\|
\end{gather}
and so 
\begin{equation}
	\left\|\rho(g_1(x))\hat \psi_1(x) - \rho\left( g_2(x) \right)\hat \psi_2(x) \right\|_{L^2(G)} > \frac \epsilon 8 \delta (1+|\Im(z)|) \|\rho\|
\label{option1}
\end{equation}
or
\begin{equation}
	\left\|\rho(g_1(y))\hat \psi_1(y) - \rho\left( g_2(y) \right)\hat \psi_2(y) \right\|_{L^2(G)} > \frac \epsilon 8 \delta (1+|\Im(z)|) \|\rho\|
	\label{option2}
\end{equation}
must hold. Without loss of generality, suppose \eqref{option2} holds. Using the Lipschitz estimate on $\hat \varphi_\ell$, we can bound
\begin{align}
	\sup_{w \in T^1_uB_\delta(x)} \left\| \left( d\left( \rho \left( g_\ell \right) \hat \psi_\ell \right) \right)_u(w) \right\|_{L^2(G)}
\end{align}
by
\begin{align}
	\sup_{w \in T^1_uB_\delta(x)} \left\| \left(\left( (d\rho)_{g_\ell(u)} \circ \left( dg_\ell \right)_u \right)(w)\right) \hat \psi_\ell(u) + \rho\left( g_\ell(u) \right) \left( d \hat \psi_\ell \right)_u(w)\right\|_{L^2(G)} 
	\label{lipprod}
\end{align}
for all $u \in B_\delta(x)$. A straightforward calculation shows that
\begin{align*}
	\sup_{w \in T^1_uB_\delta(x)} \left\| \left( dg_\ell \right)_u(w) \right\| &\leq \sup_{w \in T^1_uB_\delta(x)} \sum_{i=1}^n \left\| \left( d\left( \Hol \circ v^{(i)} \right) \right)_u(w) \right\|_{T^1G}\\
	&\leq \sum_{i=1}^n \frac{\|\Hol\|_{\mathrlap{C^1}}}{f\kappa^i}\\
	&\leq \frac{\|\Hol\|_{\mathrlap{C^1}}}{f\left( \kappa - 1 \right)}
\end{align*}
and we have
\[
	\sup_{w' \in T^1_{g_\ell(u)}G} \left\| \left( d\rho \right)_{g_\ell(u)}(w') \right\|_{L^2(G)} \leq \|\rho\|
\]
by Definition \ref{derivativenorm}; note that since $\rho$ is a homomorphism, the operator norm of $d\rho$ at $g_\ell(u)$ is equivalent to the norm at the identity. Hence, \eqref{lipprod} is at most
\[
	\|\rho\| \frac{\|\Hol\|_{\mathrlap{C^1}}}{f\left( \kappa - 1 \right)} + \frac \epsilon 8 (1+|\Im(z)|) \|\rho\|
\]
by our choice of $\delta$, since $\rho$ acts by $L^2(G)$-isometries. We can make a uniform choice of $s < 1$ so that
\begin{equation}
	\label{sineq}
	s \left(\|\rho\| \frac{\|\Hol\|_{\mathrlap{C^1}}}{f\left( \kappa - 1 \right)} + \frac \epsilon 8 (1+|\Im(z)|) \|\rho\|
 \right) < \frac \epsilon {32} (1+|\Im(z)|) \|\rho\|
\end{equation}
and we then have
\begin{equation}
	\left\| \rho\left( g_1(y) \right) \hat \psi_1(y) - \rho \left( g_1(u) \right) \hat \psi_1(u) \right\|_{L^2(G)} < \frac \epsilon {32} \delta (1+|\Im(z)|) \|\rho\|
	\label{lipest1}
\end{equation}
and
\begin{equation}
	\left\| \rho\left( g_2(y) \right) \hat \psi_2(y) - \rho \left( g_2(u) \right) \hat \psi_2(u) \right\|_{L^2(G)} < \frac \epsilon {32} \delta (1+|\Im(z)|) \|\rho\|
	\label{lipest3}
\end{equation}
for all $u \in B_{s\delta}(y)$. Combining \eqref{lipest1} and \eqref{lipest3} with \eqref{option1}, we now have
\begin{equation}
	\left\| \rho\left( g_1(u) \right)\hat \psi_1(u) - \rho \left( g_2(u) \right) \hat \psi_2(u) \right\|_{L^2(G)} > \frac \epsilon {16} \delta (1+|\Im(z)|) \|\rho\|
\end{equation}
for all $u \in B_{s\delta}(y)$. Now, fix $u \in B_{s\delta}(y)$. By Proposition \ref{linearalgebra}, we can then bound
\begin{equation}
		\left\|e^{\alpha_z^{(n)}\left(v_1^{(n)}(u)\right)}\rho\left( \Hol^{(n)}\left( v_1^{(n)}(u) \right) \right) \varphi\left( v_1^{(n)}(u) \right) + e^{\alpha_z^{(n)}\left(v_2^{(n)}(u)\right)} \rho\left( \Hol^{(n)}\left( v_2^{(n)}(u) \right) \right) \varphi\left( v_2^{(n)}(u) \right) \right\|_{\mathrlap{L^2(G)}}
	\end{equation}
	above by
	\begin{equation}
		\left( 1 - \frac {(\epsilon \delta (1+|\Im(z)|) \|\rho\|
\|)^2}{512} \right) e^{\alpha_{\Re(z)}^{(n)}\left(v_1^{(n)}(u)\right)} \Phi\left( v_1^{(n)}(u) \right) + e^{\alpha_{\Re(z)}^{(n)}\left(v_2^{(n)}(u)\right)} \Phi\left( v_2^{(n)}(u) \right) 
		\label{firstoption2}
	\end{equation}
	assuming without loss of generality that
	\begin{equation}
		\label{linalgineq}
		\left\|e^{\alpha_{\Re(z)}^{(n)}\left( v_1^{(n)}(u) \right)} \varphi\left( v_1^{(n)}(u) \right)\right\|_{L^2(G)} \leq \left\|e^{\alpha_{\Re(z)}^{(n)}\left( v_2^{(n)}(u) \right)} \varphi\left( v_2^{(n)}(u) \right)\right\|_{L^2(G)}
	\end{equation}
	held for this particular $u$ -- this is true up to interchanging $v_1^{(n)}$ and $v_2^{(n)}$. It simply remains to ensure that a similar inequality extends to $B_{s\delta}(y)$ for this particular choice of $v_1^{(n)}$ and $v_2^{(n)}$ -- this is not immediate since \eqref{linalgineq} could certainly fail to hold on the entire ball $B_{s\delta}(y)$. 
	This will take some extra work. Note that we have
	\begin{align*}
		\sup_{w \in T^1_uU} \left\| \left(d \left(e^{\alpha_{\Re(z)}^{(n)}} \Phi \right) \right)_u(w)  \right\|_{L^2(G)} \hspace{-2pt} &\leq \hspace{-2pt} \sup_{w \in T^1_uU} \left|\left(de^{\alpha_{\Re(z)}^{(n)}}\right)_u(w) \right| \left\| \Phi(u) \right\|_{L^2(G)} \hspace{-1pt} + e^{\alpha_{\Re(z)}^{(n)}(u)}\hspace{-1pt}\left\| \left( d \Phi  \right)_u(w) \right\|_{L^2(G)} \\
		&\leq \left\|\alpha_{\Re(z)}^{(n)} \right\|_{C^1} e^{\alpha_{\Re(z)}^{(n)}(u)} \left\| \Phi(u) \right\|_{L^2(G)} + Ce^{\alpha_{\Re(z)}^{(n)}(u)}\Phi(u)\\
		&\leq \left(\sum_{i=1}^{n} bK^i \left\|\alpha_{\Re(z)}\right\|_{C^1}\right)e^{\alpha_{\Re(z)}^{(n)}(u)} \left\| \Phi(u) \right\|_{L^2(G)} + Ce^{\alpha_{\Re(z)}^{(n)}(u)}\Phi(u)\\
		&\leq \left( b \left\|\alpha_{\Re(z)}\right\|_{C^1} \frac{K^{n+1} - K}{K - 1} + C \right) e^{\alpha_{\Re(z)}^{(n)}(u)} \|\Phi(u)\|_{L^2(G)}
	\end{align*}
	for all $u \in U$. By making $s$ smaller if necessary, we can ensure that, in addition to \eqref{sineq}, we also have
	\[
		s\delta\left( b \left\|\alpha_{\Re(z)}\right\|_{C^1} \frac{K^{n+1} - K}{K - 1} + C \right) < \delta C = A
	\]
	where $A$ is the uniform constant guaranteed by Lemma \ref{alternative}.  Note that this can be accomplished by a uniform choice of $s$, since $n$ is fixed and $z$ is required to satisfy $\left| \Re(z) - P(\varsigma) \right| < 1$. As a consequence, we see by Lemma \ref{alternative} that
	\begin{equation}
		\frac 1 2 e^{\alpha^{(n)}_{\Re(z)}\left( v_\ell^{(n)}(u') \right)} \Phi \left( v_\ell^{(n)}(u') \right)
\leq e^{\alpha^{(n)}_{\Re(z)}\left( v_\ell^{(n)}(u) \right)} \Phi \left( v_\ell^{(n)}(u) \right) \leq 2e^{\alpha^{(n)}_{\Re(z)}\left( v_\ell^{(n)}(u') \right)} \Phi \left( v_\ell^{(n)}(u') \right)
\label{alternative1}
	\end{equation}
	for all $u, u' \in B_{s\delta}(y)$. Now, suppose that
	\begin{equation}
		\left\|e^{\alpha_z^{(n)}\left(v_1^{(n)}(u')\right)}\rho\left( \Hol^{(n)}\left( v_1^{(n)}(u') \right) \right) \varphi\left( v_1^{(n)}(u') \right) + e^{\alpha_z^{(n)}\left(v_2^{(n)}(u')\right)} \rho\left( \Hol^{(n)}\left( v_2^{(n)}(u') \right) \right) \varphi\left( v_2^{(n)}(u') \right) \right\|_{\mathrlap{L^2(G)}}
		\label{secondoption1}
	\end{equation}
	were bounded above by 
	\begin{equation}
		e^{\alpha_{\Re(z)}^{(n)}\left(v_1^{(n)}(u')\right)} \Phi\left( v_1^{(n)}(u') \right) + \left( 1 - \frac {(\epsilon \delta (1+|\Im(z)|) \|\rho\|
)^2}{512} \right) e^{\alpha_{\Re(z)}^{(n)}\left(v_2^{(n)}(u')\right)} \Phi\left( v_2^{(n)}(u') \right) 
		\label{secondoption2}
	\end{equation}
	for some $u' \in B_{s\delta}(y)$ -- this is in contrast to the bound by \eqref{firstoption2} that we have at $u \in B_{s\delta}(y)$. If we had
	\[
		e^{\alpha_{\Re(z)}^{(n)}\left(v_2^{(n)}(u')\right)} \Phi\left( v_2^{(n)}(u') \right) \leq e^{\alpha_{\Re(z)}^{(n)}\left(v_1^{(n)}(u')\right)} \Phi\left( v_1^{(n)}(u') \right) 
	\]
	then \eqref{secondoption2} is in turn bounded above by
	\[
		\left( 1 - \frac {(\epsilon \delta (1+|\Im(z)|) \|\rho\|
)^2}{512} \right)e^{\alpha_{\Re(z)}^{(n)}\left(v_1^{(n)}(u')\right)} \Phi\left( v_1^{(n)}(u') \right) + e^{\alpha_{\Re(z)}^{(n)}\left(v_2^{(n)}(u')\right)} \Phi\left( v_2^{(n)}(u') \right) 
	\]
	which is consistent with \eqref{firstoption2}. If, on the other hand, we had
	\[
		e^{\alpha_{\Re(z)}^{(n)}\left(v_1^{(n)}(u')\right)} \Phi\left( v_1^{(n)}(u') \right) \leq e^{\alpha_{\Re(z)}^{(n)}\left(v_2^{(n)}(u')\right)} \Phi\left( v_2^{(n)}(u') \right) 
	\]
	then we can invoke \eqref{alternative1} twice to see that
	\begin{align*}
		e^{\alpha_{\Re(z)}^{(n)}\left(v_1^{(n)}(u)\right)} \Phi\left( v_1^{(n)}(u) \right) &\leq 2 e^{\alpha_{\Re(z)}^{(n)}\left(v_1^{(n)}(u')\right)} \Phi\left( v_1^{(n)}(u') \right)\\
		&\leq 2e^{\alpha_{\Re(z)}^{(n)}\left(v_2^{(n)}(u')\right)} \Phi\left( v_2^{(n)}(u') \right)\\
		&\leq 4e^{\alpha_{\Re(z)}^{(n)}\left(v_2^{(n)}(u)\right)} \Phi\left( v_2^{(n)}(u) \right)
	\end{align*}
	for \textit{any} other $u \in B_{s\delta}(y)$. Hence, \eqref{firstoption2} is at most
	\begin{equation}
		e^{\alpha_{\Re(z)}^{(n)}\left(v_1^{(n)}(u)\right)} \Phi\left( v_1^{(n)}(u) \right) + \left( 1 - \frac {(\epsilon \delta (1+|\Im(z)|) \|\rho\|
)^2}{2048} \right) e^{\alpha_{\Re(z)}^{(n)}\left(v_2^{(n)}(u)\right)} \Phi\left( v_2^{(n)}(u) \right) 
	\end{equation}
and so we can make a consistent choice of $v_1^{(n)}, v_2^{(n)}$ on the entire ball $B_{s\delta}(y)$. Note that if \eqref{option1} held instead, everything that followed would have been identical, with $B_{s\delta}(x)$ instead of $B_{s\delta}(y)$.
\end{proof}

The point of Lemma \ref{betalemma} is that, in any ball $B_\delta(x)$ of radius $\delta$, we can always find a uniformly smaller ball $B_{s\delta}(y) \subset B_\delta(x)$ on which $\mathcal L^n_{z,\rho} \varphi$ is \textit{strictly} and \textit{uniformly} bounded away from $\mathcal L^n_{\Re(z),0}\left( \Phi \right)$. This means that we can `bump' $\Phi$ down on any such ball $B_{\frac \delta 4}(y)$ without affecting our inequality. Moreover, using the diametric regularity of the measure $\nu^u$, we can ensure that we are able to do this on a set of uniformly large measure.

\begin{lemma}
	\label{mainlemma}
	Fix $C > 0$, $\varphi \in C^1(U, V^\rho)$ and $\Phi \in \mathcal K_{C}$. If we have
	\begin{align*}
		\left\| \varphi(u) \right\|_{L^2(G)} &< \Phi(u)\\
		\sup_{w \in T_u^1U} \left\| \left( d\varphi \right)_u(w) \right\|_{L^2(G)} &< C \Phi(u)
	\end{align*}
	for all $u \in U$, let $\delta > 0$ be the constant guaranteed by Lemma \ref{alternative} and let $n_0 > 0$, $\epsilon > 0$ and $s < 1$ be the constants guaranteed by Lemma \ref{betalemma}. For a given $\rho$ and $z \in \bb C$ with $|\Re(z) - P(\varsigma)| < 1$, we can find a function $\beta \in C^1(U, \bb R)$ for which we have
	\[
		\left\|\left(\mathcal L^{n_0}_{z,\rho} \varphi \right)(u)\right\|_{L^2(G)} \leq \left(\mathcal L^{n_0}_{\Re(z),0} \left(\beta\Phi\right)\right) (u)
	\]
	for all $u \in U$, as well as
	\[
		\left\| \mathcal L^{n_0}_{P(\varsigma),0} \left( \beta \Phi \right) \right\|_{L^2(\nu^u)} \leq \left( 1 - r \frac{(\epsilon \delta (1+|\Im(z)|) \|\rho\|
)^2}{2048} \nu^u\left( U_{lni} \right) \right) \left\| \Phi \right\|_{L^2(\nu^u)}
	\]
	for some uniform constant $r < 1$. The constant $r$ does not depend on $\rho$, $z$, $\varphi$ or $\Phi$, while the function $\beta$ may depend on all of these.
\end{lemma}
\begin{proof}
	By the Vitali covering lemma, we can find a finite collection of points $x_1, \ldots, x_k \in U_{lni}$ so that the balls $B_\delta(x_1), \ldots, B_\delta(x_k) \subset U_{lni}$ of radius $\delta$ are pairwise disjoint, while the balls $B_{3\delta}(x_1), \ldots, B_{3\delta}(x_k)$ of radius $3\delta$ cover $U_{lni}$. 
	By Lemma \ref{betalemma}, for each $i$ we can find a ball $B_{s\delta}(y_i) \subset B_{\delta} (x_i)$ and pasts $v_{1,i}^{(n_0)}, v_{2,i}^{(n_0)}$ so that 
	\begin{equation}
		\left\|e^{\alpha_z^{(n_0)}\left(v_{1,i}^{(n_0)}(u)\right)}\hspace{-1.5pt}\rho\left( \Hol^{(n_0)}\hspace{-2pt}\left( v_{1,i}^{(n_0)}(u) \right) \right)\hspace{-1.5pt} \varphi\left( v_{1,i}^{(n_0)}(u) \right)\hspace{-1pt} + \hspace{-1pt} e^{\alpha_z^{(n_0)}\left(v_{2,i}^{(n_0)}(u)\right)}\hspace{-1.5pt} \rho\left( \Hol^{(n_0)} \hspace{-2pt}\left( v_{2,i}^{(n_0)}(u) \right) \right)\hspace{-1.5pt} \varphi\left( v_{2,i}^{(n_0)}(u) \right) \right\|_{L^2(G)}
	\end{equation}
	is bounded above by
	\begin{equation}
		\left( 1 - \frac {(\epsilon \delta (1+|\Im(z)|) \|\rho\|
)^2} {2048} \right)e^{\alpha_{\Re(z)}^{(n_0)}\left(v_{1,i}^{(n_0)}(u)\right)} \Phi\left( v_{1,i}^{(n_0)}(u) \right) + e^{\alpha_{\Re(z)}^{(n_0)}\left(v_{2,i}^{(n_0)}(u)\right)} \Phi\left( v_{2,i}^{(n_0)}(u) \right) 
		\label{bound1}
	\end{equation}
	for all $u \in B_{s\delta}(y_i)$. For each $i$, define a $C^1$ radially-decreasing bump function $\eta_i$ centered at $v_{1,i}^{(n_0)}(y_i)$ by
	\[
		\eta_i\left(v_1^{(n_0)}(u)\right) = \begin{cases}
			\frac {(\epsilon \delta (1+|\Im(z)|) \|\rho\|
)^2} {2048} &\textrm{ if } d(u, y_i) \leq \frac {s\delta} 2\\
			\frac {(\epsilon \delta (1+|\Im(z)|) \|\rho\|
)^2}{2048} \exp\left(1 + \frac{1}{\left( \frac 1 {s\delta} d(u,y_i) - \frac 1 2 \right)^2 - 1}\right) &\textrm{ if } \frac {s\delta} 2 < d(u,y_i) < s\delta\\
			0 &\textrm{ if } s\delta \leq d(u,y_i)
		\end{cases}
	\]
	for all $u \in B_{\delta}(x_i)$. We can smoothly extend $\eta_i$ to all of $U$ by setting $\eta_i = 0$ outside $v_1^{(n_0)}(u)\left( B_{\delta}(x_i) \right)$. To define $\beta$, we simply set
	\[
		\beta(u) \coloneqq 1 - \sum_i \eta_i(u)
	\]
	and by Lemma \ref{betalemma} we clearly have 
	\[
		\left\| \left( \mathcal L^{n_0}_{z,\rho} \varphi \right)(u) \right\|_{L^2(G)} \leq \left( \mathcal L^{n_0}_{\Re(z),0} \left( \beta \Phi \right)  \right)(u)
	\]
	for all $u \in U$. It simply remains to estimate $\left\|\mathcal L^{n_0}_{P(\varsigma),0} \left( \beta \Phi \right) \right\|_{L^2(\nu^u)}$. Note that we have
	\[
		\left(\mathcal L^{n_0}_{P(\varsigma),0} \left( \beta \Phi \right)\right) (u) \leq \left( 1 - \frac {\left( \epsilon \delta (1+|\Im(z)|) \|\rho\|
 \right)^2} {2048} \right) \left(\mathcal L^{n_0}_{P(\varsigma),0} \Phi \right)(u)
	\]
	for all $u \in \bigcup B_{\frac {s\delta} 2}(y_i)$ by construction. Moreover, by the monotonicity of $\mathcal L^{n_0}_{P(\varsigma),0}$, we can bound
	\[
		\left(\mathcal L^{n_0}_{P(\varsigma),0} \left( \beta \Phi \right)\right) (u) \leq \left(\mathcal L^{n_0}_{P(\varsigma),0} \Phi \right)(u)
	\]
	for all $u \in U - \bigcup B_{\frac {s\delta}2}(y_i)$. Taken together, these inequalities mean that we can bound
	\begin{align*}
		\left\|\left(\mathcal L^{n_0}_{P(\varsigma),0} \left( \beta \Phi \right)\right)   \right\|_{L^2(\nu^u)}
	\end{align*}
		{from above by}
		\begin{align*}
			\left( \int_{\bigcup B_{\frac {s\delta} 2}(y_i)} \hspace{-3pt}\left( 1 - \frac {\left( \epsilon \delta (1+|\Im(z)|) \|\rho\| \right)^2} {2048} \right)^2 \hspace{-4pt} \left(\left(\mathcal L^{n_0}_{P(\varsigma),0} \Phi \right)\hspace{-1.5pt}(u)\right)^2 d\nu^u  \right)^{\frac 1 2} \hspace{-7pt} +  \left(\int_{U - \bigcup B_{\frac {s\delta} 2}(y_i)} \hspace{-3pt}\left(\left(\mathcal L^{n_0}_{P(\varsigma),0} \Phi \right)\hspace{-1.5pt}(u)\right)^2 d\nu^u \right)^{\frac 1 2}
	\end{align*}
	for our particular choice of $\rho$ and $z$.
	Since the operator $\mathcal L^{n_0}_{P(\varsigma),0}$ preserves the measure $\nu^u$ (by our renormalization), this is exactly
	\begin{align*}
	\left\| \Phi \right\|_{L^2(\nu^u)}\left( 1 - \frac{(\epsilon \delta (1+|\Im(z)|) \|\rho\|
)^2} {2048} \nu^u\left( \bigcup B_{\frac {s\delta}2}(y_i)\right) \right)
	\end{align*}
	By the diametric regularity of the measure $\nu^u$, there is a uniform constant $r < 1$ for which 
	\[
		\nu^u \left( B_{\frac {s\delta}2}(y_i) \right) \geq r \nu^u \left( B_{3\delta} (y_i) \right)
	\]
	for all $i$. Hence, we have
	\begin{align*}
		\nu^u \left( \bigcup B_{\frac {s\delta}2}(y_i) \right) \geq r \nu^u \left( \bigcup B_{3\delta}(y_i) \right) \geq r \nu^u \left( U_{lni} \right)
	\end{align*}
	completing the proof.
\end{proof}

It is important to recognize that many of the estimates so far do in fact depend on $\rho$, $\Im(z)$ or $C$ -- this will be problematic for the spectral bounds we want to obtain. To isolate some of these dependencies, we will restrict our attention to control functions $\Phi \in \mathcal K_{C (1+|\Im(z)|) \|\rho\|
}$, where we hope to be able to make a uniform choice of an appropriate $C$. In particular, we want to find a $C$ so that $\mathcal K_{C (1+|\Im(z)|) \|\rho\|}$ is invariant under $\mathcal L^n_{z,\rho}$, at least for $z$ with $\Re(z)$ sufficently close to $P(\varsigma)$.

\begin{proposition}
	There is a uniform choice of constant $C > 0$ so that, for all $\varphi \in C^1(U, V^\rho)$ and $\Phi \in \mathcal K_{C(1+|\Im(z)|\|\rho\|}$, if we have
	\begin{align*}
		\left\| \varphi(u) \right\|_{L^2(G)} &\leq \Phi(u)\\
		\sup_{w \in T^1_uU} \left\| (d\varphi)_u(w) \right\|_{L^2(G)} &\leq C(1+|\Im(z)|)\Phi(u)
	\end{align*}
	for all $u \in U$, then we can bound
	\[
		\sup_{w \in T^1_uU} \left\| \left(d\left( \mathcal L^n_{z,\rho}\varphi \right)\right)_u(w) \right\|_{L^2(G)} \leq C (1+|\Im(z)|)\|\rho\| \left(\mathcal L_{\Re(z),0}^n \Phi\right)(u)
	\]
	for all $u \in U$, and all $n > 0$.
	\label{uniformc}
\end{proposition}
\begin{proof}
	Fix $u \in U$. By definition, the transfer operator can be expressed as the sum
	\begin{equation}
		\left( \mathcal L^n_{z,\rho}\varphi \right)(u) = \sum_{i} e^{\alpha_{z}^{(n)}\left(v_i^{(n)}(u)\right)} \rho\left( \Hol^{(n)}\left(v_i^{(n)}(u)\right) \right) \varphi\left( v_i^{(n)}(u) \right)
		\label{tod2}
	\end{equation}
	over pasts $v_i^{(n)}$. 
	We need to control
	\begin{equation}
		\sup_{w \in T^1_uU} \left\| d\left( \mathcal L_{z,\rho}^n(\varphi) \right)_u(w)\right\|_{L^2(G)}
		\label{to1}
	\end{equation}
	which we will accomplish by differentiating \eqref{tod2} term-by-term. For a fixed $i$, the derivative 
	\begin{equation}
		d\left(e^{\alpha_{z}^{(n)}\left(v_i^{(n)}(u)\right)} \rho\left( \Hol^{(n)}\left(v_i^{(n)}(u)\right) \right)  \varphi\left( v_i^{(n)}(u) \right)\right)_u(w)
		\label{todd}
	\end{equation}
	can be bounded by the sum
	\begin{equation}
		\left\|e^{\alpha_{z}^{(n)}\left( v_i^{(n)}(u) \right)}\left(\rho\left( \Hol^{(n)}\left( v_i^{(n)}(u) \right) \right)\varphi\left( v_i^{(n)}(u) \right) \right) \left( d\left(  \alpha_{z}^{(n)} \circ v_i^{(n)}\right) \right)_u(w)\right\|_{L^2(G)}
		\label{term1}
	\end{equation}
	\begin{equation}
		+ \left\|e^{\alpha_{z}^{(n)}\left( v_i^{(n)}(u) \right)}\left(\left(d\left( \rho\left( \Hol^{(n)} \circ v_i^{(n)} \right)\right)\right)_u(w)\right) \varphi\left( v_i^{(n)}(u) \right)\right\|_{L^2(G)}
	\label{term2}
	\end{equation}
	\begin{equation}
		+ \left\|e^{\alpha_{z}^{(n)}\left( v_i^{(n)}(u) \right)}\rho\left( \Hol^{(n)}\left( v_i^{(n)}(u) \right) \right) \left(d \left(\varphi \circ v_i^{(n)} \right) \right)_u(w) \right\|_{L^2(G)}
		\label{term3}
	\end{equation}
	for all $w \in T^1_uU$. We will bound each of these terms individually, beginning with \eqref{term1}. Observe that
	\begin{align*}
	\left( d\left( \alpha_z^{(n)} \circ v_i^{(n)} \right) \right)_u(w) &= \sum_{j=1}^n \left| \left(d\left(\alpha_z \circ v_i^{(j)}\right)\right)_u(w) \right|\\
	&\leq \sum_{j=1}^n \frac{\|\alpha_z\|_{\mathrlap{C^1}}}{f\kappa^j}\\
	&\leq \frac{\|\alpha_z\|_{\mathrlap{C^1}}}{f(\kappa - 1)}
	\end{align*}
	and so \eqref{term1} is at most 
	\begin{align*}
		\frac{\|\alpha_z\|_{\mathrlap{C^1}}}{f(\kappa - 1)}e^{\alpha_{\Re(z)}^{(n)}\left( v_i^{(n)}(u) \right)} \left\|\varphi\left( v_i^{(n)}(u) \right) \right\|_{L^2(G)}
	\end{align*}
	for all $u \in U$ and $w \in T^1_uU$. Exactly the same calculations show us that
	\begin{align*}
		\left\|\left( d\left( \Hol^{(n)} \circ v_i^{(n)} \right) \right)_u(w) \right\|_{T^1G} &\leq \sum_{j=1}^n \left\| \left( d\left( \Hol \circ v_i^{(j)} \right) \right)_u(w) \right\|_{T^1G}\\
		&\leq \sum_{j=1}^n \frac{\left\| \Hol \right\|_{\mathrlap{C^1}}}{f\kappa^j}\\
		&\leq \frac{\left\|\Hol \right\|_{\mathrlap{C^1}}}{f(\kappa-1)}
	\end{align*}
	and hence that \eqref{term2} is at most
	\[
		\left( \|\rho\| \frac{\|\Hol\|_{\mathrlap{C^1}}}{f(\kappa - 1)} \right)e^{\alpha_{\Re(z)}\left( v_i^{(n)}(u) \right)} \left\|\varphi\left( v_i^{(n)}(u) \right)\right\|_{L^2(G)}
	\]
	by the definition of $\|\rho\|$ and the chain rule. Finally, we can bound \eqref{term3} by
	\[
		\frac{C(1+|\Im(z)|)\|\rho\|}{f\kappa^n} e^{\alpha_{\Re(z)}\left( v_i^{(n)}(u) \right)} \Phi\left( v_i^{(n)}(u) \right)
	\]
	by hypothesis. Combining all of these with the pointwise bound on the norm of $\varphi$, we see that so long as we have
	\begin{equation}
		\frac{\|\alpha_z\|_{\mathrlap{C^1}}}{f(\kappa - 1)} + \|\rho\| \frac{\|\Hol\|_{\mathrlap{C^1}}}{f(\kappa - 1)} + \frac{C(1+|\Im(z)|) \|\rho\|
}{f\kappa^n} < C (1+|\Im(z)|) \|\rho\|
		\label{uniformineq}
	\end{equation}
	we obtain the bounds desired. Note that $\|\Hol\|_{C^1}$ is a uniform constant and $\|\alpha_z\|_{C^1}$ is at most $\|\alpha_{\Re(z)}\|_{C^1} (1 + |\Im(z)|)$. Since $\|\alpha_{\Re(z)}\|_{C^1}$ is uniformly bounded when $\Re(z)$ is confined to the bounded interval $|\Re(z) - P(\varsigma)| < 1$, we can clearly choose a $C > 0$ so that \eqref{uniformineq} holds for all $n > 0$.
\end{proof}

\begin{theorem}
	Fix a nontrivial isotypic representation $\rho > 0$ and $z \in \bb C$ with $|\Re(z) - P(\varsigma)| < 1$. There are uniform constants $D > 0$ and $r_0 < 1$ so that
	\[
	  \left\|\mathcal L^n_{z,\rho} \varphi \right\|_{L^2(\nu^u)} \leq D r_0^n \|\varphi\|_{C^1}
	\]
	for all $\varphi \in C^1(U, V^\rho)$. Neither $D$ nor $r_0$ depends on $\rho$, $\varphi$ or $z$.
	\label{spectralgapmaintheorem}
\end{theorem}
\begin{proof}
	Fix $C > 0$, $N > 0$ as in Lemma \ref{uniformc} and $n_0 > N$ as in Lemma \ref{mainlemma}. We begin by setting
	\begin{align*}
		\varphi_0(u) &\coloneqq \varphi(u)\\
		\Phi_0(u) &\coloneqq \|\varphi\|_{C^1}
	\end{align*}
	for which we clearly have $\Phi_0 \in \mathcal K_{C (1 + |\Im(z)|) \|\rho\|}$ as well as the bounds
	\begin{align*}
		\left\| \varphi_0(u) \right\|_{L^2(G)} &\leq \Phi_0(u)\\
		\sup_{w \in T^1_uU} \left\| \left(d\varphi_0 \right)_u (w) \right\|_{L^2(G)} &\leq C (1+|\Im(z)|) \|\rho\| \Phi_0(u)
	\end{align*}
	assuming that $C > 0$ is large enough so that $C\|\rho\| > 1$ for all $\rho$. By Lemma \ref{mainlemma}, we can find a function $\beta_0 \in C^1(U, \bb R)$ for which we have
	\[
		\left\| \left( \mathcal L^{n_0}_{z,\rho} \varphi_0 \right)(u) \right\|_{L^2(G)} \leq \left( \mathcal L^{n_0}_{\Re(z),0} \left( \beta_0 \Phi_0 \right) \right)(u)
	\]
	and
	\[
		\left\| \mathcal L_{\Re(z),0}^{n_0} \left( \beta_0 \Phi_0 \right) \right\|_{L^2(\nu^u)} \leq r_0 \| \Phi_0\|_{L^2(\nu^u)}
	\]
	for a uniform choice of $n_0 > 0$. Moreover, $r_0$ and $\|\beta_0\|_{C^1}$ can be made uniform in $\rho, \varphi$ and $z$ by choosing $\delta$ so that $C(1+|\Im(z)|)\|\rho\| \cdot \delta$ is constant, as in Lemma \ref{alternative}. 

	We want to iterate this, for which it will be crucial that we can find a uniform $C_0$ so that we have $\beta_0 \Phi_0 \in \mathcal K_{C_0(1+|\Im(z)|)\|\rho\|}$ for all $\Phi_0 \in \mathcal K_{C_0(1+|\Im(z)|)\|\rho\|}$. For any given $\Phi_0 \in \mathcal K_{C(1+|\Im(z)|)\|\rho\|}$, we have
	\begin{align*}
	  \sup_{w \in T^1_uU} |(d(\Phi_0\beta_0)_u (w) | &\leq \beta_0(u) \left(\sup_{w \in T^1_uU} |(d\Phi_0)_u(w)| \right) + \Phi_0(u) \left( \sup_{w \in T^1_uU} |(d\beta_0)_u(w)| \right)\\
		  &\leq \sup_{w\in T^1_uU} \Phi_0(u) \left( C(1+ |\Im(z)|)\|\rho\| \beta_0(u) + (d\beta_0)_u(w) \right)
	\end{align*}
	for all $u \in U$. From the construction of $\beta_0$, and by our earlier remarks on our choice of $\delta$, it is clear that we can choose a uniform $C_0$ large enough so that
	\[
		\sup_{w\in T^1_uU} C_0(1+ |\Im(z)|)\|\rho\| \beta_0(u) + (d\beta_0)_u(w) \leq C_0 \beta_0(u)
	\]
	for any choice of $\beta_0$ as in Lemma \ref{mainlemma}. We then clearly have $\beta_0 \Phi_0 \in \mathcal K_{C_0 (1+|\Im(z)|) \|\rho\|}$, and hence by Lemma \ref{uniformc} we get
	\[
	  \sup_{w \in T^1_uU} \left\| (d(\mathcal L^{n_0}_{z,\rho}\varphi))_u(w) \right\|_{L^2(G)} \leq C_0 (1+|\Im(z)|) \|\rho\| \left( \mathcal L^n_{\Re(z),0} \Phi_0 \right)(u)
	\]
	for all $u \in U$. 
	
	Now, we can repeat what we have done so far and inductively choose $\beta_{i-1}$ as in Lemma \ref{mainlemma}, setting
	\begin{align*}
	  \varphi_i(u) &\coloneqq \left(\mathcal L_{z,\rho}^{n_0} \varphi_{i-1} \right)(u)\\
	  \Phi_i(u) &\coloneqq \left( \mathcal L_{\Re(z),0}^{n_0} (\beta_{i-1} \Phi_{i-1} )\right)(u)
	\end{align*}
	for $i \geq 1$. Note that we have just shown that $\Phi_1 \in \mathcal K_{C_0(1 + |\Im(z)|)\|\rho\|}$ and that we have
	\begin{align*}
		\left\| \varphi_1(u) \right\|_{L^2(G)} &\leq \Phi_1(u)\\
		\sup_{w \in T^1_uU} \left\| \left(d\varphi_1 \right)_u (w) \right\|_{L^2(G)} &\leq C_0 (1 + |\Im(z)|) \|\rho\| \Phi_1(u)
	\end{align*}
	for all $u \in U$. As before, we get
	\begin{align*}
		\left\| \varphi_i(u) \right\|_{L^2(G)} &\leq \Phi_i(u)\\
		\sup_{w \in T^1_uU} \left\| \left(d\varphi_i \right)_u (w) \right\|_{L^2(G)} &\leq C_0 (1 + |\Im(z)|) \|\rho\| \Phi_i(u)
	\end{align*}
	inductively for all $i \geq 1$, and moreover
	\[
	  \left\| \Phi_i \right\|_{L^2(\nu^u)} \leq r_0 \| \Phi_{i-1}\|_{L^2(\nu^u)}
	\]
	by construction. Chaining these inequalities together, we have
	\[
		\left\|\mathcal L^{i \cdot n_0}_{z,\rho} \varphi \right\|_{L^2(\nu^u)} \leq r_0^i \|\varphi \|_{C^1}
	\]
	which is almost what we need. To conclude, observe that we have
	\begin{equation}
		\left\|\left( \mathcal L^{i \cdot n_0 + k}_{z,\rho} \varphi \right)(u) \right\|_{L^2(\nu^u)} \leq r_0^i \|\varphi \|_{C^1} \| \mathcal L_{z,\rho}^k\|_{L^2(G)}
		\label{almost}
	\end{equation}
	where $\| \mathcal L_{z,\rho}^k \|_{L^2(G)}$ denotes the operator norm of $\mathcal L_{z,\rho}^k$. If $k < n_0$ and $|\Re(z) - P(\varsigma)| \, < 1$, then we can find a uniform bound $D > 0$ so that $\|\mathcal L_{z,\rho}^k \|_{L^2(G)} \leq D$, as desired.
\end{proof}

The differentiability of the potential $\varsigma$ was essential to much of what we have done so far in this section; to extend our results to the case when $\varsigma$ is only H\"older, however, is a relatively straightforward approximation argument, identical to the one given in \cite{dolgopyatmixing}. We sketch this below.

\begin{corollary}
	With notation as in Theorem \ref{spectralgapmaintheorem}, we have
	\[
		\|\mathcal L^n_{\Re(z),\rho} \varphi \|_{L^2(\nu^u)} \leq Dr_0^n\|\varphi\|_{C^1}
	\]
	when the potential $\varsigma$ is only H\"older continuous.
	\label{holderpotentialapprox}
\end{corollary}
\begin{proof}
	Suppose that $\varsigma$ is H\"older continuous. As in \cite{dolgopyatmixing}, we can find a sequence of smooth potentials $\varsigma^{(\Im(z))} \in C^1(U, \bb R)$ indexed by $\Im(z)$ for which we have
	\begin{equation}
		\sup_{u \in U} \left| \varsigma_0(u) - \varsigma^{(\Im(z))}(u) \right| \leq \|\varsigma_0\|_{C^\alpha} \left( \frac 1 {|\Im(z)|} \right)^{\frac \alpha 2}
		\label{importantineq2}
	\end{equation}
	and
	\begin{equation}
		\left\|\varsigma^{(\Im(z))}\right\|_{C^1} < C\sqrt{|\Im(z)|}
		\label{importantineq}
	\end{equation}
	for some uniform constant $C > 0$. We consider the alternatively-defined transfer operator
	\[
		\left(\hat {\mathcal L}_{z,\rho} \varphi \right)(u) \coloneqq \smash{\sum_{\sigma(u')=u}} e^{\hat \alpha_z(u')} (\rho(\Hol(u'))\cdot \varphi(u'))
	\]
	where we set
	\[
		\hat {\alpha}_z(u) \coloneqq \varsigma^{(\Im(z))}(u) - \Re(z) \cdot \tau(u,s) - \log\left(\varphi_{\varsigma^{(x)}}(u)\right) + \log\left(\varphi_{\varsigma^{(x)}}(\sigma(u))\right) - \log P(\varsigma^{(x)})
	\]
	for $u \in U$. By \eqref{importantineq2}, $\hat {\alpha}_z$ converges uniformly to $\alpha_{\Re(z)}$ as $\Im(z) \to \infty$, and hence $\hat {\mathcal L}_{z,\rho}\varphi$ must converge to $\mathcal L_{\Re(z),\rho}\varphi$ in the $L^2(\nu^u)$ norm. 

	Now, we simply observe that the spectral bound in Theorem \ref{spectralgapmaintheorem} holds for the operator $\hat {\mathcal L}_{z,\rho}$, since the main properties required of $\hat \alpha_z$ -- namely that $\|\hat \alpha_z\|_{C^1} \leq C(1+|\Im(z)|)$ for large $|\Im(z)|$ when $|\Re(z) - P(\varsigma)| \leq 1$ -- are guaranteed by \eqref{importantineq} and \cite[Lemma 1]{dolgopyatmixing}. Moreover, note that all the constants, and particularly those originating from Proposition \ref{uniformc}, can be chosen uniformly in $\Im(z)$. Since the inequality in Theorem \ref{spectralgapmaintheorem} holds with the same constants for $\hat {\mathcal L}_{z,\rho}$ for each $\Im(z)$, it must hold in the limit, as desired.
\end{proof}
\printbibliography
\end{document}